\documentclass{amsart}
\usepackage{hyperref,amssymb}
\usepackage[english]{babel}
\usepackage{fancyhdr}
\usepackage{colortbl}
\usepackage{enumerate}
\usepackage{graphicx}
\usepackage{float} 
\usepackage{mathrsfs}
\usepackage[left=2.5cm,right=2.5cm,top=2.75cm,bottom=2.75cm]{geometry}
\usepackage{tikz}
\usetikzlibrary{arrows,automata,shapes,calc,patterns}
\usepackage{cite}
\usepackage{xcolor}

\usepackage[colorinlistoftodos]{todonotes}
\definecolor{blue}{rgb}{0,0,1}

\newtheorem{theorem}{Theorem}[section]
\newtheorem{prop}[theorem]{Proposition}
\newtheorem{lemma}[theorem]{Lemma}
\newtheorem{cor}[theorem]{Corollary}

\theoremstyle{definition}
\newtheorem{definition}{Definition}[section]

\theoremstyle{remark}
\newtheorem{remark}{Remark}[section]

\numberwithin{equation}{section}

\def \N {\mathbb{N}}  % naturales
\def \R {\mathbb{R}}  % reales
\def\RN {\mathbb{R}^{N}} %R^{N}

\def \Wop {W_0^{1,p}(\Omega)}
\def \W {W_0^{1,p}}

\def \Linfty {L^{\infty}(\Omega)}

\def \pLaplac {-\Delta_p} 
\def \gradu {\nabla u}
\def \gradv {\nabla v}
\def \muv {\frac{\mu}{p-1}v}

\def \past {p^{\ast}}

\renewcommand{\div}{\operatorname{div}}

\renewcommand{\epsilon} {\varepsilon}

\newcommand{\dist}{\operatorname{dist}}

\hypersetup{
colorlinks=true, breaklinks=true, bookmarks=true,bookmarksnumbered,
urlcolor=blue, linkcolor=blue, citecolor=red, % Link colors
pdftitle={}, % PDF title
pdfauthor={\textcopyright}, % PDF Author
pdfsubject={}, % PDF Subject
pdfkeywords={}, % PDF Keywords
pdfcreator={pdfLaTeX}, % PDF Creator
pdfproducer={LaTeX with hyperref and ClassicThesis} % PDF producer
}

\begin{document}

\title[A p-Laplacian problem with critical growth in the Gradient]{Existence and Multiplicity for elliptic p-Laplacian 
problems with critical growth in the gradient}

%    Remove any unused author tags.

%    author one information
\author[Colette De Coster and Antonio J. Fern\'andez]{}
\address{}
\curraddr{}
\email{}
\thanks{}

%    author two information
%\author[A. J. Fern\'andez]{Antonio J. Fern\'andez}
%\address{Univ. Valenciennes, EA 4015 - LAMAV - FR CNRS 2956, F-59313 Valenciennes, France}
%\address{Laboratoire de Math\'ematiques (UMR 6623), Universit\'e de Bourgogne Franche-Comt\'e, 
%16 route de Gray, 25030 Besan\c con Cedex, France}
%\curraddr{}
%\email{}
%\thanks{}

\email{colette.decoster@univ-valenciennes.fr}
\email{antonio\_jesus.fernandez\_sanchez@univ-fcomte.fr}

\subjclass[2010]{35J20, 35J25, 35J92}

\keywords{Quasilinear elliptic equations, critical growth in the gradient, p-Laplacian, lower and upper solutions, variational methods}

\date{}

\dedicatory{}
\maketitle

\centerline{\scshape Colette De Coster}
\smallskip
{\footnotesize
 % please put the address of the second  and third author
 \centerline{Univ. Valenciennes, EA 4015 - LAMAV - FR CNRS 2956, F-59313 Valenciennes, France}
}

\bigskip

\centerline{\scshape Antonio J. Fern\'andez}
\smallskip
{\footnotesize
% please put the address of the first author
\centerline{Univ. Valenciennes, EA 4015 - LAMAV - FR CNRS 2956, F-59313 Valenciennes, France}
\smallbreak
\centerline{Laboratoire de Math\'ematiques (UMR 6623), Universit\'e de Bourgogne Franche-Comt\'e,}
\centerline{16 route de Gray, 25030 Besan\c con Cedex, France}
} % Do not forget to end the {\footnotesize by the sign }

\begin{center}\rule{1\textwidth}{0.1mm} \end{center}
\begin{abstract}
We consider the boundary value problem 
\[ \tag{$P_{\lambda}$} 
\pLaplac u =  \lambda c(x) |u|^{p-2}u + \mu(x) |\gradu|^p + h(x)\, , \quad u \in \Wop \cap L^{\infty}(\Omega)\, ,\]
where $\Omega \subset \RN$, $N \geq 2$, is a bounded domain with smooth boundary. We assume 
$c,\, h \in L^q(\Omega)$ for some $q > \max\{N/p,1\}$ with $c \gneqq 0$ and $\mu \in L^{\infty}(\Omega)$. 
We prove  existence and 
uniqueness results in the coercive case $ \lambda \leq 0$ and existence and multiplicity results in the non-coercive case $ \lambda >0$.
Also, considering stronger assumptions on the coefficients, we clarify the structure of the set of solutions in the 
non-coercive case. 
\end{abstract}
\begin{center} \rule{1 \textwidth}{0.1mm} \end{center}

\section{Introduction and main results} \label{I}

%\textcolor{red}{Texte � mettre en couleur avec} 

Let $\Delta_p u= \div( |\nabla u|^{p-2} \nabla u)$ denote the $p$-Laplacian operator. We consider,
 for any $1<p<\infty$, the boundary value problem
 
 \begin{equation} \label{Plambda} \tag{$P_{\lambda}$}
\pLaplac u = \lambda c(x)|u|^{p-2}u + \mu(x)|\gradu|^p + h(x)\,, \quad u \in \Wop \cap \Linfty\,,
\end{equation}
under the assumptions

\[\label{B10_i} \tag{$A_{0}$} 
\left \{
\begin{aligned} 
& \Omega \subset \R^N,\, N \geq 2\,, \textup{ is a bounded domain with } 
\partial \Omega \textup{ of class } \mathcal{C}^{0,1},\\
& c \textup{ and } h \textup{ belong to } L^q(\Omega) \textup{ for some } q > \max\{N/p,1\}, \\
& c \gneqq 0  \textup{ and } \mu \in \Linfty\,.
\end{aligned}
\right.
\]

The study of quasilinear elliptic equations with a gradient dependence up to the critical growth 
$|\gradu|^p$ was initiated by L. Boccardo, F. Murat and J.P. Puel in the 80's and it has been an active field 
of research until now. Under the condition $\lambda c(x) \leq -\alpha_0 < 0$  for some 
$\alpha_0 > 0$, which is now referred to as the \textit{coercive case}, the existence of solution is a particular case 
of the results of \cite{B_M_P_1983, B_M_P_1992, DA_G_P_2002}. The \textit{weakly coercive case} 
($\lambda=0$) was studied in \cite{F_M_2000} where, for $\| \mu h\|_{N/p}$ small enough,  the 
existence of a unique solution is obtained, see also \cite{AH_BV_2010}. The \textit{limit coercive case}, where one just require that $\lambda c(x) \leq 0$  
% is more complex.
%Restricted to linear operator in divergence form, there had been several contributions (see for instance 
%\cite{A_DA_P_2006, F_M_2000}) for the case $\lambda = 0$. 
%Nevertheless, the general case where
and hence $ c$  may vanish only on some parts of $\Omega$, is more complex and  was left open until 
\cite{A_DC_J_T_2015}. In that paper, for the case $p = 2$, it was observed that, under the assumption \eqref{B10_i}, the existence of solutions to \eqref{Plambda} is not guaranteed. 
% For $\lambda = 0$, it was already observed in \cite{F_M_2000}. In fact, in \cite{A_DC_J_T_2015} 
%the authors found 
Sufficient conditions in order to ensure the existence of solution were given. 
\medbreak

The case $\lambda c(x) \gneqq 0$ also remained unexplored until very recently. 
First, in \cite{J_S_2013} the authors studied problem \eqref{Plambda} with $p = 2$. 
Assuming $\lambda > 0$  and $\mu h$ small enough, in an appropriate sense, they proved the existence of at least 
two solutions. 
This result has now be complemented in several ways. In \cite{J_RQ_2016} the existence of two solutions is obtained, allowing 
the function $c$ to change sign with  $c^{+} \not\equiv 0$ but assuming $h \gneqq 0$. In both \cite{J_RQ_2016, J_S_2013} $\mu >0$ is assumed constant. In \cite{A_DC_J_T_2015}  the restriction $\mu$ 
constant was removed but assuming that $h \gneqq 0$. 
Finally, in \cite{DC_J_2017}, under stronger regularity on the 
coefficients, cases where $\mu$ is non constant and $h$ is non-positive or has no sign were treated. Actually in \cite{DC_J_2017}, under different sets of assumptions, 
the authors clarify  the structure of the set of solutions to \eqref{Plambda} in the non-coercive case. Now, concerning \eqref{Plambda} with $p\not=2$, the only results in the case 
$\lambda c \gneqq 0$ are, up to our knowledge, presented in \cite{I_L_U_2010, AH_BV_2010}. In \cite{I_L_U_2010}  the case $c$ constant and $h\equiv0$ is covered and in \cite{AH_BV_2010},  the model equation is $\pLaplac u = |\gradu|^p + \lambda f(x)(1+u)^b\,,\, b \geq p-1$ and $f \gneqq 0$. 

\medbreak

To state our first main result let us define 
\begin{equation*} 
m_{p,\lambda}^{+}:= \left\{
\begin{aligned}
& \inf_{u \in W_{\lambda }} \int_{\Omega} \Big(|\gradu|^p 
- \big(\frac{\|\mu^{+}\|_{\infty}}{p-1}\big)^{p-1}h(x)|u|^p\,\Big)\, dx, \quad 
& \textup{ if } W_{\lambda} \neq \emptyset\,,
\\
& +\infty, & \textup{ if  } W_{\lambda} = \emptyset\,,
\end{aligned}
\right.
\end{equation*}
and
\begin{equation*} 
m_{p,\lambda}^{-}:= \left\{
\begin{aligned}
& \inf_{u \in W_{\lambda }} 
\int_{\Omega} \Big(|\gradu|^p - \big(\frac{\|\mu^{-}\|_{\infty}}{p-1}\big)^{p-1}h(x)|u|^p\,\Big) \,dx, 
\quad & \textup{ if } W_{\lambda} \neq \emptyset\,,
\\
& +\infty, & \textup{ if  } W_{\lambda} = \emptyset\,.
\end{aligned}
\right.
\end{equation*}
where 
\[ W_{\lambda }:= \{ w \in \Wop: \lambda c(x)w(x) = 0 \textup{ a.e. } x \in \Omega\,,\ \|w\| = 1\}\,. \]
Note that $W_{0} = \Wop$ and $W_{\lambda}$ is independent of $\lambda$ when $\lambda \neq 0$. Using these notations, we state the following result which generalizes the results obtained in 
\cite[Section 3]{A_DC_J_T_2015}. In fact, if $h$ is either non-negative or non-positive our hypothesis corresponds to
the ones introduced in \cite{A_DC_J_T_2015} for $p = 2$. However, if $h$ does not have a sign, 
our hypothesis are weaker even for $p=2$. 

\begin{theorem}\label{coercive-louis}
Assume that \eqref{B10_i} holds and that $ \lambda \leq 0$. Then if $m_{p,\lambda }^{+} > 0$ and $m_{p, \lambda }^{-} > 0$, the 
 problem \eqref{Plambda} has at least one solution.
\end{theorem}

In the rest of the paper we assume that $\mu$ is constant. Namely, we replace \eqref{B10_i} by
\[\label{A1} \tag{$A_{1}$} 
\left \{
\begin{aligned} 
& \Omega \subset \R^N,\, N \geq 2\,, \textup{ is a bounded domain with } 
\partial \Omega \textup{ of class } \mathcal{C}^{0,1},\\
& c \textup{ and } h \textup{ belong to } L^q(\Omega) \textup{ for some } q > \max\{N/p,1\}, \\
& c \gneqq 0  \textup{ and } \mu >0\,.%\\
\end{aligned}
\right.
\]
Observe that there is  no loss 
of generality in assuming $\mu > 0$  since, if $u$ is a solution to 
\eqref{Plambda} with $\mu < 0$, then $w = -u$ satisfies
\[ \pLaplac w = \lambda c(x)|w|^{p-2}w - \mu |\nabla w|^p - h(x)\,.\]
%Hence,  we can cover also the case $\mu < 0$. 
%So, throughout all the paper, \textbf{we are going to consider the assumptions \eqref{A1} with $\mu > 0$}.

\medbreak

%\textcolor{red}{
In \cite{A_DC_J_T_2014}, for $p=2$  but assuming only \eqref{B10_i}, 
the uniqueness of  solution when $\lambda \leq 0$ was obtained as a direct consequence 
of a comparison principle, see \cite[Corollary 3.1]{A_DC_J_T_2014}. As we show in Remark 
\ref{pascomparaison}, such kind of principle does not hold in general when $p\neq 2$. 
Actually the issue of uniqueness for equations of the form of $(P_{\lambda})$ appears widely open. 
If partial results, assuming for example $1 <p \leq 2$ or $\lambda c(x) \leq - \alpha_0 <0$,  
seem reachable adapting existing techniques, see in particular \cite{L_P_R_2017, P_2008, P_S_2007}, 
a result covering the full generality of \eqref{Plambda} seems, so far, out of reach. 
Theorem \ref{uniqueness} below, whose proof makes use of some ideas from \cite{A_P_2003}, 
crucially relies on the assumption that $\mu$ is constant. It permits however to treat 
the limit case $(P_0)$ which plays an important role in our paper.%}

\begin{theorem} \label{uniqueness}
Assume that \eqref{A1} holds and suppose $\lambda \leq 0$. Then \eqref{Plambda} has at most one solution.
\end{theorem}

\medbreak
 
Let us now introduce
\begin{equation}\label{mp} 
m_p:= \inf \Big\{ \int_{\Omega} \Big(|\nabla w|^p - \big(\frac{\mu}{p-1} \big)^{p-1}h(x) |w|^p \Big)\,dx: 
w \in \Wop,\, \|w\| = 1 \Big\}\,.
\end{equation}
We can state the following result.
\begin{theorem} \label{th3}
Assume that \eqref{A1} holds. Then $(P_0)$ has a solution if, and only if, $m_p > 0$. 
\end{theorem}

Theorem \ref{th3} provides, so to say, a characterization in term of a first eigenvalue of the existence of solution to $(P_0)$. This result again  improves, for $\mu$ constant, \cite{A_DC_J_T_2015} and it allows  to observe that, in case $h\lneqq 0$, $(P_0)$ has always a solution while the case $h\gneqq0$ is the ``worse'' case for the existence of a solution. In case $h$ changes sign, the negative part of $h$ ``helps'' in order to have a solution to $(P_0)$.  We give in Appendix \ref{AR}, sufficient conditions on
 $h^+$ in order to ensure $m_p>0$.
\medbreak
%As a first result, without assuming sign restrictions on $h$, we can obtain the following result. 
%under the assumption that %This proves the importance of 
%the problem $(P_0)$ has a solution.

\begin{remark}\label{reverse}
Observe that the sufficient part of Theorem \ref{th3} is direct. Indeed, if $m_p>0$ then $m_{p,0}^{+} >0$ and $ m_{p,0}^{-} >0 $ and  Theorem \ref{coercive-louis} implies that $(P_0)$ has a solution.
\end{remark}

\begin{remark}
We see, combining Theorems  \ref{coercive-louis} and  \ref{th3}, that if $(P_0)$ has a solution then \eqref{Plambda} has a solution for any $\lambda \leq 0$. Moreover this solution is unique by Theorem \ref{uniqueness}.
\end{remark}

%Our main objective  is to extend the results previously obtained for $p = 2$ to the case of the $p$-Laplacian 
%operator. A similar case to $h = \lambda c(x) \gneqq 0$ was considered in \cite{AH_BV_2010}.
%Under suitable assumptions they proved the existence of at least two solutions
% {\color{red}Antonio, je ne comprends pas ces deux phrases et ce que tu dis sur \cite{AH_BV_2010}? Enlever???}. 
% \medbreak

Now, we turn to the study the non-coercive case, namely when $\lambda > 0$. First, using mainly variational techniques we prove the following result.

\begin{theorem} \label{th1}
Assume that \eqref{A1} holds and suppose that $(P_0)$ has a solution. Then there exists $\Lambda > 0$ such that, for any $0 < \lambda < \Lambda$, \eqref{Plambda} 
has at least two solutions.
\end{theorem}

As we shall see in Corollary \ref{corlambdaimplique 0.2}, the existence of a solution to $(P_0)$ is, in some sense, necessary for the existence of a
solution when $\lambda>0$.
\medbreak
Next, considering stronger regularity assumptions, we derive informations on the structure of the set of  solutions in the 
non-coercive case. These informations complement Theorem \ref{th1}.  
%To this aim, let us introduce some notation.
We denote by $\gamma_1 >0$ the first eigenvalue of the problem
\begin{equation} 
\label{BVPeigenvalue}
\pLaplac u = \gamma c(x) |u|^{p-2} u\,, \quad u \in \Wop\,,
\end{equation}
and, under the assumptions
\[ \label{A2} \tag{$A_{2}$} 
\left \{
\begin{aligned}
& \Omega \subset \R^N,\, N \geq 2\,, \textup{ is a bounded domain with } 
\partial \Omega \textup{ of class } \mathcal{C}^2,\\
& c \textup{ and } h \textup{ belong to } L^{\infty}(\Omega)\,, \\
& c \gneqq 0  \textup{ and } \mu > 0\,,
\end{aligned}
\right.
\]
we state the following theorem.

\begin{theorem} \label{th2}
Assume that \eqref{A2} holds and suppose that $(P_0)$ has a solution. Then:
\begin{itemize}
\item If $h \lneqq 0$, for every $\lambda > 0$, \eqref{Plambda} has at least two solutions $u_1$, $u_2$ 
with $u_1\ll0$.
\item If $h \gneqq 0$, then $u_0\gg0$ and there exists $\overline{\lambda} \in(0, \gamma_1)$ such that:
\begin{itemize}
\item for every 
$0 < \lambda < \overline{\lambda}$, \eqref{Plambda} has at least two solutions   
satisfying $u_i\geq u_0$;
\item for  $\lambda = \overline{\lambda}$, \eqref{Plambda} has at least one solution   
satisfying $u\geq u_0$;
\item for any $\lambda > \overline{\lambda}$, \eqref{Plambda} has no non-negative solution.
\end{itemize}
\end{itemize}
\end{theorem}

\begin{center}
\begin{figure}[h]
\begin{tikzpicture}[line width = 0.35mm, scale = 1.6, >=stealth]
\draw[->](-0.5,0)--(2,0);
\draw[->](0,-0.5)--(0,2); %esto son los ejes
%\draw(1,0) node[below]{$\overline{\lambda}$};
\draw(2,0) node[right]{$\lambda$}; % los dos lambdas que tenemos
%\draw[dashed](1,0)--(1,2); % el eje vertical en lambda barra
%\draw[line width = 0.75mm] (0,0.125) .. controls (0.5, 0.30) and (2,1) .. (1,0.60);
\draw[smooth,line width = 0.45mm] plot coordinates{(1.97,0.31)(1.75,0.33)(1.5,0.36)(1.25,0.4)(1,0.45)(0.9,0.48)(0.7,0.57)(0.5,0.69)(0.3,0.92)(0.15,1.5)(0.1,1.98)};
%\draw[domain = 1:2] plot(\x, cos(\x));
%\draw[domain=0:1] plot(\x,cos(\x^3-2*\x^2+2*\x));
%\draw[color=blue] plot[id=sin] function{sin(x)} 
  %  node[right] {$f(x) = \sin x$};
\draw[smooth,line width = 0.45mm] plot coordinates{(-0.5,-0.1)(0,-0.125) (0.25,-0.15)(0.5,-0.175)(0.75,-0.2)(1,-0.23)(1.25,-0.27)(1.5,-0.32)(1.75,-0.38)(1.97,-0.45)};
%(0.6,0.63)(0.8,0.52)(0.4,0.76) 
\end{tikzpicture}
\caption{Illustration of Theorem \ref{th2} with $h \lneqq 0$}
\label{fig1}
\end{figure}

\begin{figure}[h]
\begin{tikzpicture}[line width = 0.35mm, scale = 1.6, >=stealth]
\draw[->](-0.5,0)--(2,0);
\draw[->](0,-0.2)--(0,2); %esto son los ejes
\draw(1,0) node[below]{$\overline{\lambda}$};
\draw(2,0) node[right]{$\lambda$}; % los dos lambdas que tenemos
\draw[dashed](1,0)--(1,2); % el eje vertical en lambda barra
%\draw[line width = 0.75mm] (0,0.125) .. controls (0.5, 0.30) and (2,1) .. (1,0.60);
\draw[smooth,line width = 0.45mm] plot coordinates{(-0.5,0.1)(0,0.125) (0.25,0.15)(0.5,0.175)(0.75,0.2)(1,0.26)};
\draw[smooth,line width = 0.45mm] plot coordinates{(1,0.465)(0.9,0.48)(0.7,0.57)(0.5,0.69)(0.3,0.92)(0.15,1.5)(0.1,1.98)};
\end{tikzpicture}
\caption{Illustration of Theorem \ref{th2} with $h \gneqq 0$}
\label{fig2}
\end{figure}

\end{center}
\newpage
\begin{remark}$ $
\begin{itemize}
\item[a)] As observed above, in the case $h \lneqq 0$, the assumption that $(P_0)$ has a solution is automatically 
satisfied. 
\item[b)] In the case $\mu<0$,  we have the opposite result i.e., two solutions for every 
$\lambda>0$ in case $h \gneqq 0$ and, in  case $h \lneqq 0$, the existence of $\overline \lambda>0$ such that 
 \eqref{Plambda} has at least two negative solutions,  at least one negative solution or no non-positive solution according to 
 $0<\lambda<\overline\lambda$,  $\lambda=\overline\lambda$  or $\lambda>\overline\lambda$.
\end{itemize}
\end{remark}

In case $h \gneqq 0$, we know that for $\lambda > \overline{\lambda}$, \eqref{Plambda} 
has no non-negative  solution but this does not exclude the possibility of having negative or sign changing solutions. 
Actually, we are able to prove the following result changing a little the point of view. 
We consider the boundary value problem
\begin{equation}
\label{eq cor} \tag{$P_{\lambda,k}$} 
 \pLaplac u = \lambda c(x)|u|^{p-2}u+ \mu |\nabla u|^p + kh(x)\,, 
\quad u \in W^{1,p}_0(\Omega) \cap L^{\infty}(\Omega)\,,
\end{equation}
with a dependence in the size of $h$ and we obtain the following result. 

\begin{theorem} \label{th4}
Assume that \eqref{A2} holds and that $h \gneqq 0$. Let 
\[
k_0=\sup\Big\{ k\in [0, +\infty) : \forall\ w\in \Wop\,,\  \int_{\Omega} \Big(|\nabla w|^p - 
\big(\frac{\mu}{p-1} \big)^{p-1} k\, h(x) |w|^p \Big)\,dx>0\Big\}\,.
\] 
Then:
\begin{itemize}
\item
For all $\lambda\in(0,\gamma_1)$, there exists $\overline k= \overline k(\lambda)\in(0,k_0)$ such that,
 for all $k\in(0,\overline k)$, the problem \eqref{eq cor}
has at least two  solutions $u_1$, $u_2$ with $u_i\gg0$ and for all $k> \overline k$, 
the problem \eqref{eq cor} has no solution. 
Moreover, the function  $\overline k(\lambda)$ is non-increasing.
%\\
%$\bullet$ for $\lambda=\lambda_1$, the problem \eqref{eq cor} has exactly one positive solution;
%\\
%$\bullet$ for $\lambda>\lambda_1$, the problem \eqref{eq cor} has no non-negative solution;
%\item[(ii)]
%for $\lambda=\gamma_1$ the problem \eqref{eq cor} has no solution;
\item
For $\lambda=\gamma_1$, the problem \eqref{eq cor} has a solution if and only if $k=0$. 
In that case, the solution is unique and it is equal to $0$.
\item
For all $\lambda>\gamma_1$, there exist $0<\tilde k_1\leq \tilde k_2<+\infty$ such that, 
for all $k\in(0,\tilde k_1)$, the problem \eqref{eq cor}
 has at least two solutions with 
$u_{\lambda,1}\ll 0$ and $\min u_{\lambda,2}<0$, for all $k>   \tilde k_2$, 
the problem \eqref{eq cor} has no solution and, in case $\tilde k_1< \tilde k_2$, for all $k\in(\tilde k_1, \tilde k_2)$, the problem \eqref{eq cor}
 has at least one solution $u$ with 
$u\not\ll 0$ and $\min u<0$. Moreover, the function $\widetilde{k}_1(\lambda)$ is non-decreasing.
%\\
%$\bullet$ for $\lambda=\lambda_2$,  the problem \eqref{eq cor} has a unique non-positive solution; 
%\\
%$\bullet$ for $\lambda<\lambda_2$,  the problem \eqref{eq cor} has no non-positive solution. 
\end{itemize}

% there exists $\overline{\lambda} > 0$ such that, for every 
%$0 < \lambda < \overline{\lambda}$, \eqref{Plambda} has at least two positive solutions and,
%for any $\lambda > \overline{\lambda}$, \eqref{Plambda} has no positive solution.
%\end{itemize}
\end{theorem}
\begin{center}
\begin{figure}[h]
\begin{tikzpicture}[line width = 0.4mm, scale = 2, >=stealth]
\fill[color=gray!15] (0,0) -- (1,0) -- (0.75,0.18) -- (0.5,0.38) -- (0.25,0.55)  -- (0,0.65) -- cycle   ;
\fill[color=gray!15] (2.97,0)--(1,0)--(1.25,0.15)--(1.5,0.25)--(1.75,0.32)--(2,0.39)--(2.25,0.42)--(2.5,0.45)--(2.75,0.47)--(2.97,0.485) -- cycle   ;
\fill[color=gray!55] (1,0)--(1.25,0.28)--(1.5,0.45)--(1.75,0.55)--(2,0.62)--(2.25,0.66)--(2.5,0.69)--(2.75,0.71)--(2.97,0.725)--(2.97,0.485)--(2.75,0.47)--(2.5,0.45)--(2.25,0.42)--(2,0.39)--(1.75,0.32)--(1.5,0.25)--(1.25,0.15)--(1,0)--cycle;
\draw(0.33,0.22) node{2};
\draw(2.15,0.18) node{2};
\draw(2.15,0.510) node{1};
\draw(1,0.35) node{0};
\draw(0.5,0.6) node{$\overline{k}(\lambda)$};
\draw(1,0) node[below]{$\gamma_1$};
%\draw[-](1,-0.03)--(1,0.03);
\draw(3,0) node[right]{$\lambda$}; 
\draw[-](-0.03,0.66)--(0.03,0.66);
\draw(0,0.65) node[left]{$k_0$};
\draw(0,1.25) node[above]{$k$};
\draw[smooth, line width = 0.6mm] plot coordinates{(0,0.65)(0.25,0.54)(0.5,0.38)(0.75,0.18)(1,0)};
\draw[smooth, line width = 0.6mm] plot coordinates{(1,0)(1.25,0.15)(1.5,0.25)(1.75,0.33)(2,0.39)(2.25,0.42)(2.5,0.45)(2.75,0.47)(2.97,0.485)}node[right]{$\widetilde{k}_1(\lambda)$};
\draw[smooth, line width = 0.6mm] plot coordinates{(1,0)(1.25,0.28)(1.5,0.45)(1.75,0.55)(2,0.62)(2.25,0.66)(2.5,0.69)(2.75,0.71)(2.97,0.725)}node[right]{$\widetilde{k}_2(\lambda)$};
\draw[->](-0.2,0)--(3,0);
\draw[->](0,-0.2)--(0,1.25); %esto son los ejes
\end{tikzpicture}
\caption{Existence regions of Theorem \ref{th4}}
\label{fig3}
\end{figure}
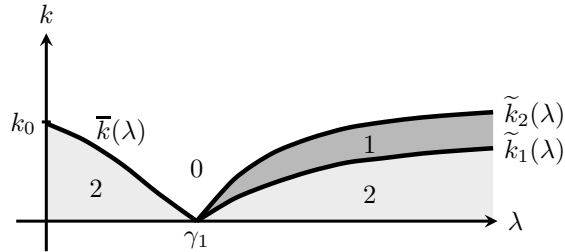
\end{center}
%\begin{remark}
%{\color{red}a revoir}
%If $m_p>0$ and $h \gneqq 0$, consider the first eigenvalue $\mu_1$ of
%$$
%\pLaplac u = \mu  \big(\frac{\mu}{p-1} \big)^{p-1}h(x) |u|^{p-2} u, \qquad u \in \Wop\,.
%$$
%Then, for all $\lambda<\gamma_1$, we have $\overline k\leq \mu_1$.
%\end{remark}

Let us now say some words about our proofs. First note that when  $\mu$ is assumed constant it is possible to perform a Hopf-Cole change of variable. Introducing 
$$
v = \frac{p-1}{\mu} \Big( e^{\frac{\mu}{p-1}u} -1 \Big)\,,
$$
we can check that $u$ is a solution of \eqref{Plambda} if, and only if, $v > -\frac{p-1}{\mu}$ is a solution of
\begin{equation}
\label{pbm2}
\pLaplac v = \lambda c(x) g(v) + \Big( 1 + \frac{\mu}{p-1}v\Big)^{p-1} h(x)\,, 
\quad v \in \Wop\,,
\end{equation}
where $g$ is an arbitrary function satisfying
$$
g(s) = 
%\begin{aligned}
%& 
\Big|\frac{p-1}{\mu}\big( 1+\frac{\mu}{p-1}s \big) \ln \big(1+ \frac{\mu}{p-1}s \big)\Big|^{p-2}
\frac{p-1}{\mu}\big( 1+\frac{\mu}{p-1}s \big) \ln \big(1+ \frac{\mu}{p-1}s \big), 
%\\
%& \hspace{110mm}
\quad
\mbox{if }s > -\frac{p-1}{\mu}\,.
%\end{aligned}
$$
Working with problem \eqref{pbm2} presents the advantage that one may assume, with a suitable choice of $g$ when $s \leq -\frac{p-1}{\mu}$, that it has a variational structure. Nevertheless from this point we face several difficulties.
\medbreak
First, we need a control from below on the solutions to \eqref{pbm2}, 
i.e. having found a solution to \eqref{pbm2}  one needs to check that it satisfies  $v > -\frac{p-1}{\mu}$, in order  to perform the opposite change of
variable and  obtain a solution to \eqref{Plambda}. 
To that end, 
in Section \ref{TLSATAP-L}, we  prove 
the existence of a lower solution $\underline{u}_{\lambda}$ to \eqref{Plambda} such that every  
upper solution $\beta$ of \eqref{Plambda} satisfies
$\beta\geq \underline{u}_{\lambda}$. This allows us to transform the problem
\eqref{pbm2} in a new one, which has the advantage of being completely equivalent to \eqref{Plambda}. Note that the existence of the lower solution ultimately relies on the existence of an a priori lower bound. See Lemma \ref{lowerBound} for a more general result.
%the property that, under the additional assumption that $0 < \mu_1 \leq \mu(x) $ for some $\mu_1 >0$ in \eqref{B10_i} but without having to assume that the function $c$ has a sign, the solutions to \eqref{Plambda} are uniformly bounded from below,  see Lemma \ref{lowerBound}. 
\medbreak
We denote by $I_{\lambda}$ the functional associated to the new problem, see \eqref{defI} for a precise definition. The ``geometry" of $I_{\lambda}$ crucially depends on the sign of $\lambda$. When $\lambda \leq 0$ is it essentially coercive and one may search for a critical point as a global minimum. When $\lambda >0$ the functional $I_{\lambda}$ becomes unbounded from below and presents something like a concave-convex geometry. Then, in trying to obtain a critical point, the fact that $g$ is only slightly 
superlinear at infinity is a difficulty. It implies that $I_{\lambda}$ does not satisfies an Ambrosetti-Rabinowitz-type condition and proving that Palais-Smale or Cerami sequences are bounded may be challenging. In the case of the Laplacian, when $p=2$, dealing with this issue is now relatively standard but for elliptic problems with a $p$-Laplacian things are more complex and we refer to   \cite{ DF_G_RQ_U_2017, I_L_U_2010, I_S_U_2014, L_Y_2010} in that direction. Note however that in these last works, it is always assumed a kind of homogeneity condition which is not available here. Consequently, some new ideas are required, see Section \ref{OTCCATMPG}. 
\medbreak
Having at hand the Cerami condition for $I_{\lambda}$ with $\lambda > 0$, in order to prove Theorems \ref{th1}, \ref{th2} and \ref{th4}, we shall look for critical points which are either local-minimum or of mountain-pass type. In Theorem \ref{th1} the geometry of $I_{\lambda}$ is ``simple" and permits to use only variational arguments. In Theorems \ref{th2} and \ref{th4} however it is not so clear, looking directly to $I_{\lambda}$, where to search for critical points. We shall then make uses of lower and upper solutions arguments. In both theorems a first solution is obtained through the existence of well-ordered lower and upper solutions. This solution is further proved to be a local minimum of $I_{\lambda}$ and it is then possible to obtain a second solution by a mountain pass argument. Our approach here follows the strategy presented in \cite{Chang_1, Chang_2, DF_S_1983}. See also \cite{A_B_C_1994}.
\medbreak
Finally, concerning Theorem \ref{coercive-louis}, where $\mu$ is not assumed to be constant, we obtain our solution through the existence of lower and an upper solution which correspond to solutions to $(P_{\lambda})$ where $\mu = -\|\mu^-\|_{\infty}$ and $\mu =  \|\mu^+\|_{\infty}$ respectively, see Section \ref{Coercive-Case}.
%\medbreak

%\textcolor{red}{Je pense qu'ici il faut donner quelques informations sur les preuves Th\'eor\`emes \ref{th2} et \ref{th4}, et le lien avec le reste de l'article.  Cet article est int\'er\'essant mais il est long et surtout tr\`es/trop technique. Si on veut avoir des lecteurs il faut absolument faire une introduction qui donne envie de lire l'article et pour cela il est n\'ecessaire, selon moi, de donner une id\'ee des preuves.} 

%\medbreak

%\textcolor{red}{Commentaire g\'en\'eral 1: trop souvent dans les parties que j'ai regard\'e (celles variationnelles) on suppose, en \'ecrivant par exemple "as we have done it in Section .." que le lecteur va lire l'article du d\'ebut \`a la fin. Cela n'est pas une bonne chose, il faut aider plus le lecteur. Cela passe aussi par plus num\'eroter les \'equations.}

%\medbreak

%\textcolor{red}{Commentaire g\'en\'eral 2 : Je n'ai pas lu la plupart des preuves de l'article ! Cependant avec la relecture de Colette je suis convaincu que tout est essentiellement ok.}

\bigbreak

The paper is organized as follows. In Section \ref{P}, we recall preliminary general results that are used 
in the rest of the paper. In Section \ref{CPAUR}, we give a comparison principle and
 prove the uniqueness result for $\lambda\leq 0$. Section \ref{TLSATAP-L} is devoted to the existence of the lower solution. In Section \ref{TLSATAP}, 
 we construct the modified problem that we use to obtain the existence results. The coercive and limit-coercive cases, corresponding to $\lambda \leq 0$ are studied in Section \ref{Coercive-Case} where we prove Theorem \ref{coercive-louis}. Theorem \ref{th3} which gives a necessary and sufficient condition to the existence of a solution to $(P_0)$ is established in Section \ref{Weakly-coercive-Case}. In Section \ref{OTCCATMPG} we show that $I_{\lambda}$ has, for $\lambda >0$ small, a mountain pass geometry and that the Cerami compactness condition holds. This permits to give the proof of Theorem \ref{th1}. Section \ref{SER2} contains the proofs of Theorems \ref{th2} and \ref{th4}. Finally in an Appendix we give conditions on $h^+$ that ensure that $m_p>0$.
\medbreak

\noindent \textbf{Acknowledgments.}
The authors thank   warmly L. Jeanjean for his help improving the presentation of the results.
\medbreak

%\vskip9pt
\begin{center}
\textbf{Notation.}
\end{center}
\begin{enumerate}{\small 
\item[1)] 
For $p\in [1,+\infty[$, the norm $(\int_{\Omega}|u|^pdx)^{1/p}$
in $L^p(\Omega)$ is denoted by $\|\cdot\|_p$. We denote by $p^{\prime}$
the  conjugate exponent of $p$, namely $p^{\prime} = p/(p-1)$ and by $p^*$ the Sobolev critical exponent i.e. 
$p^*=\frac{Np}{N-p}$ if $p<N$ and $p^*=+\infty$ in case $p\geq N$. The norm in $L^{\infty}(\Omega)$ is $\|u\|_{\infty}=\mbox{esssup}_{x\in \Omega}|u(x)|$.
\item[2)] For $v \in L^1(\Omega)$ we define $v^+= \max(v,0)$ and $v^- =
\max(-v,0)$.
\item[3)] The space $W^{1,p}_0(\Omega)$ is equipped with the
norm $\|u\|:=\big( \int_\Omega |\nabla u|^p\,dx\big)^{1/p}$.
%\item[6)] We denote by $C,D>0$ any positive constants which are not
%essential in the problem and may vary from one line to another.
\item[4)] We denote $\mathbb R^+=(0,+\infty)$ and $\mathbb R^-=(-\infty,0)$.}
\item[5)] For $a$, $b \in L^1(\Omega)$ we denote $\{a \leq b\} = \{x\in \Omega:\, a(x)\leq b(x)\}\,.$
\end{enumerate}

\section{Preliminaries} \label{P}
%\textcolor{red}{Je n'ai rien modifi\' e d'essentiel dans cette section} 

%\medbreak

In this section we present some definitions and  known results which are going to play an important role 
throughout all the work. First of all, we present some results on lower and upper solutions adapted to our 
setting. Let us consider the problem
\begin{equation} \label{p1}
\pLaplac u + H(x,u,\gradu) = f(x)\,, \quad u \in \Wop \cap \Linfty\,,
\end{equation} 
where $f$ belongs to $L^1(\Omega)$ and $H: \Omega \times \R \times \RN \rightarrow \R$ is a 
Carath\'eodory function.

\begin{definition} 
\label{lowerUpperSolution}
We say that $\alpha \in W^{1,p}(\Omega) \cap \Linfty$ is a \textit{lower solution of \eqref{p1}} if 
$\alpha^{+} \in \Wop$ and, for all $\varphi \in \Wop \cap \Linfty$ with $\varphi \geq 0$, if follows that
\[ 
\int_{\Omega} |\nabla \alpha|^{p-2} \nabla \alpha \nabla \varphi\,dx + \int_{\Omega} H(x,\alpha, \nabla \alpha) 
\varphi\, dx \leq \int_{\Omega} f(x) \varphi\,dx\,.
\]
Similarly, $\beta \in W^{1,p}(\Omega) \cap \Linfty$ is an \textit{upper solution of \eqref{p1}} if 
$\beta^{-} \in \Wop$ and, for all $\varphi \in \Wop \cap \Linfty$ with $\varphi \geq 0$, if follows that
\[ 
\int_{\Omega} |\nabla \beta|^{p-2} \nabla \beta \nabla \varphi\,dx + \int_{\Omega} H(x,\beta, \nabla \beta) 
\varphi\, dx \geq \int_{\Omega} f(x) \varphi\,dx\,.
\]
\end{definition} 

\begin{theorem} \rm \cite[Theorems 3.1 and 4.2]{B_M_P_1988} \label{BMP1988}
\it Assume the existence of a non-decreasing function $b: \R^{+} \rightarrow \R^{+}$ and a function 
$k \in L^1(\Omega)$ such that
\[ 
|H(x,s,\xi)| \leq b(|s|)[ k(x) + |\xi|^p ], \quad \textup{a.e. } x \in \Omega, \,\,
\forall (s,\xi) \in \R \times \RN\,.\]
If there exist a lower solution $\alpha$ and an upper solution $\beta$ of \eqref{p1} with $\alpha \leq \beta$, 
then there exists a solution $u$ of \eqref{p1} with $\alpha \leq u \leq \beta$.
Moreover, there exists $u_{min}$ (resp. $u_{max}$) minimum (resp. maximum) solution of \eqref{p1} with 
$\alpha \leq u_{min}\leq u_{max} \leq \beta$ and such that,  every solution $u$ of \eqref{p1} with
 $\alpha \leq u \leq \beta$ satisfies $u_{min}\leq u \leq u_{max}$.
  \end{theorem}

Next, we state  the strong comparison principle for the $p$-Laplacian and the following order notions.

\begin{definition} For $h_1$, $h_2\in L^1(\Omega)$ we write
\begin{itemize}
\item $h_1\leq h_2$ if $h_1(x)\leq h_2(x)$ for a.e. $x\in\Omega$,
\item $h_1\lneqq h_2$ if $h_1\leq h_2$ and 
$\mbox{meas}(\{x\in\Omega:
h_1(x)<h_2(x)\})>0$.
\end{itemize}
For $u$, $v \in \mathcal{C}^{1}(\overline{\Omega})$ we write
\begin{itemize}
\item $u < v$ if, for all $x \in \Omega\,,$ $u(x) < v(x)$,
\item $u \ll v$ if $u < v$  and, for all $x \in \partial \Omega$, either $u(x) < v(x)$, or, $u(x) = v(x)$ and $\frac{\partial u}{\partial \nu}(x) > \frac{\partial v}{\partial \nu}(x)$, where $\nu$ denotes the exterior unit 
normal. 
\end{itemize}
\end{definition}

\begin{theorem}\rm \cite[Theorem 1.3]{L_P_2004} \cite[Proposition 2.4]{C_T_2000}
\label{smpLP}
\it Assume that $\partial \Omega$ is of class $\mathcal{C}^2$ and let $f_1$, $f_2 \in L^{\infty}(\Omega)$ with 
$f_2 \gneqq f_1\geq0$. 
If $u_1$, $u_2 \in \mathcal{C}_0^{1,\tau}(\overline{\Omega})\,,\ 0 < \tau \leq 1\,,$ are respectively solution of
\[ \tag{$P_i$} \pLaplac u_i = f_i\,, \quad \textup{ in } \Omega\,,\ \textup{ for } i = 1,2\,,\]
such that $u_2 = u_1 = 0$ on $\partial \Omega$. % and $u_2 \geq u_1 \geq 0$. 
Then $u_2 \gg u_1$.
\end{theorem}

We need also the following anti-maximum principle.

\begin{prop}\rm \cite[Theorem 5.1]{G_G_P_2002}
\label{anti}
\it Let $\Omega \subset \R^N$, $N \geq 2$,  a bounded domain with  
$\partial \Omega$  of class   $\mathcal{C}^{1,1}$, 
$c$, $\bar{h}\in L^{\infty}(\Omega)$,
 $\gamma_1$  the  first eigenvalue of
\eqref{BVPeigenvalue}. If $\bar{h} \gneqq 0$,
%\begin{equation}
%\label{cl}
%\pLaplac  u  = \bar{\nu}_1 \bar{m}(x) u, \quad u \in W^{1,p}_0(\Omega).
%\end{equation}
then there exists $\delta_0>0$ such that, for all $\lambda \in (\gamma_1,\gamma_1+\delta_0)$, 
every  solution $w$ of
\begin{equation}
\label{cll}
\pLaplac w = \lambda \, c(x) |w|^{p-2}w  + \bar{h}(x),\quad u \in W^{1,p}_0(\Omega)
\end{equation}
satisfies $w\ll0$.
\end{prop}

The following result  is the well known Picone's inequality for the $p$-Laplacian. 
We state it for completeness.

\begin{prop}\rm \cite[Theorem 1.1]{A_H_1998}
\label{picone}
\it Let $u$, $v\in W^{1,p}(\Omega)$ with $u \geq 0$, $v > 0$  in $\Omega$ and 
$\frac{u}{v}\in L^{\infty}(\Omega)$. Denote
\begin{equation*}
\begin{aligned}
L(u,v) & = |\gradu|^p + (p-1)\Big( \frac{u}{v} \Big)^p |\gradv|^p 
- p \Big( \frac{u}{v} \Big)^{p-1} |\gradv|^{p-2}\gradv \gradu \, ,\\
R(u,v) & = |\gradu|^p -  \nabla \Big( \frac{u^p}{v^{p-1}} \Big)|\gradv|^{p-2}  \gradv\,.
\end{aligned}
\end{equation*}
Then, it follows that
\begin{itemize}
\item $L(u,v) = R(u,v) \geq 0$ a.e. in $\Omega$.
\item $L(u,v) = 0$ a.e. in $\Omega$ if, and only if, $u = kv$ for some constant $k \in \R$.
\end{itemize}
\end{prop}
\medbreak
Now, we consider the boundary value problem
\begin{equation} \label{p2}
\pLaplac v = g(x,v) , \qquad v \in \Wop \cap \Linfty\,,
\end{equation}
being $g: \Omega \times \R \rightarrow \R$ a Carath\'eodory function such that, for all 
$s_0 > 0$, there exists $A > 0$, with
\begin{equation} \label{linftyCarat}
 |g(x,s)| \leq A\,, 
\quad \textup{ a.e. } x \in \Omega\,, \,\, \forall\ s \in [-s_0,s_0]\,.
\end{equation}
This problem can be handled variationally. Let us consider the associated functional $\Phi: \Wop \rightarrow \R$ 
defined by
\[ 
\Phi(v) :=  \frac{1}{p} \int_{\Omega} |\gradv|^p\,dx - \int_{\Omega}G(x,v)\,dx\,,
\quad \mbox{ where }
\quad 
G(x,s) := \int_0^s g(x,t)\,dt\,.
\]
We can state the following result. 

\begin{prop}\rm \cite[Proposition 3.1]{DF_G_U_2009}
\label{minimumInterval}
\it Under the assumption \eqref{linftyCarat}, assume that $\alpha$ and $\beta$ are respectively a lower and an 
upper solution of \eqref{p2} with 
$\alpha \leq \beta$  and consider 
\[ M:= \big\{ v \in \Wop: \alpha \leq v \leq \beta \big\}.\]
Then the infimum of $\Phi$ on $M$ is achieved at some $v$, and such $v$ is a solution of \eqref{p2}.
\end{prop}

%In order to define the notion of strict lower and upper solutions, we need the following order notion.

%\begin{definition}
%Let $u$, $v \in \mathcal{C}^{1}(\overline{\Omega})$. We say that $u \ll v$ in $\Omega$ if: 
%for all $x \in \Omega$, 
%$u(x) < v(x)$ and, for all $x \in \partial \Omega$, either $u(x) < v(x)$, or, $u(x) = v(x)$ and 
%$\frac{\partial u}{\partial \nu}(x) > \frac{\partial v}{\partial \nu}(x)$, where $\nu$ denotes the exterior unit 
%normal. 
%\end{definition}

%\begin{definition}
%We say that $\alpha \in \mathcal{C}^1(\overline{\Omega})$ is a \textit{strict lower solution} of \eqref{p1} if for 
%every solution $u$ of \eqref{p1} such that $\alpha \leq u$ in $\Omega$ it follows that $\alpha \ll u$ in $\Omega$.

%Similarly, $\beta\in \mathcal{C}^1(\overline{\Omega})$ is said to be a \textit{strict upper solution} of 
%\eqref{p1} if every solution $u$ of \eqref{p1} such that $\beta \geq u$ in $\Omega$ satisfies $\beta \gg u$ in 
%$\Omega$.
%\end{definition}

\begin{definition}
A \textit{lower solution} $\alpha \in \mathcal{C}^1(\overline{\Omega})$ is said to be \textit{strict}
if every solution $u$ of \eqref{p1} with $u\geq \alpha$ satisfies $u\gg\alpha$. 
\smallbreak
Similarly, an \textit{upper solution} $\beta\in \mathcal{C}^1(\overline{\Omega})$  is said to be \textit{strict}
 if every solution $u$ of \eqref{p1} such that $u\leq \beta$ satisfies $u\ll\beta$.
\end{definition}

\begin{cor} \label{localMinimizer}
Assume that \eqref{linftyCarat} is valid and that $\alpha$ and $\beta$ are strict lower and upper solutions 
of \eqref{p2} belonging to 
$\mathcal{C}^{1}(\overline{\Omega})$ and satisfying $\alpha \ll \beta$. 
Then there exists a local minimizer $v$ of the functional $\Phi$ in the $\mathcal{C}_0^1$-topology. 
Furthermore, this minimizer is a solution of  \eqref{p2} with $\alpha \ll v \ll \beta$.
\end{cor}

\begin{proof}
First of all observe that  Proposition \ref{minimumInterval} implies the existence of  $v \in \Wop$ solution of 
\eqref{p2}, which minimizes $\Phi$ on $ M:= \{ v \in \Wop: \alpha \leq v \leq \beta\}$.
Moreover, since $g$ is an $L^{\infty}$-Carath\'eodory function, the classical regularity results 
(see \cite{DB_1983, L_1988}) imply that $v \in \mathcal{C}^{1,\tau}(\overline{\Omega})$ for some 
$0 < \tau < 1$.  Since the lower and the upper solutions are strict, it follows that $\alpha \ll v \ll \beta$ 
and so, there is a 
$\mathcal{C}_0^1$-neighbourhood of $v$ in $M$. Hence, it follows that $v$ minimizes locally $\Phi$ in the 
$\mathcal{C}_0^1$-topology. 
\end{proof}

\begin{prop}\rm \cite[Proposition 3.9]{DF_G_U_2009}
\label{c01vsWop}
\it Assume that $g$ satisfies the following growth condition
\[ |g(x,s)| \leq d \, (1+ |s|^{\sigma}), \quad \textup{a.e. }x \in \Omega\,,\ all\ s \in \R\,,\]
for some $\sigma \leq \past - 1$ and some positive constant $d$. Let $v \in \Wop$ be a local minimizer of 
$\Phi$ for the $\mathcal{C}_0^1$-topology. Then $v \in \mathcal{C}_0^{1,\tau}(\overline{\Omega})$ for some 
$0 < \tau < 1$ and $v$ is a local minimizer of $\Phi$ in the $\W$-topology.
\end{prop}

We now recall abstract results in order to find critical points of $\Phi$ other than local minima.

%We can look also to other critical points of functional than minimum.
%To this aim, let us recall some definition and variational results in an abstract framework.

%\begin{definition} \label{mountainPassGeometry}
%Given a real Banach space $ (X , \| \cdot \|) $ we say that a functional $\Phi: X \rightarrow \R$ has 
%mountain-pass geometry if:
%\begin{enumerate}
%\item There exist $x_1 \in X$ and constants $\rho, \gamma$ such that, for all $x \in \partial B(x_1,\rho)$, 
%$\Phi(x) \geq \gamma > \Phi(x_1)$;
%\item There exists $x_2 \in X \setminus B(x_1,\rho)$ such that $\Phi(x_2) < \gamma$.
%\end{enumerate}
%Moreover, we define the mountain-pass level $c$ as
%\[ c := \inf_{\varphi \in \Gamma} \max_{t \in [0,1]} \Phi(\varphi(t))\,,\]
%where
%\end{definition}

\begin{definition} \label{defCerami}
Let $ (X , \| \cdot \|) $  be a real Banach space with dual space $ (X^{\ast}, \| \cdot \|_{\ast}) $ and let 
$\Phi: X \rightarrow \R$ be a $\mathcal{C}^1$ functional. The functional $ \Phi $ satisfies the \textit{Cerami 
condition at level $c \in \R$} if, for any \textit{Cerami sequence} at level $c \in \R$, i.e. for any sequence 
$\{x_n\} \subset X$ with
\begin{equation*}
\begin{aligned}
&\Phi(x_n) \to c \quad \textup{ and } \quad \|\Phi'(x_n)\|_{\ast} (1+\|x_n\|) \to 0\,,
\end{aligned}
\end{equation*}
there exists a subsequence $\{x_{n_k}\}$ strongly convergent in $X$.
\end{definition}

\begin{theorem}\rm \cite[Corollary 9, Section 1, Chapter IV]{E_1990} 
\label{mpTheorem}
\it Let $ (X , \| \cdot \|) $  be a real Banach space. Suppose that $\Phi: X \rightarrow \R$ is a $\mathcal{C}^1$ 
functional. Take two points $e_1$, $e_2 \in X$ and define
\[ \Gamma := \{ \varphi \in \mathcal{C}([0,1],X): \varphi(0) = e_1, \,\varphi(1) = e_2\}\,,\]
and
\[ c := \inf_{\varphi \in \Gamma} \max_{t \in [0,1]} \Phi(\varphi(t))\,.\]
Assume that $\Phi$ satisfies the Cerami condition at level $c$ and that
\[ c > \max \{ \Phi(e_1),\Phi(e_2)\}\,.\]
Then, there is a critical point of $\Phi$ at level $c$, i.e. there exists 
$x_0 \in X$ such that $\Phi(x_0) = c$ and $\Phi'(x_0) = 0$.
\end{theorem}

\begin{theorem}\rm \cite[Corollary 1.6]{G_1993}
\label{characterizacionMinimizer}
\it Let $ (X , \| \cdot \|) $  be a real Banach space and let $\Phi: X \rightarrow \R$ be a $\mathcal{C}^{1}$ functional. 
Suppose that $u_0 \in X$ is a local minimum, i.e. there exists $\epsilon > 0$ such that
\[ \Phi(u_0) \leq \Phi(u), \quad \textup{ for } \|u-u_0\| \leq \epsilon\,,\]
and assume that $\Phi$ satisfies the Cerami condition at any level $d \in \R$. Then, the following alternative holds:
\begin{enumerate}
\item[i)] either there exists $0 < \gamma < \epsilon$ such that 
$ \inf \{\Phi(u): \|u-u_0\| = \gamma \} > \Phi(u_0),$
\item[ii)] or, for each $0 < \gamma < \epsilon$, $\Phi$ has a local minimum at a point $u_{\gamma}$ with 
$\|u_{\gamma}-u_0\| = \gamma$ and $\Phi(u_{\gamma}) = \Phi(u_0)$.
\end{enumerate}
\end{theorem}

\begin{remark}
In \cite{G_1993}, Theorem \ref{characterizacionMinimizer} is proved  assuming the Palais-Smale condition which 
is stronger than our Cerami condition. 
Nevertheless, modifying slightly the proof, it is possible to obtain the same result with the Cerami condition.
\end{remark}

\section{Comparison principle and uniqueness results} 
\label{CPAUR}

% \textcolor{red}{Je n'ai rien modifi\' e d'essentiel dans cette section, sauf peut \^etre un peu la r\'edaction \`a la fin de la Section} 

%\medbreak

In this section, we state a comparison principle
%degenerate elliptic equation with natural growth in the gradient 
and, as a consequence, we obtain uniqueness result for \eqref{Plambda} with $\lambda \leq 0$, proving Theorem \ref{uniqueness}. 
%Let $f: \Omega \times \R \rightarrow \R$ be a $L^1$-Carath\'eodory function, 
Consider the boundary value problem
\begin{equation} \label{pCompPrinciple}
\pLaplac u = \mu |\gradu|^p + f(x,u)\,, \quad u \in \Wop \cap \Linfty\,,
\end{equation}
under the assumption
\begin{equation} \label{aCompPrinciple} 
\left \{
\begin{aligned}
& \Omega \subset \R^N,\, N \geq 2\,, \textup{ is a bounded domain with } 
\partial \Omega \textup{ of class } \mathcal{C}^{0,1}\,,\\
&f: \Omega \times \R \to \R \mbox{ is  a  $L^1$-Carath\'eodory function with }
%\\
%& 
f(x,s) \leq f(x,t) \textup{ for a.e. } x\in  \Omega\mbox{ and all } t \leq s,
\\
& \mu >0.
\end{aligned}
\right.
\end{equation}

\begin{remark} 
\label{muPositve2}
As above, the assumption $\mu>0$ is not a restriction. If $u \in \Wop \cap \Linfty$ is a solution of 
\eqref{pCompPrinciple} with $\mu < 0$ then 
$w = -u \in \Wop \cap \Linfty$ is a solution of
\begin{equation*} \label{pCoercive}
\pLaplac w =  - \mu|\nabla w|^p - f(x,-w)\,, \quad w \in \Wop \cap L^{\infty}(\Omega)\,,
\end{equation*}
with $-f(x,-s)$ satisfying the assumption \eqref{aCompPrinciple}.
\end{remark}

Under a stronger regularity on the solutions, we can prove a comparison principle for \eqref{pCompPrinciple}. 
The proof relies on the Picone's inequality (Proposition \ref{picone}) and is inspired by some ideas of 
\cite{A_P_2003}.

\begin{theorem} \label{compPrinciple}
Assume that \eqref{aCompPrinciple} holds. 
If $u_1$, $u_2 \in W^{1,p}(\Omega) \cap \mathcal{C}(\overline{\Omega})$ are respectively a lower and an upper 
solution of \eqref{pCompPrinciple}, then $u_1 \leq u_2$.
\end{theorem}

\begin{proof}
Suppose that $u_1$, $u_2$ are respectively a lower and an upper solution of \eqref{pCompPrinciple}. 
For simplicity denote $t = \frac{p\mu}{p-1}$ and consider as test function
\[ \varphi = \big[ e^{tu_1} - e^{tu_2} \big]^{+} \in \Wop \cap \Linfty\,.\]
First of all, observe that
\[ \nabla \varphi = t \big[ \nabla u_1 e^{tu_1} - \nabla u_2 e^{tu_2} \big] \chi_{\{u_1 > u_2\}},\]
with $\chi_A$ the characteristic function of the set $A$.
Hence, using assumptions \eqref{aCompPrinciple}, it follows that 
\begin{equation*}
\begin{aligned}
\int_{\{u_1 > u_2\}} \!\!\! \Big(\big[  |\nabla u_1|^{p-2} \nabla u_1 - |\nabla u_2 |^{p-2} \nabla u_2 \big] 
\big( t\,  \nabla u_1 e^{tu_1} -& t\, \nabla u_2 e^{tu_2}  \big)
 - \mu 
 %\int_{\{u_1 > u_2\}} 
 \big[ |\nabla u_1|^p - |\nabla u_2|^p \big] \big(e^{tu_1} - e^{tu_2} \big)\Big)\, dx \\
& \leq \int_{\{u_1 > u_2\}}  \big( f(x,u_1) - f(x,u_2) \big) \big(e^{tu_1} - e^{tu_2} \big)\, dx 
\leq 0.
\end{aligned}
\end{equation*}
%Now, observe that under the assumptions \eqref{aCompPrinciple} it follows that
%\[ \int_{\{u_1 > u_2\}}  \big( f(x,u_1) - f(x,u_2) \big) \big(e^{tu_1} - e^{tu_2} \big)\, dx \leq 0 \,,\]
Observe that
\begin{equation} \label{cp1}
\begin{aligned}
%0 & \geq 
\int_{\{u_1 > u_2\}} & \Big[  |\nabla u_1|^{p-2} \nabla u_1 - |\nabla u_2 |^{p-2} \nabla u_2 \Big]
\big( t\,  \nabla u_1 e^{tu_1} - t\, \nabla u_2 e^{tu_2}  \big)\, dx 
\\
& \hspace{37mm} 
- \mu \int_{\{u_1 > u_2\}} \Big[ |\nabla u_1|^p - |\nabla u_2|^p \Big] \big(e^{tu_1} - e^{tu_2} \big)\, dx \\
& = \int_{\{u_1 > u_2\}} e^{tu_1} \Big[ |\nabla u_1|^p (t - \mu) 
+ \mu |\nabla u_2 |^p - t |\nabla u_2|^{p-2} \nabla u_2 \nabla u_1 \Big] \, dx \\
& \hspace{37mm} 
+ \int_{\{u_1 > u_2\}} e^{tu_2} \Big[ |\nabla u_2|^p (t - \mu) 
+ \mu |\nabla u_1 |^p - t |\nabla u_1|^{p-2} \nabla u_1 \nabla u_2 \Big] \, dx .
\end{aligned}
\end{equation}
Next, as $\nabla e^{tu_i} = t\, \nabla u_i e^{tu_i}$, $i = 1,2$, we have
\[ |\nabla u_i |^p = \frac{|\nabla e^{t u_i}|^p}{t^p e^{tpu_i}} \qquad i = 1,2\,.\]
Hence, using the above identities, and as $\frac{\mu}{t-\mu} = p-1$ and $\frac{t}{t-\mu} = p$, 
it follows that, %for $i=1$, $2$,
\begin{equation*}
\begin{aligned}
%\int_{\{u_1 > u_2\}} 
e^{tu_1} \big[ |\nabla u_1|^p (t - \mu) 
&
+ \mu |\nabla u_2 |^p - t |\nabla u_2|^{p-2} \nabla u_2 \nabla u_1 \big]% \, dx 
%\\
%& 
%= 
%\int_{\{u_1 > u_2\}} e^{tu_1} 
%\Big[ \frac{|\nabla e^{tu_1}|^p}{e^{tpu_1}} \frac{t-\mu}{t^p} 
%+ \frac{|\nabla e^{tu_2}|^p}{e^{ptu_2}} \frac{\mu}{t^p} - t \frac{|\nabla e^{tu_2}|^{p-2}
 %\nabla e^{tu_2} \nabla e^{tu_1} }{t^p e^{t(p-1)u_2} e^{tu_1}} \Big]\, dx 
%\\
%& = \int_{\{u_1 > u_2\}} \frac{e^{tu_1}}{t^p e^{tpu_1}} \Bigg[ |\nabla e^{tu_1}|^p (t - \mu) 
%+ \mu \left( \frac{e^{tu_1}}{e^{tu_2}} \right)^p |\nabla e^{tu_2}|^p \\
%& \hspace{4.76cm} \qquad - t \left( \frac{e^{tu_1}}{e^{tu_2}} \right)^{p-1} |\nabla e^{tu_2}|^{p-2}
% \nabla e^{tu_2} \nabla e^{tu_1} \Bigg]\, dx \\
%& 
%= 
%\int_{\{u_1 > u_2\}} \frac{t-\mu}{t^p e^{t(p-1)u_1}} \Big[ |\nabla e^{tu_1}|^p 
%+ \frac{\mu}{t-\mu} \big( \frac{e^{tu_1}}{e^{tu_2}} \big)^p  |\nabla e^{tu_2} |^p \\
%& \hspace{4.28cm} \qquad - \frac{t}{t-\mu} \big( \frac{e^{tu_1}}{e^{tu_2}} \big)^{p-1} 
%|\nabla e^{tu_2}|^{p-2} \nabla e^{tu_2} \nabla e^{tu_1} \Big]\, dx
%\end{aligned}
%\end{equation*}
%Moreover, as $\frac{\mu}{t-\mu} = p-1$ and $\frac{t}{t-\mu} = p$, it follows that
%\begin{equation*}
%\begin{aligned}
 %\int_{\{u_1 > u_2\}} e^{tu_1} \big[ |\nabla u_1|^p (t - \mu) 
% &
% + \mu |\nabla u_2 |^p - t |\nabla u_2|^{p-2} \nabla u_2 \nabla u_1 \big] \, dx 
\\
& = 
%\int_{\{u_1 > u_2\}} 
\frac{t-\mu}{t^p e^{t(p-1)u_1}} 
\Big[ |\nabla e^{tu_1}|^p + (p-1) \big( \frac{e^{tu_1}}{e^{tu_2}} \big)^p  |\nabla e^{tu_2} |^p
%\\
%& \hspace{50mm} \qquad 
- p \big( \frac{e^{tu_1}}{e^{tu_2}} \big)^{p-1} |\nabla e^{tu_2}|^{p-2} 
\nabla e^{tu_2} \nabla e^{tu_1} \Big],%\, dx\,.
\\
%\end{aligned}
%\end{equation*}
%In the same way, we have
%\begin{equation*}
%\begin{aligned}
%\int_{\{u_1 > u_2\}} 
e^{tu_2} \big[ |\nabla u_2|^p (t - \mu) 
&+ \mu |\nabla u_1 |^p - t |\nabla u_1|^{p-2} \nabla u_1 \nabla u_2 \big] %\, dx 
\\
& 
= 
%\int_{\{u_1 > u_2\}} 
\frac{t-\mu}{t^p e^{t(p-1)u_2}} 
\Big[ |\nabla e^{tu_2}|^p + (p-1) \big( \frac{e^{tu_2}}{e^{tu_1}} \big)^p  |\nabla e^{tu_1} |^p 
%\\
%& \hspace{4.28cm} \qquad 
- p \big( \frac{e^{tu_2}}{e^{tu_1}} \big)^{p-1} |\nabla e^{tu_1}|^{p-2} 
\nabla e^{tu_1} \nabla e^{tu_2} \Big].%\, dx\,.
\end{aligned}
\end{equation*}
Then, by \eqref{cp1}, we have
\begin{equation} 
\label{cp2}
\begin{aligned}
& \int_{\{u_1 > u_2\}} \frac{t-\mu}{t^p e^{t(p-1)u_1}} 
 \Big[ |\nabla e^{tu_1}|^p + (p-1) \big( \frac{e^{tu_1}}{e^{tu_2}} \big)^p  |\nabla e^{tu_2} |^p 
 %\\
%& \hspace{4.28cm} \qquad 
- p \big( \frac{e^{tu_1}}{e^{tu_2}} \big)^{p-1} |\nabla e^{tu_2}|^{p-2}
 \nabla e^{tu_2} \nabla e^{tu_1} \Big] dx \\
& 
+ \int_{\{u_1 > u_2\}} \frac{t-\mu}{t^p e^{t(p-1)u_2}} \Big[ |\nabla e^{tu_2}|^p + (p-1)
 \big( \frac{e^{tu_2}}{e^{tu_1}} \big)^p  |\nabla e^{tu_1} |^p 
 %\\
%& \hspace{4.28cm} \qquad 
- p \big( \frac{e^{tu_2}}{e^{tu_1}} \big)^{p-1} |\nabla e^{tu_1}|^{p-2} 
\nabla e^{tu_1} \nabla e^{tu_2} \Big] dx
%\\
%&\hspace{150mm}
 \leq 0\,.\hspace{-2.3mm}
\end{aligned}
\end{equation}
By  Picone's inequality (Proposition \ref{picone}), we know that both brackets in \eqref{cp2} are positive and 
are equal to zero if and only if $e^{tu_1}=k e^{tu_2}$ for some $k\in \mathbb R$. 
As $t-\mu > 0$, thanks to \eqref{cp2}, we deduce the existence of $k\in \mathbb R$ such that
\begin{equation}
\label{15v}
e^{tu_1}=k e^{tu_2}  \quad \mbox{ in } \{u_1>u_2\}.
\end{equation}
Since $u_1$ and $u_2$ are continuous on $\overline\Omega$ and satisfy $u_1-u_2\leq 0$ on $\partial\Omega$, 
we deduce that $u_1=u_2$ on $\partial  \{u_1>u_2\}$. Hence, \eqref{15v} applied to $x\in \partial  \{u_1>u_2\}$, 
implies $k=1$. This implies that $u_1=u_2$ in $\{u_1>u_2\}$, which proves $u_1\leq u_2\,,$ as desired.
%\begin{equation*}
%\begin{aligned}
%0 \leq & \int_{\{u_1 > u_2\}} \frac{t-\mu}{t^p e^{t(p-1)u_1}} \Bigg[ |\nabla e^{tu_1}|^p + (p-1) 
%\left( \frac{e^{tu_1}}{e^{tu_2}} \right)^p  |\nabla e^{tu_2} |^p \\
%& \hspace{4.28cm} \qquad - p \left( \frac{e^{tu_1}}{e^{tu_2}} \right)^{p-1} |\nabla e^{tu_2}|^{p-2} 
%\nabla e^{tu_2} \nabla e^{tu_1} \Bigg]\, dx \\
%& + \int_{\{u_1 > u_2\}} \frac{t-\mu}{t^p e^{t(p-1)u_2}} \Bigg[ |\nabla e^{tu_2}|^p + (p-1) 
%\left( \frac{e^{tu_2}}{e^{tu_1}} \right)^p  |\nabla e^{tu_1} |^p \\
%& \hspace{4.28cm} \qquad - p \left( \frac{e^{tu_2}}{e^{tu_1}} \right)^{p-1} |\nabla e^{tu_1}|^{p-2} 
%\nabla e^{tu_1} \nabla e^{tu_2} \Bigg]\, dx \leq 0\,,
%\end{aligned}
%\end{equation*}
%Hence, by taking int account again the Proposition \ref{picone}, it follows that 
%\[ e^{tu_1} = e^{\frac{p\mu}{p-1}u_1} = k e^{tu_2} = k e^{\frac{p\mu}{p-1}u_2}\,
% \qquad \textup{ in } \{u_1 > u_2\}\,.\]
%Then, observe that, this implies that $k > 1$ and so, it follows that
%\[ u_1 = u_2 + \frac{p-1}{p\mu}\ln k\,, \qquad \textup{ in } \{u_1 > u_2\}\,,\]
%being $ \ln k > 0$. Nevertheless, by using that $u_1$ and $u_2$ belong to $\mathcal{C}(\overline{\Omega})$ 
%and the boundary condition we deduce that the above identity is not possible and so, that $|\{u_1 > u_2\}| = 0$. 
%Consequently, we can conclude that $u_1 \leq u_2$ in $\Omega$, as desired.
%% \textbf{\textcolor{red}{Maybe, we should explain this a little bit better. I see it clear but... I don't know.}}
\end{proof}

\begin{cor} \label{compPrincipPlambda}
Assume that \eqref{A1} holds and suppose $\lambda \leq 0$. If $u_1$, $u_2 \in W^{1,p}(\Omega) \cap \mathcal{C}(\overline{\Omega})$ are respectively a lower and an upper 
solution of \eqref{Plambda}, then $u_1 \leq u_2$.
\end{cor}

\begin{proof}
Define the function $f: \Omega \times \R \rightarrow \R$ given by
\[ f(x,s) = \lambda c(x)|s|^{p-2}s + h(x)\,.\]
Since \eqref{A1} holds and $\lambda \leq 0$, $f$ is a $L^1$-Carath\'eodory function which satisfies 
\eqref{aCompPrinciple}. Consequently, the proposition follows  from Theorem \ref{compPrinciple}
\end{proof}

The following result guarantees the regularity that we need  to apply the previous comparison principle.

\begin{lemma} \label{lemmaUniqueness}
Assume that \eqref{A1} holds and suppose $\lambda \leq 0$. Then, any solution of \eqref{Plambda} belongs to 
$\mathcal{C}^{0,\tau}(\overline{\Omega})$.
\end{lemma}

\begin{proof}
This follows directly from \cite[Theorem IX-2.2]{L_U_1968}.
\end{proof}

%\begin{prop} \label{uniqueness}
%Assume that \eqref{A1} holds and suppose $\lambda \leq 0$. Then \eqref{Plambda} has at most one solution.
%\end{prop}

\begin{proof}[\textbf{Proof of Theorem \ref{uniqueness}}]
The proof is just the combination of Corollary \ref{compPrincipPlambda} and Lemma \ref{lemmaUniqueness}. % \textbf{\textcolor{red}{Lemmata or Lemmas?}}
\end{proof}

\begin{remark}
It is important to note that this comparison and uniqueness results do not hold in general for solution belonging 
only to 
$\Wop$. See \cite[Example 1.1]{P_2008}
\end{remark}

%\textcolor{red}{
\begin{remark}\label{pascomparaison}
The following counter-example  (see \cite[p.7]{P_S_2007}) shows that there is no hope to obtain, when $p \neq 2$, a comparison principle like \cite[Corollary 3.1]{A_DC_J_T_2014}. For $N = 2$ and $R>0$, consider  the following problem on the ball
\begin{equation*}
\left\{
\begin{aligned}
 -\Delta_4 u & = |\gradu|^2 \qquad  \textup{ in } B(0,R)\,,\\
u & = 0   \qquad \qquad \textup{ on } \partial B(0,R)\,. 
\end{aligned}
\right.
\end{equation*}
We easily see that  $u_1 = 0$ and $u_2 = \frac{1}{8}(R^2 - |x|^2)$ are both solutions of the above problem 
belonging to $W_0^{1,4}(B(0,R)) \cap L^{\infty}(B(0,R))$.
\end{remark}%}

\section{A priori lower bound and existence of a lower solution }
\label{TLSATAP-L}
As explained in the introduction, the aim of this section is to find a lower solution below every upper solution of problem \eqref{Plambda}.
First of all, we show that under a rather mild assumption (in particular no sign on $c$ is required) the solutions to \eqref{Plambda} admit a lower bound. Precisely
we consider problem \eqref{Plambda} assuming now
\begin{equation}  \label{lb2}
\left \{
\begin{aligned}
& \Omega \subset \R^N,\, N \geq 2\,, \textup{ is a bounded domain with } \partial \Omega \textup{ of class }
 \mathcal{C}^{0,1}\,.\\
& c \textup{ and } h \textup{ belong to } L^q(\Omega) \textup{ for some } q > \max\{N/p,1\}, \\
& \mu \in L^{\infty}(\Omega) \textup{ satisfies } 0 < \mu_1 \leq \mu (x) \leq \mu_2.
\end{aligned}
\right.
\end{equation}
Adapting the proof of \cite[Lemma 3.1]{DC_J_2017}, based in turn on ideas of \cite{A_P_P_2007}, we obtain 
%, and so, on the solutions of \eqref{Plambda}  under the assumptions \eqref{A1}.
% with $\mu > 0$. Nevertheless, as we have already pointed out in Remark \ref{muPositive1}, 
%to assume $\mu > 0$ is not a restriction, since $h$ does not have a sign. \\

% \textbf{\textcolor{red}{I have to modify this. I think we don't need that the dependence of the constant is 
%uniform om $\lambda$. We can write for each $\lambda$. I don't care too much.}}

% \textbf{\textcolor{red}{Here, we will have to break by hand the dependence on $M$ but, 
%for the moment as I am not very convinced of the margin and the size of the letter... I let it be.}}

\begin{lemma}\label{lowerBound}
Under the assumptions \eqref{lb2}, for any $\lambda \geq 0$, there exists a constant 
$M_{\lambda}>0$ with  $M_{\lambda}:= M(N,p,q,|\Omega|,\lambda, \mu_1,$ $\|c^{+}\|_q,\|h^-\|_q ) > 0$
% $M := M(p,\, \diam(\Omega),\,\Lambda, \\ \mu_1,\, \|c^{+}\|_q,\, \|h^{-}\|_q) > 0$
such that, every $u \in W^{1,p}(\Omega) \cap \Linfty$ upper solution of \eqref{Plambda} satisfies
\[ \min_{\Omega} u > -M_{\lambda}.\]
\end{lemma}

\begin{proof} Let us split the proof in two steps.
\medbreak
\noindent \textbf{Step 1:} \textit{There exists a positive constant $M_1 = M_1(p,q,N,|\Omega|,\lambda,\mu_1,
 \|c^{+}\|_q, \|h^{-}\|_q) > 0$ such that $\|u^{-}\| \leq M_1 $.}
\medbreak
First of all, observe that for every function $u\in W^{1,p}(\Omega)$, it follows that
\begin{equation}
\label{(1)}
\nabla \big( (u^{-})^{\frac{p+1}{p}} \big) = \frac{p+1}{p} (u^{-})^{1/p} \nabla u^{-},
\quad \mbox{ and so, } \quad
 |\nabla u^{-}|^p u^{-} = \big( \frac{p}{p+1}\big)^p \big| \nabla (u^{-})^{\frac{p+1}{p}} \big|^p.
 \end{equation}
 Suppose that $u \in W^{1,p}(\Omega) \cap \Linfty$ is an upper solution of \eqref{Plambda}
and let us consider $\varphi = u^{-}$ as a test function. Under the assumptions \eqref{lb2}, it follows that
\begin{equation}
\label{(2)}
\begin{aligned}
-\int_{\Omega} |\nabla u^{-}|^p dx & \geq - \lambda \int_{\Omega} c(x)|u^{-}|^p dx 
+ \int_{\Omega} \mu(x) |\nabla u^{-}|^p u^{-} dx + \int_{\Omega} h(x)u^{-} dx \\
& \geq -\lambda \int_{\Omega} c^{+}(x)|u^{-}|^p dx + \mu_1 \int_{\Omega} |\nabla u^{-}|^p u^{-} dx 
- \int_{\Omega} h^{-}(x) u^{-} dx .
%& = - \lambda \int_{\Omega} c^{+}(x) |u^{-}|^p + \mu_1 \left( \frac{p}{p+1}\right)^p \int_{\Omega} 
%\left| \nabla (u^{-})^{\frac{p+1}{p}} \right|^p dx - \int_{\Omega} h^{-}(x) u^{-} dx \, ,
\end{aligned}
\end{equation}
By \eqref{(1)} and \eqref{(2)}, we have that
\begin{equation} \label{lb3}
\mu_1 \Big( \frac{p}{p+1}\Big)^p \int_{\Omega} \big| \nabla (u^{-})^{\frac{p+1}{p}} \big|^p dx 
+ \int_{\Omega} |\nabla u^{-}|^p dx \leq \lambda \int_{\Omega} c^{+}(x)|u^{-}|^p dx 
+ \int_{\Omega} h^{-}(x)u^{-} dx\,.
\end{equation}
%Now, let us distinguish two cases. First of all, let us assume that $\lambda = 0$. 
%Then, the above inequality implies that
%\begin{equation*}
%\mu_1 \left( \frac{p}{p+1}\right)^p \int_{\Omega} \left| \nabla (u^{-})^{\frac{p+1}{p}} \right|^p dx 
%+ \int_{\Omega} |\nabla u^{-}|^p dx \leq \int_{\Omega} h^{-}(x)u^{-} dx 
%\leq \frac{1}{S_N}\|h^-\|_q \|\nabla u^{-}\|_p\,,
%\end{equation*}
%where we have applied H\"older and Sobolev inequalities also. So, we can conclude that
%\[ \|u^{-}\| \leq \left( \frac{1}{S_N} \|h^{-}\|_q \right)^{\frac{1}{p-1}} =: M_1\,.\]
%Now, let us assume that $\lambda > 0$. 
Firstly, we apply Young's inequality and, for every $\epsilon > 0$, it follows that
\begin{equation*}
\begin{aligned}
\int_{\Omega} c^{+}(x)|u^{-}|^p dx & = \int_{\Omega} (c^{+}(x))^{1/p} |u^{-}|^{1/p} 
(c^{+}(x))^{\frac{p-1}{p}} |u^{-}|^{\frac{(p+1)(p-1)}{p}} dx \\
& \leq C(\epsilon) \int_{\Omega} c^{+}(x)u^{-} dx 
+ \epsilon  \int_{\Omega} c^{+}(x) \big( (u^{-})^{\frac{p+1}{p}} \big)^p dx
\end{aligned}
\end{equation*}
Moreover, applying H\"older and Sobolev inequalities, observe that
\begin{equation*}
\int_{\Omega} c^{+}(x) \big( (u^{-})^{\frac{p+1}{p}} \big)^p dx 
\leq \|c^{+}\|_q \|(u^{-})^{\frac{p+1}{p}} \|_{\frac{qp}{q-1}}^p \leq S\, \|c^{+}\|_q 
\| \nabla (u^{-})^{\frac{p+1}{p}}\|_p^p
\end{equation*}
with $S$ the constant from the embedding from $W^{1,p}_0(\Omega)$ into $L^{\frac{qp}{q-1}}(\Omega)$. 
Hence, choosing $\epsilon$ small enough to ensure that 
$\epsilon\,S\,\lambda  \|c^{+}\|_q\leq  \frac{\mu_1}{2 } \big( \frac{p}{p+1} \big)^p$ 
and substituting in \eqref{lb3}, we apply again H\"older and Sobolev inequalities and we find a constant 
$C=C(\mu_1,\lambda,  \|c^{+}\|_q, p, q, |\Omega|, N)$ such that
\begin{equation*}
\begin{aligned}
\frac{\mu_1}{2} \Big( \frac{p}{p+1} \Big)^p \|\nabla (u^{-})^{\frac{p+1}{p}} \|_p^p + \|\nabla u^{-}\|_p^p 
%&
 \leq 
 \big(\|h^{-}\|_{q}  + C(\epsilon) \|c^{+}\|_q\big) \|u^{-}\|_{\frac{q}{q-1}} 
%\\
%&
 \leq 
C ( \|h^{-}\|_q + \|c^{+}\|_q )\|\nabla u^{-}\|_p.
%\\
%& = \frac{1}{S_N} \|\nabla u^{-}\|_p \left( \|h^{-}\|_q + C(\epsilon) \|c^{+}\|_q \right).
\end{aligned}
\end{equation*}
This allows to conclude that
\begin{equation*}
\|u^{-}\| \leq  \Big(C\big( \|h^{-}\|_q + \|c^{+}\|_q \big)\Big)^{\frac{1}{p-1}} =: M_1 \, .
\end{equation*}

\noindent \textbf{Step 2:} \textit{Conclusion.}
\medbreak
Since \eqref{lb2} holds, every $u \in W^{1,p} (\Omega)\cap \Linfty $ upper solution of \eqref{Plambda} 
satisfies
%\begin{equation*}
%\begin{aligned}
%\pLaplac u & = \lambda c(x)|u|^{p-2}u + \mu(x)|\gradu|^p + h(x) \\
%& \geq \lambda c(x)|u|^{p-2}u + \mu_1|\gradu|^p + h(x) \\
%& \geq \Lambda c(x)|u|^{p-2}u - h^{-}(x) \textup{ in } \Omega.
%\end{aligned}
%\end{equation*} 
%That is, $u$ is a $\Wop \cap L^{\infty}(\Omega)$ solution of
\begin{equation} \label{lbS2}
\pLaplac u \geq \lambda c(x)|u|^{p-2}u - h^-(x)\,, \quad \textup{ in } \Omega\,.
\end{equation}
Moreover, observe that $0$ is also an upper solution of \eqref{lbS2}. Hence, since the minimum of two upper solution is an upper solution (see \cite[Corollary 3.3]{C_1997}), it follows that $\min(u,0)$ is an upper solution of \eqref{lbS2}. Furthermore, observe that $\min(u,0)$ is an upper solution of
\[ \pLaplac u \geq \lambda c^{+}(x) |u|^{p-2} u - h^{-}(x)\,, \quad \textup{ in } \Omega\,.\]
Hence, applying \cite[Theorem 6.1.2]{P_S_2007}, we have the existence of 
$M_2 = M_{2}(N,p,\lambda,|\Omega|,\|c^{+}\|_q) > 0$ and 
$M_3 = M_{3}(N,p,\lambda,|\Omega|,\|c^{+}\|_q) > 0$ such that
\[ \sup_{\Omega}u^{-} \leq  M_{2} \big[ \|u^{-}\|_p + \|h^-\|_q \big] 
\leq M_3 \big[\|u^{-}\| + \|h^-\|_q \big].\]
%Now, observe that $M_2 \in (0,\infty)$ for any $\lambda \in [0,\Lambda]$. Consequently, we can define 
%$M_3 = \max_{\lambda \in [0,\Lambda]} M_2$ and observe that 
%$M_3 = M_3(N,p,\Lambda,|\Omega|,\|c\|_q) > 0$ and it follows that
%\[ \sup_{\Omega}u^{-} \leq M_3 \left[ \frac{1}{S_N} \|u^{-}\| + \|h\|_q \right].\]
Finally,
the result follows  by  Step 1.
%we can conclude that there exists $M := M(N,p,|\Omega|,\lambda, \mu_1,\|c\|_q,\|h\|_q ) > 0$ such that
%\[ u \geq -M \textup{ in } \Omega\,.\]
\end{proof}

\begin{remark} $ $
\begin{itemize}
\item[a)]Observe that the lower bound does not depend on $h^+$ and $c^{-}$. In particular, we have the same lower bound for all 
$h\geq 0$ and all $c \leq 0$. \label{remuniflb}
\item[b)]Since $c $ does not have a sign, there is no loss of generality in assuming $\lambda \geq 0$. If we consider $\lambda \leq 0$, we recover the same result with $M_{\lambda}$ depending on $\|c^{-}\|_q$ instead of $\|c^{+}\|_q\,.$
\end{itemize}
\end{remark}
%Observe that the lower bound does not depend on $h^+$ and $c^{-}$. In particular, we have the same lower bound for all 
%$h\geq 0$ and all $c \leq 0$. \label{remuniflb}
%\begin{itemize}
%\item[a)] We can see from the proof that, for every $\Lambda>0$, 
%the constant $M_{\lambda}>0$ can be chosen independent of 
%$\lambda\in [- \Lambda,\Lambda]$.
%\item[b)] 
%\end{itemize}

%\textcolor{red}{Je ne suis pas persuad\'e de l'int\' er\^et de la remarque ci-dessus. Si elle n'est pas utilis\'ee dans les preuves des Th\'eor\`emes \ref{th2} o\`u \ref{th4} je propose de la supprimer.} 

\begin{prop} \label{lowerSol}
Under the assumptions \eqref{A1}, for any $\lambda \in \R$, there exists 
$\underline{u}_{\lambda} \in \Wop \cap \Linfty$ lower solution of \eqref{Plambda} such that, for every $\beta$ 
upper solution of \eqref{Plambda}, we have 
$\underline{u}_{\lambda} \leq \min\{0,\beta\}$.
\end{prop}

\begin{proof}
We need to distinguish in our proof the cases $\lambda \leq 0$ and $\lambda \geq 0$. First we assume that $\lambda \leq 0$.
By Lemma \ref{lowerBound}, we have  a constant $M > 0$ such that every upper solution $\beta$  
of \eqref{Plambda} satisfies 
$\beta \geq -M$. 
%We consider the auxiliary problems
%\begin{equation}\label{B3}
%\left\{
%\begin{aligned}
%\pLaplac u & = \mu |\gradu|^p - h^{-}(x), & \textup{ in } \Omega,
%\\
%u & = 0, & \textup{ on } \partial \Omega,
%\end{aligned}
%\right.
%\end{equation}
%and
Let $\alpha$ be the solution of
\begin{equation*}\label{B4}
\pLaplac u  = - h^{-}(x),  \quad u \in \Wop \cap L^{\infty}(\Omega).
%\\
%u & = 0, & \textup{ on } \partial \Omega\,.
%\end{aligned}
%\right.
\end{equation*}
%Let us denote by $\alpha \in \Wop \cap L^{\infty}(\Omega)$ the solution of \eqref{B4}. 
%Observe that
%\[ \pLaplac \alpha = -h^{-}(x) -1 \leq \mu |\nabla \alpha|^p -h^{-}(x)-1\,, \quad \textup{ in } \Omega\,,\]
%and $\alpha = 0 $ on $\partial \Omega$. 
%It is easy to observe that  $\alpha$ is a non-positive lower solution of \eqref{B3}. 
%On the other hand, $0$ is an upper solution of \eqref{B3}. Hence, Theorem \ref{BMP1988} implies that the 
%existence of  $u \in \Wop \cap \Linfty$ solution of \eqref{B3} with $\alpha \leq u \leq 0$. 
%Observe that $\alpha\leq 0$. 
It is then easy to prove that  
$\underline{u} = \alpha -M \in W^{1,p}(\Omega) \cap \Linfty$ 
%and observe that
%\[ 
%\pLaplac \underline{u} \leq d(x) \underline{u} + \mu |\nabla \underline{u} |^p + h(x)\,, 
%\textup{ in } \Omega\,,
%\]
%and $\underline{u} \leq -M$ and $\underline{u} = -M$ on $\partial \Omega$. Hence, 
%$\underline{u}$ 
is a lower solution of \eqref{Plambda} with $\underline{u} \leq -M$. By the choice of $M$, this implies that  
$\underline{u}\leq \overline{u}$
for every upper solution $\overline{u}$  of \eqref{Plambda}.
\medbreak
Now, when $\lambda \geq 0$ we first introduce the auxiliary problem
\begin{equation} \label{ls1}
\left \{
\begin{aligned}
\pLaplac u & = \lambda c(x) |u|^{p-2} u + \mu |\gradu|^p - h^{-}(x) -1, & \textup{ in } \Omega\, , \\
u & = 0\, , & \textup{ on } \partial \Omega\, .
\end{aligned}
\right.
\end{equation}
Thanks to the previous lemma, there exists $M_{\lambda} > 0$ such that, 
for every $\beta_1 \in W^{1,p} (\Omega)\cap \Linfty$ upper solution of \eqref{ls1}, 
we have $\beta_1 \geq -M_{\lambda}$. Now, for $k > M_{\lambda}$, we introduce the problem
\begin{equation} \label{ls2}
\left \{
\begin{aligned}
\pLaplac u & = -\lambda c(x) k^{p-1}  - h^{-}(x) - 1\,, & \textup{ in } \Omega\, , \\
u & = 0\,,  & \textup{ on } \partial \Omega\, ,
\end{aligned}
\right.
\end{equation}
and denote by $\alpha_{\lambda}$ its solution. Since $-\lambda  c(x) k^{p-1}- h^{-}(x) - 1 < 0$, the comparison
principle  (see for instance \cite[Lemma A.0.7]{P_1997})
 implies that  $\alpha_{\lambda} \leq 0$. 
Observe that, for every $\beta_1$ upper solution of \eqref{Plambda}, we have that
\begin{equation*}
%\begin{aligned}
\pLaplac \beta_1 
%& 
%= \lambda c(x)|\beta_1|^{p-2}\beta_1 + \mu |\nabla \beta_1|^p + h(x) 
%\\
%&
 \geq \lambda c(x)|\beta_1|^{p-2}\beta_1 + \mu |\nabla \beta_1|^p -h^{-}(x) -1 
 %\\
%&
 \geq - \lambda c(x) k^{p-1} - h^{-}(x) - 1 = \pLaplac \alpha_{\lambda}.
%\end{aligned}
\end{equation*}
Consequently, it follows that
\begin{equation*}
\left\{
\begin{aligned}
\pLaplac \beta_1 & \geq \pLaplac \alpha_{\lambda}\,, & \textup{ in } \Omega\,,\\
 \beta_1 & \geq \alpha_{\lambda} = 0\,, & \textup{ on } \partial \Omega,
\end{aligned}
\right.
\end{equation*}
and, applying again the comparison principle, that $\beta_1 \geq \alpha_{\lambda}$. 
\smallbreak
Now, we introduce the problem
\begin{equation} \label{ls3}
\left \{
\begin{aligned}
\pLaplac u & = \lambda c(x) |\widetilde{T}_k(u)|^{p-2} \widetilde{T}_k(u) + \mu |\gradu|^p - h^{-}(x) -1, 
& \textup{ in } \Omega, \\
u & = 0 & \textup{ on } \partial \Omega,
\end{aligned}
\right.
\end{equation}
where
\begin{equation*}
\widetilde{T}_k(s) = \left\{
\begin{aligned}
 -k\,,  \qquad&\textup{ if } s \leq -k\,,\\
 s\,,  \qquad&\textup{ if } s > -k\,.
\end{aligned}
\right.
\end{equation*}
Observe that  $\beta_1$  and $0$ are upper solutions of \eqref{ls3}. Recalling that the minimum of two upper 
solution is an upper solution (see \cite[Corollary 3.3]{C_1997}), it follows that 
$\overline{\beta} = \min \{0,\beta_1\}$ is an upper solution of \eqref{ls3}. As 
$\alpha_{\lambda}$ is a lower solution of \eqref{ls3} with  $\alpha_{\lambda} \leq \overline{\beta}$, 
applying Theorem \ref{BMP1988},  we  conclude the existence of $\underline{u}_{\lambda}$ minimum solution of 
\eqref{ls3} with   $\alpha_{\lambda} \leq \underline{u}_{\lambda} \leq \overline{\beta} = \min \{0,\beta_1\}$. 
\smallbreak

As, for every upper solution $\beta$ of \eqref{Plambda}, $\beta$ is an upper solution of 
 \eqref{ls3},  we have $\alpha_{\lambda}\leq\beta$. Recalling that  
$\underline{u}_{\lambda}$ is the minimum solution of \eqref{ls3} with  
$\alpha_{\lambda}\leq \underline{u}_{\lambda} \leq 0$, we deduce that 
 $\underline{u}_{\lambda}\leq \beta$. 
\smallbreak
%Observe that  $\alpha$ and 
%$\beta$ are a couple of lower and upper solution of \eqref{ls3} with 
%$\alpha \leq \beta$. Moreover, observe that $\alpha$ and $0$ are another couple of lower 
%and upper solution with $\alpha \leq 0$. Consequently, by using that the minimum of two upper 
%solution is an upper solution (see \cite[Corollary 3.3]{C_1997}), it follows that $\alpha$ and 
%$\overline{\beta} = \min \{0,\beta\}$ are a couple of lower and upper solution with 
%$\alpha \leq \overline{\beta}$. By applying then Theorem \ref{BMP1988},  
%we can conclude that there exists $\underline{u}$ solution of \eqref{ls3} with   
%$\alpha \leq \underline{u} \leq \overline{\beta} = \min \{0,\beta\}$. \\

It remains to prove that $\underline{u}_{\lambda}$ is a lower solution of \eqref{Plambda}.
First, observe that $\underline{u}_{\lambda}$ is an upper solution of \eqref{ls1}. By construction, this implies that 
$\underline{u}_{\lambda} \geq -M_{\lambda} > -k$. Consequently, 
$\underline{u}_{\lambda}$ is a solution of 
\eqref{ls1} and so,  a lower solution of \eqref{Plambda}.  
%Moreover, since every upper solution of
 %\eqref{Plambda} is an upper solution of \eqref{ls1}, this lower solution satisfies 
 %$\underline{u}_{\lambda} \leq \min\{0,\beta\}$, for every $\beta$ upper solution of \eqref{Plambda}. 
 % Hence, we can conclude that $\underline{u}$ is a lower solution of \eqref{Plambda}, 
 %for every $\lambda \in [0,\Lambda]$. \\
\end{proof}

\section{The Functional setting}\label{TLSATAP}

% \textcolor{red}{Je n'ai rien modifi\'  dans cette section, sauf le dernier lemme, ici aussi en enlevant l'hypoth\`ese que $\lambda \geq 0$} 

% \medbreak

Let us introduce some auxiliary functions which are going to play an important role in the rest of the work. 
Define
\begin{equation}
\label{defg}
g(s) = \left \{
\begin{aligned}
& \Big|\frac{p-1}{\mu}\Big( 1+\frac{\mu}{p-1}s \Big) \ln \Big(1+ \frac{\mu}{p-1}s \Big)\Big|^{p-2}
\frac{p-1}{\mu}\Big( 1+\frac{\mu}{p-1}s \Big) \ln \Big(1+ \frac{\mu}{p-1}s \Big), \quad s > -\frac{p-1}{\mu}\,,
\\
& \ 0\,, \hspace{116mm} s \leq - \frac{p-1}{\mu}\,,
\end{aligned}
\right.\hspace{-1mm}
\end{equation}
\[ G(s) = \int_{0}^s g(t)\, dt\,
\qquad \mbox{ and } \qquad  
H(s) = \frac{1}{p}g(s)s - G(s)\,.\]
In the following lemma we prove some properties of these functions.

\begin{lemma} \label{gProperties} $ $
\begin{enumerate}
\item[i)] The function $g$ is continuous on $\R$, satisfies  $g > 0$ on $\R^{+}$ and there exists $D > 0$ 
with $-D \leq g \leq 0$ on $\R^{-}$. Moreover, $G \geq 0$ on $\R$.
\item[ii)] For any $\delta > 0$, there exists $\overline{c} = \overline{c}(\delta,\mu,p) > 0$ such that, 
 for any $s > \frac{p-1}{\mu}$, 
$g(s) \leq \overline{c}\, s^{p-1+\delta}$.
\item[iii)] $\lim_{s\to +\infty}g(s)/s^{p-1} =+ \infty$ and $\lim_{s\to+\infty}G(s)/s^p = + \infty$.
\item[iv)] There exists $R > 0$ such that the function $H$ satisfies $H(s) \leq \big(\frac{s}{t}\big)^{p-1}H(t)$, 
for $R \leq s \leq t$.
\item[v)] The function $H$ is bounded on $\R^{-}$.
\end{enumerate}
\end{lemma}

\begin{proof}
\noindent \textit{i)} By  definition, it is obvious that $g$ is continuous, $g > 0$ on $\R^{+}$ and 
$g$ is bounded and $g \leq 0$ on $\R^{-}$. This implies also that
$ G \geq 0$ by integration. 
\medbreak

\noindent \textit{ii)} First of all, recall that for any $\epsilon > 0$ there exists $c = c(\epsilon) > 0$ such that 
$\ln(s) \leq c(\epsilon)s^{\epsilon}$ for all $s \in (1,\infty)$. This implies that, for any $\delta>0$,
$$
\lim_{s\to+\infty} \frac{g(s)}{s^{p-1+\delta}}=
\lim_{s\to+\infty} \Big(\frac{(p-1)(1+\frac{\mu}{p-1}s)}{\mu s}\Big)^{p-1} 
\frac{\big(\ln(1+\frac{\mu}{p-1}s)\big)^{p-1}}{s^{\delta}}=0\,.
$$
Hence, there exists $R>\frac{p-1}{\mu}$ such that, for all $s>R$,
$$
\frac{g(s)}{s^{p-1+\delta}}\leq 1\,.
$$
As the function $ \frac{g(s)}{s^{p-1+\delta}}$ is continuous on the compact set $[\frac{p-1}{\mu},R]$, 
we have a constant $C>0$ with 
$$
\frac{g(s)}{s^{p-1+\delta}}\leq C \quad \mbox{ on }  [\tfrac{p-1}{\mu},R]\,.
$$
The result follows for $\overline C = \max (C,1)$.
\medbreak
%Consequently, for any $\delta > 0$, there exists $C = C(\delta) > 0$ such that for any $s > 0$
%\begin{equation} \label{gProp1}
%g(s) \leq C \left( \frac{p-1}{\mu} \right)^{p-1} \left(1+\frac{\mu}{p-1}s \right)^{p-1+\delta}.
%\end{equation}
%On the other hand, it is well known that, for any $a, b \geq 0$,
%\begin{equation*}
%a^p + b^p \leq (a+b)^p \leq 2^{p-1}(a^p+b^p)
%\end{equation*}
%if $p \geq 1$ and the reverse inequalities hold for $0 < p \leq 1$ (see for instance \cite[Theorem 1]{J_2014}). 
%Hence, we can claim that, for any $p > 0$, there exists $c_{p} > 0$ such that
%\begin{equation} \label{ineqBinomio}
%(a+b)^p \leq c_p (a^p + b^p)
%\end{equation}
%Consequently, since $p-1+\delta > 0$, there exists $c_{p,\delta} > 0$ such that, for any $s \geq 0$,
%\begin{equation} \label{gProp2}
%\left(1+\frac{\mu}{p-1}s \right)^{p-1+\delta} \leq c_{p,\delta}
%\left(1 + \left( \frac{\mu}{p-1} \right)^{p-1+\delta} s^{p-1+\delta} \right),
%\end{equation}
%By gathering together \eqref{gProp1} and \eqref{gProp2}, for any $s > 0$, it follows that
%\begin{equation*}
%\begin{aligned}
%g(s) \leq c_{p,\delta} \cdot C \left( \left(\frac{p-1}{\mu} \right)^{p-1} 
%+ \left( \frac{\mu}{p-1} \right)^{\delta} s^{p-1+\delta} \right).
%\end{aligned}
%\end{equation*}
%Finally, since we suppose that $s > \frac{p-1}{\mu}$, we can conclude that, there exists 
%$\overline{c} = \overline{c}(p,\delta, \mu) > 0$ such that
%\[ g(s) \leq \overline{c}s^{p-1+\delta}.\]

\noindent \textit{iii)} 
As
$$
\lim_{s\to+\infty}
\frac{\frac{p-1}{\mu}\big(1+\frac{\mu}{p-1}s\big) \ln\big(1+\frac{\mu}{p-1}s\big) }{s}=+\infty\,,
$$
and $p>1$, we easily deduce that 
$$
\lim_{s\to+\infty}\frac{g(s)}{s^{p-1}} =+\infty
$$
and, by L'Hospital's rule
$$
\lim_{s\to+\infty}\frac{G(s)}{s^{p}} =+\infty\,.
$$
%First of all, observe that, for any $s > 0$,
%\begin{equation*}
%\begin{aligned}
%\frac{g(s)}{s^{p-1}} & 
%= 
%\frac{\left(\frac{p-1}{\mu}\left(1 + \frac{\mu}{p-1} s\right)
%\ln\left(1 + \frac{\mu}{p-1} s\right)\right)^{p-1}}{s^{p-1}} 
%= 
%\frac{\left(\left(1 + \frac{\mu}{p-1} s\right)\ln\left(1 + \frac{\mu}{p-1} s\right)\right)^{p-1}}
%{\left(\frac{\mu}{p-1}s\right)^{p-1}} \\
%& \geq  \frac{\left(1 + \frac{\mu}{p-1} s\right)^{p-1} \left(\ln\left(1 + \frac{\mu}{p-1} s \right)\right)^{p-1}}
%{\left(1+\frac{\mu}{p-1}s\right)^{p-1}} = \left(\ln\left(1 + \frac{\mu}{p-1} s \right)\right)^{p-1}.
%\end{aligned}
%\end{equation*}
%Consequently, since $\ln(s) \xrightarrow{s \rightarrow \infty} \infty$, it follows that
%\[ \lim_{s \rightarrow \infty} \frac{g(s)}{s^{p-1}} = + \infty\,.\]
%Now, by using L'H\^opital, observe that
%\begin{equation*}
%\begin{aligned}
%\lim_{s \rightarrow \infty} \frac{G(s)}{s^p} = \frac{1}{p}\lim_{s \rightarrow \infty} \frac{g(s)}{s^{p-1}} 
%= + \infty\,.
%\end{aligned}
%\end{equation*}

\noindent \textit{iv)} First of all, integrating by parts, we observe that, for any $s \geq 0\,,$
\begin{equation*}
%\begin{aligned}
G(s)  = %&
% \left( \frac{p-1}{\mu} \right)^{p-1} \int_0^s \left( \left(1+\frac{\mu}{p-1} t \right)
%\ln\left(1+\frac{\mu}{p-1} t \right) \right)^{p-1} dt \\
 %= 
%  \left( \frac{p-1}{\mu} \right)^{p} \left[ \frac{1}{p}\left( \left(1+\frac{\mu}{p-1} s \right)
%\ln\left(1+\frac{\mu}{p-1} s \right) \right)^{p-1}\left(1+\frac{\mu}{p-1} s \right) \right. \\
%& \left. - \frac{\mu}{p} \int_0^s \left(1+\frac{\mu}{p-1} t \right)^{p-1} \left(\ln\left(1+\frac{\mu}{p-1} t \right)
%\right)^{p-2} dt \right],
\Big( \frac{p-1}{\mu} \Big)^{p} \Big[ \frac{1}{p} \Big(1+\frac{\mu}{p-1} s \Big)^p 
\Big(\ln\big(1+\frac{\mu}{p-1} s \big)\Big)^{p-1}
 %\right.
% \\
% &\hspace{70mm}
%\left. 
 - 
 \frac{\mu}{p} \int_0^s \Big(1+\frac{\mu}{p-1} t \Big)^{p-1} 
 \Big(\ln\big(1+\frac{\mu}{p-1} t \big)\Big)^{p-2} dt \Big],
%\end{aligned}
\end{equation*} 
and so, for any $s \geq 0$, it follows that
\begin{equation*}
%\begin{aligned}
H(s) = %& 
%= \frac{1}{p} g(s)s - G(s) \\
%& = \frac{1}{p} \left(\frac{p-1}{\mu} \right)^p \frac{\mu}{p-1} s\left( \left(1+\frac{\mu}{p-1} s \right)
%\ln\left(1+\frac{\mu}{p-1} s \right) \right)^{p-1} \\
%&  \qquad - \left( \frac{p-1}{\mu} \right)^{p}  \frac{1}{p}\left( \left(1+\frac{\mu}{p-1} s \right)
%\ln\left(1+\frac{\mu}{p-1} s \right) \right)^{p-1}\left(1+\frac{\mu}{p-1} s \right) \\
%&  \qquad  + \left( \frac{p-1}{\mu} \right)^{p}  \frac{\mu}{p} \int_0^s \left(1+\frac{\mu}{p-1} t \right)^{p-1} 
%\left(\ln\left(1+\frac{\mu}{p-1} t \right)\right)^{p-2} dt \\
 \frac{1}{p} \Big( \frac{p-1}{\mu} \Big)^{p} \Big[ \mu \int_0^s \Big(1+\frac{\mu}{p-1} t \Big)^{p-1} 
 \Big(\ln\big(1+\frac{\mu}{p-1} t \big)\Big)^{p-2} dt 
%\\
%&
%\hspace{70mm}
 %\qquad 
 - \Big(1+\frac{\mu}{p-1} s \Big)^{p-1}\Big(\ln\big(1+\frac{\mu}{p-1} s \big) \Big)^{p-1} \Big].
%\end{aligned}
\end{equation*}
To prove \textit{iv)}, we show that the function 
$\varphi(s):=  \frac{H(s)}{s^{p-1}}$
is non-decreasing on $[R,+\infty)$ for some  $R>0$. Observe that
$$
\varphi'(s) = \frac{1}{s^p}[ H'(s)s-(p-1)H(s)]\,.
$$
Hence,  we just need to prove that $H'(s)s-(p-1)H(s) \geq 0$ for $s\geq R$. 
After some simple computations, we see that it is enough to prove the existence of $R>0$ such that, 
for all $s\geq R$, $\kappa(s)\geq 0$ where
$$
\begin{aligned}
\kappa(s)=&
\Big(1 + \frac{\mu}{p-1} s \Big)^{p-2}
\Big(\ln\big(1 + \frac{\mu}{p-1} s \big) \Big)^{p-2} 
\Big( \big( \frac{\mu s}{p-1} \big)^2 + \ln\big(1 + \frac{\mu}{p-1} s \big) \Big) 
%\right. \\
%& \left. \quad 
\\
&\hspace{70mm}
- \mu \int_0^s \Big(1+\frac{\mu}{p-1} t \Big)^{p-1} 
\Big(\ln\big(1+\frac{\mu}{p-1} t \big)\Big)^{p-2} dt \,.
\end{aligned}
$$
%So, we should prove that $\varphi '(s) > 0$ for $s \geq R$. Directly, observe that for any $s > 0$,
%\begin{equation*}
%\begin{aligned}
%\varphi '(s) & = \frac{H'(s)s^{p-1} - (p-1)s^{p-2}H(s)}{s^{2(p-1)}} 
%= \frac{H'(s)}{s^{p-1}} - (p-1)\frac{H(s)}{s^p} = \frac{1}{s^p}\Big[ H'(s)s-(p-1)H(s) \Big]\\
%& = \frac{\mu}{ps^p} \left( \frac{p-1}{\mu} \right)^{p} \left[ \left[ \left(1 + \frac{\mu}{p-1} s \right)
%\ln\left(1 + \frac{\mu}{p-1} s \right) \right]^{p-2} \left( \frac{\mu s^2}{p-1} 
%- s\ln\left(1 + \frac{\mu}{p-1} s \right) \right) \right. \\
%& \quad + \frac{p-1}{\mu} \left[ \left(1 + \frac{\mu}{p-1} s \right)
%\ln\left(1 + \frac{\mu}{p-1} s \right) \right]^{p-1} \\
%& \left. \quad - (p-1) \int_0^s \left(1+\frac{\mu}{p-1} t \right)^{p-1} 
%\left(\ln\left(1+\frac{\mu}{p-1} t \right)\right)^{p-2} dt \right] \\
%& = \frac{\mu}{ps^p} \left( \frac{p-1}{\mu} \right)^{p+1} \left[ \left[ \left(1 + \frac{\mu}{p-1} s \right)
%\ln\left(1 + \frac{\mu}{p-1} s \right) \right]^{p-2} 
%\left( \left( \frac{\mu s}{p-1} \right)^2 + \ln\left(1 + \frac{\mu}{p-1} s \right) \right) \right. \\
%& \left. \quad - \mu \int_0^s \left(1+\frac{\mu}{p-1} t \right)^{p-1} 
%\left(\ln\left(1+\frac{\mu}{p-1} t \right)\right)^{p-2} dt \right] \\
%& = \frac{\mu}{ps^p} \left( \frac{p-1}{\mu} \right)^{p+1} \kappa(s)\,.
%\end{aligned}
%\end{equation*}
%So, we have to show that there exists $R > 0$ such that, for any $s \geq R$, $\kappa(s) > 0$. 
Observe that 
%$\kappa(0) = 0$ and
\begin{equation*}
\begin{aligned}
\kappa'(s) & = 
\frac{\mu}{p-1} \Big(1 + \frac{\mu}{p-1} s \Big)^{p-3} \Big(\ln \big(1 + \frac{\mu}{p-1} s \big) \Big)^{p-3} 
\\
&\hspace{50mm}
\Big[ (p-2)\Big(\frac{\mu s} {p-1} - \ln\big(1 + \frac{\mu}{p-1} s \big) \Big)^{2}
%\\
%& \qquad\  
+ \Big(\frac{\mu s}{p-1}\Big)^2 \ln\big(1 + \frac{\mu}{p-1} s \big) \Big] .
\end{aligned}
\end{equation*} 
%\begin{equation*}
%\begin{aligned}
%\kappa'(s) & = \left[\left(1+\frac{\mu}{p-1}s \right)^{p-3} \ln\left(1+\frac{\mu}{p-1}s \right)^{p-2}
%+ \ln\left(1+\frac{\mu}{p-1}s \right)^{p-3}\left(1+\frac{\mu}{p-1}s \right)^{p-3}\right] \\
%& \qquad\ \times \frac{\mu(p-2)}{p-1} \left( \left(\frac{\mu s}{p-1} \right)^{2}
%+ \ln\left(1+\frac{\mu}{p-1}s \right) \right) \\
%& \qquad\ + \frac{\mu}{p-1} \left[ \frac{2 \mu s}{p-1} + \frac{1}{1+\frac{\mu}{p-1}s} \right] 
%\left[ \left(1 + \frac{\mu}{p-1} s \right)\ln\left(1 + \frac{\mu}{p-1} s \right) \right]^{p-2} \\
%& \qquad\ - \mu \left(1 + \frac{\mu}{p-1} s \right)^{p-1} \ln\left(1 + \frac{\mu}{p-1} s \right)^{p-2} \\
%& = \frac{\mu}{p-1} \left[ \left(1 + \frac{\mu}{p-1} s \right) \ln \left(1 + \frac{\mu}{p-1} s \right) \right]^{p-3} 
%\Bigg[ (p-2) \left(\frac{\mu s}{p-1} \right)^2 \ln\left(1 + \frac{\mu}{p-1} s \right) \\
%& \qquad\ +(p-2) \left(\ln\left(1 + \frac{\mu}{p-1} s \right)\right)^2 + (p-2) \left(\frac{\mu s}{p-1} \right)^2 
%+ (p-1) \ln\left(1+\frac{\mu}{p-1}s \right) \\
%& \qquad\ +\frac{2\mu s}{p-1}\left(1+\frac{\mu}{p-1}s \right)\ln\left(1+\frac{\mu}{p-1}s \right) 
%- (p-1)\left(1+\frac{\mu}{p-1}s \right)^2 \ln\left(1+\frac{\mu}{p-1}s \right) \Bigg] \\
%& = \frac{\mu}{p-1} \left[ \left(1 + \frac{\mu}{p-1} s \right) \ln \left(1 + \frac{\mu}{p-1} s \right) \right]^{p-3} 
%\Bigg[ (p-2)\left(\frac{\mu s} {p-1} - \ln\left(1 + \frac{\mu}{p-1} s \right) \right)^{2}\\
%& \qquad\  + \left(\frac{\mu s}{p-1}\right)^2 \ln\left(1 + \frac{\mu}{p-1} s \right) \Bigg] .
%\end{aligned}
%\end{equation*} 
Hence, we distinguish two cases:
\begin{enumerate}
\item[i)] In case  $p \geq 2$, it is obvious that $\kappa'(s) > 0$, for any $s > 0$. This implies that $\kappa$ is 
increasing and so, that $\kappa(s) > 0$ for $s > 0$, since $\kappa(0) = 0$.
\item[ii)] If $ 1 < p < 2$,
 as $\lim_{s\to\infty} \kappa'(s)=+\infty$,  there exists $R_1 > 0$ such that, for any $s \geq R_1$, we have 
 $\kappa'(s) > 1$ and hence, there exists  $R_2 \geq R_1$ such that $\kappa(s) > 0$, for any $s \geq R_2$. 
\end{enumerate}
In any case, we can conclude the existence of  $R \geq  0$ such that $\kappa(s) > 0$ for any $s \geq R$.
 Consequently, there exists $R > 0$ such that $\varphi'(s) > 0$, for $s \geq R$, which means that $\varphi$ 
 is non-decreasing for $s \geq R$ and hence  $H$ satisfies $H(s) \leq \big(\frac{s}{t}\big)^{p-1}H(t)$, 
 for $R \leq s \leq t$.
\medbreak

\noindent \textit{v)} This follows directly from the definition of the functions $g$ and $G$. 
\end{proof}

Next, we define the function 
\begin{equation} \label{alpha}
 \alpha_{\lambda} = \frac{p-1}{\mu} \Big( e^{\frac{\mu}{p-1}\underline{u}_{\lambda}} -1 \Big) 
 \in \Wop \cap L^{\infty}(\Omega)\, ,
\end{equation}
where $\underline{u}_{\lambda} \in \Wop \cap \Linfty$ is the lower solution of \eqref{Plambda} obtained in 
 Proposition \ref{lowerSol}. Before going further, since $\underline{u}_{\lambda} \leq 0$, observe that 
 $0 \geq \alpha_{\lambda} \geq -\frac{p-1}{\mu}+\epsilon$ for some $\epsilon>0$. 
\medbreak

Now, for any $\lambda \in \R$, let us consider the auxiliary problem
\[ \label{Qlambda} \tag{$Q_{\lambda}$}
\pLaplac v = f_{\lambda}(x,v)\,, \quad v \in \Wop\,,\]
where 
\begin{equation} \label{flambda} 
f_{\lambda}(x,s) = \left\{
\begin{aligned}
& \lambda c(x) g(s) + \Big( 1 + \frac{\mu}{p-1}s\Big)^{p-1} h(x)\,, \quad
&\mbox{if }s \geq \alpha_{\lambda}(x)\,,\\
& \lambda c(x) g(\alpha_{\lambda}(x)) + \Big( 1 + \frac{\mu}{p-1}\alpha_{\lambda}(x) \Big)^{p-1} h(x)\,, \quad
&\mbox{if } s \leq \alpha_{\lambda}(x)\,,
\end{aligned}
\right.
\end{equation} 
where $g$ is defined by \eqref{defg}.
In the following lemma, we prove some properties of  the solutions of \eqref{Qlambda}.

\begin{lemma} 
\label{linkProblem}
Assume that \eqref{A1} holds. Then, it follows that:
\begin{enumerate}
\item[i)] Every solution of \eqref{Qlambda} belongs to $L^{\infty}(\Omega)$.
\item[ii)] Every solution $v$ of \eqref{Qlambda} satisfies $v \geq \alpha_{\lambda}$.
\item[iii)] A function $v \in \Wop$ is a solution of \eqref{Qlambda} if, and only if, the function
\[ u = \frac{p-1}{\mu} \ln \big( 1 + \frac{\mu}{p-1}v \big) \in \Wop \cap L^{\infty}(\Omega)\]
is a solution of \eqref{Plambda}.
\end{enumerate}
\end{lemma}

\begin{proof}
\textit{i)} This follows directly from \cite[Theorem IV-7.1]{L_U_1968}. 
\medbreak

\noindent \textit{ii)} First of all, observe that $\alpha_{\lambda}$ is a lower solution of \eqref{Qlambda}. 
For a solution  $v \in \Wop$   of \eqref{Qlambda}, we have 
$v \in \Wop \cap L^{\infty}(\Omega)$ by the previous step and, 
for all $\varphi \in \Wop \cap L^{\infty}(\Omega)$ with $\varphi \geq 0$,
\[ \int_{\Omega} \big[ |\gradv|^{p-2} \gradv - |\nabla \alpha_{\lambda}|^{p-2} \nabla \alpha_{\lambda} \big] 
\nabla \varphi\, dx 
\geq \int_{\Omega} \big[ f_{\lambda}(x,v) - f_{\lambda}(x,\alpha_{\lambda}) \big] \varphi\, dx.\]
Now, since there exist constants $d_1$,  $d_2 > 0$ such that for all $\xi$, $\eta \in \RN$, 
\begin{equation} \label{peralScalarProduct}
\langle |\xi|^{p-2}\xi - |\eta|^{p-2}\eta \,, \xi - \eta \rangle \geq 
\left\{
\begin{aligned}
& d_1 (|\xi|+|\eta|)^{p-2}|\xi-\eta|^2\,, \textup{ if } 1<p<2\,,\\
& d_2 |\xi-\eta|^p\,, \textup{ if } p \geq 2\,,
\end{aligned}
\right.
\end{equation}
(see for instance \cite[Lemma A.0.5]{P_1997}), we choose $\varphi = (\alpha_{\lambda} - v)^{+}$ and obtain that
\begin{equation*}
%\begin{aligned}
0 %&
 \geq
% - \int_{ \{\alpha_{\lambda} \geq v\} } \left[ |\gradv|^{p-2} \gradv - |\nabla \alpha_{\lambda}|^{p-2} 
%\nabla \alpha_{\lambda} \right] \nabla (v - \alpha_{\lambda})\, dx  \\
%& = 
\int_{ \{\alpha_{\lambda} \geq v\} } \big[ |\gradv|^{p-2} \gradv 
- |\nabla \alpha_{\lambda}|^{p-2} \nabla \alpha_{\lambda} \big] \nabla (\alpha_{\lambda} - v)\, dx 
%\\
%& 
\geq \int_{ \{\alpha_{\lambda} \geq v\} } \big[ f_{\lambda}(x,v) 
- f_{\lambda}(x,\alpha_{\lambda}) \big] (\alpha_{\lambda} - v)\, dx = 0\,.
%\end{aligned}
\end{equation*}
Consequently, using again \eqref{peralScalarProduct}, we deduce that $\alpha_{\lambda} = v$ in 
$\{\alpha_{\lambda} \geq v\} $ and so, that $v \geq \alpha_{\lambda}$.

\medbreak

\noindent \textit{iii)} Suppose that $v \in \Wop$ is a solution of \eqref{Qlambda}. The first parts, $i)$, $ii)$ 
imply that 
$v \in \Wop \cap L^{\infty}(\Omega)$ is such that 
$v \geq \alpha_{\lambda} \geq  -\frac{p-1}{\mu}+\epsilon$ with $\epsilon>0$
%Then, for any $\varphi \in \Wop$,  %the following identity holds
%\begin{equation}
%\label{Lem3.4.1}
%\begin{aligned}
%\int_{\Omega} |\gradv |^{p-2} \nabla v \, \nabla \varphi\ dx 
%= 
%\int_{\Omega} \Big[ \lambda c(x) g(v) + \big(1+\frac{\mu}{p-1}v\big)^{p-1} h(x) \Big] \varphi\, dx
%\end{aligned}
%\end{equation}
and hence
%Moreover, $v \in \Wop \cap L^{\infty}(\Omega)$ with $v\geq \alpha_{\lambda}>-\frac{p-1}{\mu}$ implies 
$u \in \Wop \cap L^{\infty}(\Omega)$. Let us prove that $u$ is a (weak) solution of \eqref{Plambda}.
Let $\phi$ be an arbitrary function belonging to $\mathcal{C}_0^{\infty}(\Omega)$ and define
$\varphi = \phi/(1+\frac{\mu}{p-1}v)^{p-1}$. It follows that $\varphi \in \Wop$. 
%Consequently, the above identity holds (for this $\varphi$). \\
As $e^{\frac{\mu u}{p-1}} = 1 + \frac{\mu}{p-1}v$, we have the following identity
\begin{equation*}
\begin{aligned}
\int_{\Omega} |\gradv |^{p-2} \gradv \nabla \varphi\ dx 
& 
= \int_{\Omega} e^{\mu u} |\gradu|^{p-2} \gradu \Big( \frac{\nabla \phi}{\big(1 + \muv\big)^{p-1}} 
- \frac{\mu\, \phi\, \gradv}{\big(1+\muv\big)^p} \Big) \,dx 
\\
& 
= \int_{\Omega}  \frac{e^{\mu u}}{\big(1 + \muv\big)^{p-1}} |\gradu|^{p-2} \gradu 
\Big( \nabla \phi 
- \dfrac{\mu \phi \nabla \big(\frac{p-1}{\mu}(e^{\frac{\mu}{p-1}u}-1)\big)}{1+ \muv} \Big) \,dx 
\\
& =  \int_{\Omega}  |\gradu|^{p-2} \gradu \big( \nabla \phi - \mu \phi \gradu \big) \,dx 
= \int_{\Omega}|\gradu|^{p-2} \gradu \nabla \phi\, dx - \mu \int_{\Omega}  |\gradu|^p \phi\,  dx.
\end{aligned}
\end{equation*}
On the other hand, by definition of $g$,  observe that
\begin{equation*}
\begin{aligned}
\int_{\Omega} \Big[  \lambda  c(x) g(v) +& \big(1+\frac{\mu}{p-1}v\big)^{p-1} h(x) \Big] \varphi\, dx \\
& = \int_{\Omega}  \Big[ \lambda c(x) \Bigl|\frac{p-1}{\mu} \ln\bigl(1+\frac{\mu}{p-1}v \bigr)\Bigr|^{p-2}
\Big(\frac{p-1}{\mu} \ln \bigl(1+\frac{\mu}{p-1}v \bigr)\Big) +h(x) \Big] \phi\, dx \\
& = \int_{\Omega} \bigl[ \lambda c(x) |u|^{p-2}u + h(x) \bigr] \phi\, dx.
\end{aligned}
\end{equation*}
As $v$ is a solution of  \eqref{Qlambda}
%satisfies \eqref{Lem3.4.1}, 
we deduce from these two identities that
\[ \int_{\Omega}|\gradu|^{p-2} \gradu \nabla \phi\ dx = \int_{\Omega} 
 \bigl[ \lambda c(x) |u|^{p-2}u +  \mu  |\gradu|^p +h(x) \bigr] \phi\,  dx,\]
and so, $u$ is a solution of \eqref{Plambda}, as desired. 
\medbreak

In the same way,   assume that $u \in \Wop \cap \Linfty$ is a solution of \eqref{Plambda}. 
By Proposition \ref{lowerSol} we know that $u \geq\underline{u}_{\lambda}$. 
Hence, it follows that $ v = \frac{p-1}{\mu} \big( e^{\frac{\mu u}{p-1}} - 1 \big)  \in \Wop \cap \Linfty$ 
and satisfies $v \geq \alpha_{\lambda} \geq  -\frac{p-1}{\mu}+\epsilon$ for some $\epsilon>0$. 
Arguing exactly as before, we deduce that $v$ is a solution of \eqref{Qlambda}.
\end{proof}

\begin{remark} \label{linkLowerUpper}
Arguing exactly as in the proof of Lemma \ref{linkProblem}, $iii)$, we can show that $v_1  \in H^1(\Omega) \cap \Linfty$ (respectively $v_2 \in  H^1(\Omega) \cap \Linfty$) is a lower solution (respectively an upper solution) of \eqref{Qlambda} if, and only if, the function
\[ u_1 = \frac{p-1}{\mu} \ln\big(1+\frac{\mu}{p-1}v_1 \big) \qquad \qquad \Big( \textup{respectively } u_2 = \frac{p-1}{\mu} \ln\big(1+\frac{\mu}{p-1}v_2 \big) \Big)\]
is a lower solution (respectively an upper solution) of \eqref{Plambda}.
\end{remark}

The interest of problem \eqref{Qlambda} comes from the fact that it has a variational formulation. 
We can obtain the solutions of \eqref{Qlambda} as critical points of the functional 
$I_{\lambda}: \Wop \rightarrow \R$ defined as
\begin{equation}\label{defI}
 I_{\lambda}(v) = \frac{1}{p} \int_{\Omega} |\gradv|^p \,dx- \int_{\Omega} F_{\lambda}(x,v) dx\,,
\end{equation} 
where we define $F_{\lambda}(x,s)=\int_0^s f_{\lambda}(x,t)\,dt$ i.e.
\begin{equation} \label{defFLambdaPos}
F_{\lambda}(x,s) = \lambda c(x)G(s) + \frac{p-1}{\mu p} \Big( 1 + \frac{\mu}{p-1} s \Big)^p h(x), 
\quad\textup{ if } s \geq \alpha_{\lambda}(x)\,,
\end{equation}
and
\begin{equation} \label{defFLambdaNeg}
\begin{aligned}
F_{\lambda}(x,s) =  \Big[\lambda c(x)g(\alpha_{\lambda}(x)) 
& 
+ \big(1 + \frac{\mu}{p-1}\alpha_{\lambda}(x)\big)^{p-1}h(x) \Big](s-\alpha_{\lambda})  \\
& 
+ \lambda c(x)G(\alpha_{\lambda}(x)) 
+ \frac{p-1}{\mu p} \big( 1 + \frac{\mu}{p-1} \alpha_{\lambda}(x) \big)^p h(x), 
\quad \textup{ if } s\leq \alpha_{\lambda}(x)\,.
\end{aligned}
\end{equation}
Observe that under the assumptions \eqref{A1}, since $g$ has subcritical growth 
(see Lemma \ref{gProperties}), $I \in \mathcal{C}^1(\Wop,\R)$ (see for example \cite{D_J_M_2001} page 356).

\begin{lemma} \label{strongConvergenceCerami}
Assume that \eqref{A1} holds and let $\lambda \in \R$ be arbitrary,  Then, any bounded Cerami sequence for 
$I_{\lambda}$ admits a convergent subsequence.
% at any level $d \in \R$.
\end{lemma}

\begin{proof}
Let $\{v_n\} \subset \Wop$ be a bounded Cerami sequence for $I_{\lambda}$ at level $d \in \R$. 
We are going to show that, up to a subsequence, $v_n \rightarrow v \in \Wop$ for a  $ v \in \Wop$.
\medbreak

Since $\{v_n\}$ is a bounded sequence in $\Wop$, up to a subsequence, 
we can assume that $v_n \rightharpoonup v$ in $\Wop$, $v_n \rightarrow v$ in 
$L^r(\Omega)$, for $1 \leq r < \past$, and $v_n \rightarrow v$ a.e. in $\Omega$. First of all, recall that
$\langle I_{\lambda}'(v_n), v_n-v \rangle \to 0$ with
\begin{equation*} 
\begin{aligned}
\langle I_{\lambda}'(v_n), v_n-v \rangle = 
\int_{\Omega} &|\nabla  v_n|^{p-2} \nabla v_n \nabla (v_n-v)\, dx 
 - \int_{\{v_n \geq \alpha_{\lambda}\}} \lambda c(x)g(v_n)(v_n-v) \, dx  
 \\
& - \int_{\{v_n \geq \alpha_{\lambda}\}} \big(1+\frac{\mu}{p-1}v_n\big)^{p-1} (v_n-v)\, h(x)\, dx 
- \int_{\{v_n \leq \alpha_{\lambda}\}} f_{\lambda}(x,\alpha_{\lambda}(x))(v_n-v)\, dx .
%\to 0\\
%& 
%- \int_{\{v_n \leq \alpha_{\lambda}\}} \Big[ \lambda c(x) g(\alpha_{\lambda}) 
%+ \big(1+\frac{\mu}{p-1}\alpha_{\lambda}\big)^{p-1} h(x) \Big] (v_n-v)\, dx \to 0\,.
\end{aligned}
\end{equation*}
Let $0<\delta<(\frac{p}{N}-\frac1q)p^*$, $r<p^*$ and $s<\frac{p^*}{p-1+\delta}$ such that 
$\frac1q+\frac1r+\frac1s=1$.
Using  Lemma \ref{gProperties} $ii)$, and the Sobolev embedding as well as H\"older inequality, we have that
\begin{equation*}
\begin{aligned}
 \Big| \lambda \int_{\{v_n \geq \alpha_{\lambda}\}} c(x)g(v_n)(v_n-v)\, dx \Big| 
 &\leq 
 %\left| \lambda \int_{\Omega} c(x)g(v_n)(v_n-v) dx \right| \\
%& \leq 
|\lambda| \int_{\Omega} |c(x)||g(v_n)||v_n-v|\, dx 
\leq |\lambda| \|c\|_q \|g(v_n)\|_s \|v_n - v\|_r 
\\
& \leq D|\lambda| \|c\|_q \big( 1+\|v_n\|_{(p-1+\delta)s}^{p-1+\delta} \big) \|v_n-v\|_r \\
& \leq DS |\lambda| \|c\|_q  \big( 1+\|v_n\|^{p-1+\delta} \big) \|v_n-v\|_r\,.
%\\
%& = \lambda D \|c\|_q \left[ \|v_n-v\|_{\frac{q}{q-1}} 
%+ \left( \int_{\Omega} |v_n|^{\frac{(p-1+r)q}{q-1}}|v_n - v |^{\frac{q}{q-1}} dx \right)^{\frac{q-1}{q}} \right]
\end{aligned}
\end{equation*}
%for every $r \in (0,1)$. Now, by choosing $r$ small enough we can ensure that there exists $s > 1$ such that
%\[ \frac{(p-1+r)q}{q-1}\ s < \past \qquad \textup{ and } \qquad \frac{qs}{(q-1)(s-1)} < \past,\]
%and so, by applying H\"older and Sobolev inequalities, it follows that
%\begin{equation*}
%\begin{aligned}
%\left( \int_{\Omega} |v_n|^{\frac{(p-1+r)q}{q-1}}|v_n - v |^{\frac{q}{q-1}} dx \right)^{\frac{q-1}{q}} 
%\leq \|v_n\|_{\frac{(p-1+r)q s}{q-1}}^{p-1+r} \|v_n-v\|_{\frac{qs}{(q-1)(s-1)}} 
%\leq \frac{1}{S_N} \|v_n\|^{p-1+r} \|v_n-v\|_{\frac{qs}{(q-1)(s-1)}}\,.
%\end{aligned}
%\end{equation*}
Since $\|v_n\|$ is bounded and $v_n \rightarrow v$ in $L^r(\Omega)$, for $1 \leq r < \past$, we obtain
%\[ \|v_n - v \|_{\frac{q}{q-1}} \to 0,\]
%and that
%\[ \frac{1}{S_N} \|v_n\|^{p-1+r} \|v_n-v\|_{\frac{qs}{(q-1)(s-1)}} \to 0.\]
%Hence, we can conclude that 
\[ \lambda \int_{\{v_n \geq \alpha_{\lambda}\}} c(x)g(v_n)(v_n-v) \,dx \to 0\,.\]
Arguing in the same way,  we have
\[ \int_{\{v_n \geq \alpha_{\lambda}\}} \big(1+\frac{\mu}{p-1}v_n\big)^{p-1} (v_n-v)\, h(x)\, dx 
%\to 0\,,\]
%and that
%\[
+ \int_{\{v_n \leq \alpha_{\lambda}\}}  f_{\lambda}(x,\alpha_{\lambda}(x))  (v_n-v)\, dx \to 0\,.\]
So, we deduce that
\begin{equation} \label{sc1}
\int_{\Omega} |\nabla v_n|^{p-2} \nabla v_n \nabla(v_n-v)\, dx \to 0\,.
\end{equation}
Hence, applying \cite[Theorem 10]{D_J_M_2001}, we conclude that $v_n \rightarrow v $ in $\Wop$, as desired.
%On the other hand, by using that $v_n \rightharpoonup v$ in $\Wop$, it follows that
%\begin{equation} \label{sc2}
%\int_{\Omega} |\gradv|^{p-2} \gradv \,\nabla(v_n-v)\, dx \to 0\,.
%\end{equation}
%By gathering together \eqref{sc1} and \eqref{sc2}, we deduce that
%\[ \int_{\Omega} \left( |\nabla v_n|^{p-2} \nabla v_n - |\gradv|^{p-2} \gradv \right)
%\left( \nabla v_n - \gradv \right)\ dx \to 0\,,\]
%and so, by applying \cite[Lemma 2.7]{L_M_1994}, {\color{red}Meilleure reference} we can conclude that 
%$v_n \to v \textup{ in } \Wop$, 
%as desired.
\end{proof}

\section{Sharp existence results on the limit coercive case} 
\label{Coercive-Case}

%\textcolor{red}{Attention il est indispensable de bien tout v\'erifier dans cette section il reste aussi des modifications \`a faire.} 

% \medbreak

In this section, following ideas from \cite[Section 3]{A_DC_J_T_2015},  we prove Theorem \ref{coercive-louis}.  As a preliminary step, considering $\mu >0$ constant, we introduce
\begin{equation*} 
m_{p,\lambda}:= \left\{
\begin{aligned}
& \inf_{u \in W_{\lambda}} \int_{\Omega} \Big(|\gradu|^p - \big(\frac{\mu}{p-1}\big)^{p-1}h(x)|u|^p\,\Big) dx\,, 
\quad & \textup{ if} \quad W_{\lambda} \neq \emptyset\,,
\\
& +\infty\,, & \textup{ if} \quad W_{\lambda} = \emptyset\,.
\end{aligned}
\right.
\end{equation*}
where
\[ W_{\lambda}:= \{ w \in \Wop: \lambda c(x)w(x) = 0 \textup{ a.e. } x \in \Omega\,,\ \|w\| = 1\}\, \]
and we define
\[ m:= \inf_{u \in \Wop} I_{\lambda}(u) \in \R \cup \{-\infty\}\,.\]

\begin{prop} \label{globalMinimum}
Assume that \eqref{A1} holds,  $\lambda \leq 0$ and that $m_{p,\lambda} > 0$. Then $m$ is finite and it is reached by a function 
$v \in \Wop$. Consequently the problem \eqref{Plambda} has a solution.
\end{prop}

\begin{proof}
To prove that $I_{\lambda}$ has a global minimum  since, by Lemma \ref{strongConvergenceCerami}, any bounded Cerami sequence has a convergent subsequence it suffices to show that $I_{\lambda}$ is coercive. Having found a global minimum $v \in \Wop$  we deduce, by  Lemma \ref{linkProblem},  that 
$u = \frac{p-1}{\mu} \ln \big( 1 + \frac{\mu}{p-1}v \big)$ % \in \Wop \cap L^{\infty}(\Omega)$ 
is a solution of \eqref{Plambda}.  
To show that $I_{\lambda}$ is coercive we consider an arbitrary sequence $\{v_n\} \subset \Wop$  such that $\|v_n\| \to \infty$ and we prove that
\[ \lim_{n \rightarrow \infty} I_{\lambda}(v_n) = +\infty\,.\]
\medbreak

Assume by contradiction that, along a subsequence, $I_{\lambda}(v_n)$ is bounded from above and 
hence
\begin{equation} \label{B7}
\limsup_{n \rightarrow \infty} \frac{I_{\lambda}(v_n)}{\|v_n\|^p} \leq 0\,.
\end{equation}
We introduce the sequence $w_n = \frac{v_n}{\|v_n\|}$, for all $n \in \N$ and observe that, up to a subsequence 
$w_n \rightharpoonup w$ weakly in $\Wop$, $w_n \rightarrow w$ in $L^r(\Omega)$, 
for $1 \leq r < \past\,,$ and $w_n \rightarrow w$ a.e. in $\Omega$. We consider two cases: 
\medbreak

\noindent\textit{Case 1): $w^+ \not \in W_{\lambda}$.} In that case, the set 
$\Omega_0 = \{ x \in \Omega: \lambda c(x)w^+(x) \neq 0\} \subset \Omega$ has non-zero measure and so, 
it follows that $v_n(x) = w_n(x)\,\|v_n\| \to \infty$ a.e. in $\Omega_0$. Hence, taking into account that 
$G \geq 0$ and  $\lim_{s\to+\infty}G(s)/s^p = + \infty$ (see Lemma \ref{gProperties}) 
and using Fatou's Lemma, we have
\begin{equation} \label{B8}
\begin{aligned}
\limsup_{n \rightarrow \infty} \int_{\Omega} \frac{\lambda c(x)G(v_n)}{|v_n|^p} |w_n|^p\, dx
& \leq \limsup_{n \rightarrow \infty} \int_{\Omega_0} \frac{\lambda c(x)G(v_n)}{|v_n|^p} |w_n|^p \,dx 
\\
&\leq \int_{\Omega_0} \limsup_{n \rightarrow \infty} \frac{\lambda c(x)G(v_n)}{|v_n|^p}|w_n|^p \, dx = -\infty\,.
\end{aligned}
\end{equation}
On the other hand, observe that for any $v \in \Wop$, we can rewrite
\begin{equation*}
\begin{aligned}
& I_{\lambda}(v) = \frac{1}{p} \int_{\Omega} |\gradv|^p \, dx - \int_{\Omega} \lambda c(x) G(v) \, dx 
+ \int_{\{v \leq \alpha_{\lambda}\}} \lambda c(x)G(v)\,  dx \\
& \hspace{58mm} -\frac{p-1}{p\mu} \int_{\{v \geq \alpha_{\lambda}\}} \Big(1+\frac{\mu}{p-1}v \Big)^p h(x)\, dx 
- \int_{\{v \leq \alpha_{\lambda}\}} F_{\lambda}(x,v)\, dx\,.
\end{aligned}
\end{equation*}
Hence, considering together \eqref{B7} and \eqref{B8}, we obtain
\[ 0 \geq \limsup_{n \rightarrow \infty} \frac{I_{\lambda}(v_n)}{\|v_n\|^p} 
\geq \liminf_{n \rightarrow \infty} \frac{I_{\lambda}(v_n)}{\|v_n\|^p}
 \geq -C - \limsup_{n \rightarrow \infty} \int_{\Omega} \frac{\lambda c(x)G(v_n)}{\|v_n\|^p} \, dx = + \infty\,,\]
and so,  \textit{Case 1)} cannot occur. 
\medbreak

\noindent\textit{Case 2): $w^+ \in W_{\lambda}$.} First of all, since $\lambda c \leq 0$ and $G \geq 0$
(see Lemma \ref{gProperties}), observe that for any $v \in \Wop$,
\begin{equation*}
\begin{aligned}
& I_{\lambda}(v) \geq \frac{1}{p}\int_{\Omega} \Big( |\gradv|^p 
- \big(\frac{\mu}{p-1} \big)^{p-1} h(x) (v^{+})^p \Big) \,dx 
- \frac{1}{p}\big(\frac{\mu}{p-1}\big)^{p-1} \int_{\{v \geq \alpha_{\lambda}\}}h(x) (v^{-})^p \, dx \\
& \hspace{40mm}- \frac{p-1}{\mu p} \int_{\{v \geq \alpha_{\lambda}\}} \Big[\big(1+\frac{\mu}{p-1}v\big)^{p}
-\big(\frac{\mu}{p-1}\big)^p |v|^p \Big]h(x)\,dx - \int_{\{v \leq \alpha_{\lambda}\}} F_{\lambda}(x,v) \,dx\,.
\end{aligned}
\end{equation*}
Moreover, observe that
\begin{equation} \label{ineqPminus1}
\begin{aligned}
\frac{1}{p}  \Big| \int_{\{v \geq \alpha_{\lambda}\}} \Big[ \big( 1 + \frac{\mu}{p-1} v \big)^p 
& - \big( \frac{\mu}{p-1} \big)^p |v|^p \Big] h(x) \,dx \Big| \\
& \hspace{20mm}=  \Big| \int_{\{v \geq \alpha_{\lambda}\}} \Big( \int_0^1 \big| s + \frac{\mu}{p-1} v \big|^{p-2} 
\big(  s + \frac{\mu}{p-1} v \big) \,ds \Big) h(x) \, dx \Big| 
\\
& \hspace{20mm}\leq 
\int_{\Omega} \Big( 1+\frac{\mu}{p-1} |v| \Big)^{p-1}
|h(x)|\, dx  \leq D \|h\|_q\, \big(1+ \|v\|^{p-1}\big)\,, 
\end{aligned}
\end{equation}
for some constant $D > 0$. Thus, for any $v \in \Wop$, it follows that
\begin{equation} \label{B13}
\begin{aligned}
 I_{\lambda}(v) \geq \frac{1}{p}\int_{\Omega} \Big( |\gradv|^p 
- \big(\frac{\mu}{p-1} \big)^{p-1} h(x) (v^{+})^p \Big) \,dx 
& - \frac{1}{p}\big(\frac{\mu}{p-1}\big)^{p-1} \int_{\{v \geq \alpha_{\lambda}\}}h(x) (v^{-})^p \, dx \\
&   - D \|h\|_q\, \big(1+\|v\|^{p-1}\big)- \int_{\{v \leq \alpha_{\lambda}\}} F_{\lambda}(x,v) \,dx\,.
\end{aligned}
\end{equation}
Hence, using that by the definition of $F_{\lambda}$ (see \eqref{defFLambdaPos} and \eqref{defFLambdaNeg}) there exists $m \in L^q(\Omega)\,,\ q > \max\{N/p,1\}$, such that, for a.e. $x \in \Omega$ and all $s \leq 0$,
\begin{equation} \label{Flambdaminus}
|F_{\lambda}(x,s)|\leq m(x) (1+|s|)\,,
\end{equation}
and applying \eqref{B7} and \eqref{B13},  we deduce, as $w^{+} \in W_{\lambda}$, that
\[ 
 0 \geq \limsup_{n \rightarrow \infty} 
 \frac{I_{\lambda}(v_n)}{\|v_n\|^p} \geq \liminf_{n \rightarrow \infty} \frac{I_{\lambda}(v_n)}{\|v_n\|^p} 
\geq 
\frac{1}{p} \int_{\Omega} \Big( |\nabla w|^p - \big(\frac{\mu}{p-1} \big)^{p-1} h(x) 
(w^{+})^p \Big) \,dx 
\geq 
\frac{1}{p} \min\{1, m_{p,\lambda}\} \|w\|^p \geq 0\,,\]
and so, that
%\textcolor{red}{Comme la preuve de la Proposition \ref{globalMinimum} est maintenant donn\'ee avant la Section \ref{OTCCATMPG} il faut ici donner des preuves (compl\`etes) de ce que l'on \'ecrit et ce d\'ebrouiller pour ne pas avoir \`a tout refaire dans la Section \ref{OTCCATMPG} ! En particulier il me semble que \`a plusieurs endroits dans le papier il faut estimer un terme du type $$ \int_{\{v \geq \alpha\}} \Big[\big(1+\frac{\mu}{p-1}v\big)^{p}
%-\big(\frac{\mu}{p-1}\big)^p |v|^p \Big]h(x)\,dx. $$
%R\'efl\'echi alors peut \^etre un peu s'il n'est pas possible de donner ``une seule preuve".}
\begin{equation*}
\lim_{n \rightarrow \infty} \frac{I_{\lambda}(v_n)}{\|v_n\|^p} = 0 \qquad \textup{ and } \qquad  w \equiv 0\,.
\end{equation*}
Finally, taking into account that $w_n \rightarrow 0$ in $L^r(\Omega)$, for $1 \leq r < \past$, 
we obtain the contradiction
\[ 0 = \lim_{n \rightarrow \infty} \frac{I_{\lambda}(v_n)}{\|v_n\|^p} \geq \frac{1}{p}\,.\]
Hence, \textit{Case 2)} cannot occur.
\end{proof}

\begin{proof}[\textbf{Proof of Theorem \ref{coercive-louis}}]
To prove this result, we look for a couple of lower and upper solutions $(\alpha, \beta)$ of \eqref{Plambda} 
with  $\alpha \leq \beta$ and then we apply Theorem \ref{BMP1988}. First, assume that both 
$\|\mu^{+}\|_{\infty} >0$ and $\|\mu^{-}\|_{\infty} >0$. Observe that any solution of
\begin{equation} \label{B11}
\pLaplac u = \lambda c(x)|u|^{p-2}u + \|\mu^{+}\|_{\infty}|\gradu|^p + h(x), \quad u \in \Wop \cap \Linfty,
\end{equation}
is an upper solution of \eqref{Plambda} and, any solution of
\begin{equation} \label{B12}
\pLaplac u = \lambda c(x)|u|^{p-2}u - \|\mu^{-}\|_{\infty}|\gradu|^p + h(x), \quad u \in \Wop \cap \Linfty,
\end{equation}
is a lower solution of \eqref{Plambda}. Now, since $m_{p,\lambda}^{+} > 0$,  Proposition  \ref{globalMinimum} 
ensures the existence of $\beta \in \Wop \cap \Linfty$ solution of \eqref{B11}. 
In the same way, $m_{p,\lambda}^{-} > 0$ implies the existence of 
$v\in \Wop \cap \Linfty$ solution of 
\begin{equation*}
\pLaplac v = \lambda c(x)|v|^{p-2}v + \|\mu^{-}\|_{\infty}|\gradv|^p - h(x), \quad v \in \Wop \cap \Linfty,
\end{equation*}
and hence $\alpha=-v$ is a solution of 
\eqref{B12}. Moreover, Lemma \ref{lemmaUniqueness} implies $\alpha$, 
$\beta \in \Wop \cap \mathcal{C}(\overline{\Omega})$. Hence, since $\alpha$ is a lower solution of \eqref{B11}, 
it follows that $\alpha \leq \beta$, thanks to Theorem \ref{compPrinciple}. Thus, we can apply Theorem 
\ref{BMP1988} to conclude the proof. Now note that if $\|\mu^{+}\|_{\infty} =0$,
\eqref{B11} reduces to 
\begin{equation} 
\pLaplac u = \lambda c(x)|u|^{p-2}u  + h(x), \quad u \in \Wop \cap \Linfty,
\end{equation}
which has a solution by \cite[Theorem 13]{D_J_M_2001}.
%choosing a $\mu^+ >0$ small enough a solution to
%\begin{equation} \label{B11}
%\pLaplac u = \lambda c(x)|u|^{p-2}u + \mu^+ |\gradu|^p + h(x), \quad u \in \Wop \cap \Linfty,
%\end{equation}
%exists and corresponds to an upper solution to $(P_{\lambda})$. Similarly, we can justify the existence of 
%the lower solution when  $\|\mu^{-}\|_{\infty} =0$.
This solution corresponds again to an upper solution to $(P_{\lambda})$. Similarly, we can justify the existence of the lower solution when  $\|\mu^{-}\|_{\infty} =0$.
\end{proof}

% \textcolor{red}{Attention j'ai fait des modifications dans la preuve du Th\'eor\`eme \ref{coercive-louis} aussi.} 

\section{A necessary and sufficient condition for the existence of a solution to $(P_0)$} 
\label{Weakly-coercive-Case}

%\textcolor{red}{Je n'ai rien modifi\'e dans cette section.}

%\medbreak
In this section we prove Theorem \ref{th3}. First of all, following the ideas of \cite{A_DC_J_T_2015}, inspired in turn in ideas of \cite{A_DA_P_2006}, 
we  find a necessary condition for the existence of a solution of \eqref{P0}.
Recall that the problem  \eqref{P0} is given by
% for the non-coercive case, namely $\lambda > 0$, as well as for $\lambda = 0$. 
%Moreover, as a consequence of the existence result for $\lambda = 0$, we will obtain existence result for the 
%coercive case. First of all, let us recall the boundary value problem.
\[ \label{P0} \tag{$P_0$} 
\pLaplac u = \mu |\gradu|^p + h(x)\, , \quad u \in \Wop \cap L^{\infty}(\Omega)\,.\]
%By following the ideas of \cite{A_DC_J_T_2015}, inspired in turn in ideas of \cite{A_DA_P_2006}, 
%we can find a necessary condition for the existence of a solution of \eqref{P0}.

\begin{prop} \label{necessaryCondition}
Assume that \eqref{A1} holds and suppose that \eqref{P0} has a solution. 
Then $m_p$ defined by \eqref{mp} satisfies 
$m_p > 0$.
\end{prop}

\begin{proof}
Assume that \eqref{P0} has a solution $u \in \Wop \cap \Linfty$. 
Then, for any $\phi \in \mathcal{C}_0^{\infty}(\Omega)$, it follows that
\begin{equation}\label{nc1}
\int_{\Omega} |\gradu|^{p-2} \gradu \nabla(|\phi|^p )\, dx - \mu \int_{\Omega} |\gradu|^p |\phi|^p\,dx 
- \int_{\Omega}h(x)|\phi|^p \,dx = 0\,.
\end{equation}
Now, applying Young's inequality, observe that
\begin{equation*}
\begin{aligned}
\int_{\Omega} |\gradu|^{p-2}\gradu \nabla (|\phi|^p)\, dx & 
= p \int_{\Omega}  |\phi|^{p-2}\phi  |\gradu|^{p-2}\gradu \nabla \phi \, dx 
\leq p \int_{\Omega} |\phi|^{p-1} |\gradu|^{p-1} | \nabla \phi| \, dx 
\\
& \leq \mu \int_{\Omega} |\phi|^p |\gradu|^p\, dx 
+ \left( \frac{p-1}{\mu} \right)^{p-1} \int_{\Omega} | \nabla \phi|^p \,dx\,.
\end{aligned}
\end{equation*}
Hence, substituting in \eqref{nc1}, multiplying by $\left( \frac{\mu}{p-1} \right)^{p-1}$ and using the density of  $\mathcal{C}_0^{\infty}(\Omega)$ in $ \Wop$, we obtain
%\[ \int_{\Omega}\Big( |\nabla \phi|^p dx - \big( \frac{\mu}{p-1} \big)^{p-1}  h(x) |\phi|^p \Big)\,dx \geq 0, 
%\qquad \forall\ \phi \in \mathcal{C}_0^{\infty}(\Omega)\,.\]
%Moreover, by density we deduce that
\begin{equation}\label{nc2}
\int_{\Omega} \Big(|\nabla \phi|^p  - \big( \frac{\mu}{p-1} \big)^{p-1}  h(x) |\phi|^p \Big)\,dx \geq 0\,, 
\quad \forall\ \phi \in \Wop\,.
\end{equation}
Arguing by contradiction, assume that
\[ 
\inf\Bigg\{ \int_{\Omega} \Big(|\nabla \phi|^p - \big( \frac{\mu}{p-1} \big)^{p-1}  h(x) |\phi|^p \Big)\,dx:
 \phi \in \Wop,\ \|\phi\| = 1 \Bigg\} = 0.
 \]
By standard arguments there exists $ \phi_0 \in \mathcal{C}^{0,\tau}(\overline{\Omega})$  for some $\tau \in (0,1)$, 
with $\phi_0 > 0$ in $\Omega$ such that
\begin{equation}\label{nc3}
\int_{\Omega} |\nabla \phi_0 |^p\, dx 
=  \Big( \frac{\mu}{p-1} \Big)^{p-1} \int_{\Omega} h(x) |\phi_0|^p \,dx\,.
\end{equation} 
Now, substituting the above identity in \eqref{nc1} with $\phi = \phi_0$,
% and multiplying by $\big( \frac{\mu}{p-1} \big)^{p-1}$, 
we have that
%\[ \int_{\Omega} p \phi_0^{p-1} |\gradu|^{p-2}\gradu \nabla \phi_0 - \mu |\gradu|^p \phi_0^p 
%- \left( \frac{p-1}{\mu} \right)^{p-1} \int_{\Omega} |\nabla \phi_0 |^p dx = 0\,,\]
\begin{equation}\label{nc4}
\int_{\Omega} \Big(|\nabla \phi_0|^p + (p-1) \big( \frac{\mu}{p-1} \big)^{p} \phi_0^p |\gradu|^p 
- p \big( \frac{\mu}{p-1} \big)^{p-1} \phi_0^{p-1} |\gradu|^{p-2}\gradu \nabla \phi_0\Big)\, dx = 0\,.
\end{equation}
Finally, observe that
\[ \frac{\mu}{p-1}\gradu =  \frac{1}{e^{\frac{\mu}{p-1}u}} \nabla e^{\frac{\mu}{p-1}u}.\]
Hence, by substituting in \eqref{nc4}, we deduce that
\begin{equation}  \label{nc5}
\int_{\Omega} \Big( |\nabla \phi_0|^p 
+ (p-1) \big( \frac{\phi_0}{e^{\frac{\mu}{p-1}u}} \big)^{p}  |\nabla e^{\frac{\mu}{p-1}u}|^p 
- p \big( \frac{\phi_0}{e^{\frac{\mu}{p-1}u}} \big)^{p-1} 
|\nabla e^{\frac{\mu}{p-1}u}|^{p-2}\nabla e^{\frac{\mu}{p-1}u} \nabla \phi_0\Big)\, dx = 0\,.
\end{equation}
%Now, let us recall that $u \in \Wop \cap \mathcal{C}(\overline{\Omega})$, thanks to the Lemma 
%\ref{lemmaUniqueness}. So, $e^{\frac{\mu}{p-1}u}$ and $\phi_0$ are positive and continuous function 
%differentiable a.e. in $\Omega$. 
Applying  Proposition \ref{picone}, this proves the existence of 
$k\in \mathbb R$ such that
%\[ \gradu = \frac{p-1}{\mu} \frac{\nabla \phi_0}{\phi_0}\,,\]
$$
\phi_0=k e^{\frac{\mu}{p-1}u}.
$$
%which implies (see \cite[Remark 3.2]{A_DC_J_T_2015}) that $u \notin \Wop$. This is a contradiction. 
%Consequently, we can conclude that $h^{+} \not\equiv 0$ implies that
%\[ \inf\left\{ \int_{\Omega} |\nabla \phi|^p dx - \left( \frac{\mu}{p-1} \right)^{p-1}  h(x) |\phi|^p dx:
%\phi \in \Wop,\ \|\phi\| = 1 \right\} > 0\,.\]
%In both cases, we deduce that if \eqref{P0} has a solution then $m_p > 0$. So, we can conclude that 
%$m_p > 0$ is a necessary condition.
As $\phi_0=0$ and $e^{\frac{\mu}{p-1}u}=1$ on $\partial\Omega$, this implies that $k=0$ which 
contradicts the fact that $\phi_0>0$ in $\Omega$.
\end{proof}

% \textcolor{red}{\textbf{We can avoid theses preliminaries and the proposition by using 
%Lemma \ref{geometrySmall1} and Proposition \ref{localMinimumSmall}. The proof is written. 
%We only have to find a suitable place to put it. }} \\

% Thanks to the Proposition \ref{minimizingSeq1}, 
%we have the existence of a sequence $\{v_n\} \subset B(0,r)$ such that
%$I_{\lambda}(v_n) \leq I_{\lambda}(u_n)$ for all $n$,
% $\lim_n \|v_n-u_n\| = 0$
%and $\lim_n I_{\lambda}'(u_n) = 0$.
%\end{itemize}
%Furthermore, by using that $\{v_n\}$ is bounded, we deduce that $\{v_n\}$ is a Cerami 
%sequence at the level $m$. 
%Moreover, under this assumptions the Lemma \ref{strongConvergenceCerami} holds. 
%Hence, $\{v_n\}$ satisfies the 
%Cerami condition at the level $m$ and so, the Palais-Smale condition at the same level. By applying then the 
%Proposition \ref{minimizingSeq2}, we deduce that $I_{\lambda}$ achieves its minimum in $B(0,r)$ and so, 
%we prove the existence of a local minimum.
%{\color{red}Antonio, ta demonstration ne me convainc pas et me semble bien compliquee. 
%n plus si la Proposition 2.11 n'est pas valable pour un minimum local ... A revoir. 
%Voir par exemple Jeanjean-Sirakov}
%\end{proof}

\begin{proof}[\textbf{Proof of  Theorem \ref{th3}}]
The proof is just the combination of Proposition \ref{necessaryCondition} and of Remark \ref{reverse}.
\end{proof}

\section{On the Cerami conditon and the Mountain-Pass Geometry} 
\label{OTCCATMPG}

% \begin{equation*}
%F_{\lambda}(x,v) = \left \{
%\begin{aligned}
%& \lambda c(x)G(v) + \frac{p-1}{\mu p} \left( 1 + \frac{\mu}{p-1} v \right)^p h(x), \quad 
%& v \geq \alpha_{\lambda}, \\
%& \left[\lambda c(x)g(\alpha_{\lambda}) + \left(1 + h(x) \frac{\mu}{p-1}\alpha_{\lambda}\right)^{p-1} 
%\right](v-\alpha_{\lambda}) + \lambda c(x)G(\alpha_{\lambda}) + h(x) \frac{p-1}{\mu p} 
%\left( 1 + \frac{\mu}{p-1} \alpha_{\lambda} \right)^p  , \quad & v \leq \alpha_{\lambda}.
%\end{aligned}
%\right.
%\end{equation*}

%Now, following \cite{J_RQ_2016}, we introduce
%\[  m_p^c := \inf \left \{ \int_{\Omega} 
%\Big( |\gradu|^p - \big( \frac{\mu}{p-1} \big)^{p-1}h(x)|u|^p \Big)\, dx: 
%u \in \Wop,
%\ u \geq 0,
%\  \|u\| = 1,\ cu \equiv 0 \right\}\,,\]
%{\color{red}Antonio, c'est la meme chose}
%and, by assuming that $m_p^c > 0$, 

% \textcolor{red}{J'ai r\'eorganis\' e un certain nombre de choses ici, attention en particulier \`a la nouvelle preuve de la Proposition \ref{localMinimumSmall}.}

\medbreak

We are going to show that, for any $\lambda > 0$,  the Cerami sequences for 
$I_{\lambda}$ at any level are bounded. The proof is inspired by
\cite{J_S_2013}, see also \cite{J_1999}. Nevertheless it requires to develop some new ideas. In view of Lemma \ref{strongConvergenceCerami}, this will imply that $I_{\lambda}$ satisfies the Cerami condition at any level $d \in \R$.

\begin{lemma} 
\label{boundednessCerami}
Fixed $\lambda > 0$ arbitrary, assume that \eqref{A1} holds and suppose that $m_p> 0$
 with $m_p$ defined by \eqref{mp}. Then, the Cerami 
sequences for $I_{\lambda}$ at any level $d \in \R$ are bounded.
\end{lemma}
 
%\textcolor{red}{J'ai modifi\'e le d\'ebut de la preuve (c'\'etait nettement trop d\'etaill\'e) et il reste aussi des modifications \`a faire.}

\begin{proof}
Let $\{v_n\} \subset \Wop$ be a Cerami sequence for $I_{\lambda}$ at level $d \in \R$. First we claim  that $\{v_n^{-}\}$ is bounded. Indeed since 
$\{v_n\}$ is a Cerami sequence, we have that
%for any $\varphi \in \Wop$, it follows that $\langle I_{\lambda}'(v_n), \varphi \rangle \rightarrow 0$. 
%Hence, by taking $\varphi = v_n^{-}$, it follows that
\begin{equation}\label{bc1}
\langle I_{\lambda}'(v_n), v_n^{-} \rangle = 
- 
\int_{\Omega}|\nabla v_n^{-}|^p dx 
 - \int_{\Omega} f_{\lambda} (x, v_n) \, v_n^{-} \, dx
\to 0\,
\end{equation}
from which, since $f_{\lambda} (x,s)$ is bounded on $\Omega\times \mathbb R^-$, the claim follows. To prove that $\{v_n^{+}\}$ is also bounded we assume by contradiction  that $\|v_n\| \rightarrow \infty$. We define
\[ \Omega_n^{+} = \{x \in \Omega: v_n(x) \geq 0\} \qquad \textup{ and } \qquad 
\Omega_n^{-} = \Omega \setminus \Omega_n^+,\]
and introduce the sequence $\{w_n\} \subset \Wop $ given by $w_n = v_n/\|v_n\|$. 
Observe that $\{w_n\} \subset \Wop$ is bounded in $\Wop$. Hence, up to a subsequence, it follows that 
$w_n \rightharpoonup  w$ in $\Wop$, 
$w_n \rightarrow w$ strongly in $L^r(\Omega)$ for $1 \leq r < \past$, and $w_n \rightarrow w$ a.e. in $\Omega$. 
We split the proof in several steps. 
\bigbreak

\noindent \textbf{Step 1:} $cw \equiv 0$. 
\medbreak
As $\|v_n^-\|$ is bounded and by assumption $\|v_n\|\to\infty$, clearly  $w^{-} \equiv 0$. It remains to show that $cw^{+} \equiv 0$. 
Assume by contradiction that $cw^{+} \not\equiv 0$ i.e., defining 
$\Omega^{+} := \{ x \in \Omega: c(x)w(x) > 0\}$, we assume
$|\Omega^{+}| > 0$. 
Since $\|v_n\| \rightarrow \infty$ and $\langle I_{\lambda}'(v_n), v_n \rangle\to 0$, it follows that
\begin{equation}
\label{p18.1}
 \frac{\langle I_{\lambda}'(v_n), v_n \rangle}{\|v_n\|^p } \to 0\,.
 \end{equation}
First of all, observe that
\begin{equation} \label{cw00}
\begin{aligned}
\langle I_{\lambda}'(v_n), v_n \rangle = \|v_n\|^p & - \big( \frac{\mu}{p-1} \big)^{p-1}
 \int_{\Omega_n^+} h(x) |v_n|^p  \, dx - \int_{\Omega_n^+}
  \lambda c(x)g(v_n)v_n \,dx  
 - \int_{\Omega_n^-} f_{\lambda}(x,v_n) \,v_n \,dx
 \\
& - \int_{\Omega_n^+} \Big[ \big(1+\frac{\mu}{p-1} v_n\big)^{p-1}  
- \big( \frac{\mu}{p-1} \big)^{p-1} |v_n|^{p-2} v_n \Big] v_n h(x)\, dx .
\end{aligned}
\end{equation}
Now, since $f_{\lambda}(x,s)$ is bounded on $\Omega\times \R^{-}$, 
we deduce that
\begin{equation}
\label{p18.2}
\frac{1}{\|v_n\|^p} \int_{\Omega_n^{-}}  f_{\lambda}(x,v_n) \,v_n\, dx \to 0\,.
\end{equation}
Moreover, using that $w_n \rightarrow w$ in $L^r(\Omega)$, $1 \leq r < \past$, with $w^{-} \equiv 0$, 
we have 
\begin{equation}
\label{p18.3}
\frac{1}{\|v_n\|^p}  \int_{\Omega_n^{+}} |v_n|^p h(x)\, dx 
=  \int_{\Omega_n^{+}} |w_n|^p h(x)\, dx \to \int_{\Omega} w^p h(x)\, dx\,,
\end{equation}
Next, we are going to show that
\begin{equation}
\label{p19.1}
\frac{1}{\|v_n\|^p} \int_{\Omega_n^{+}} \Big[ \big( 1+\frac{\mu}{p-1} v_n \big)^{p-1} - 
\big(\frac{\mu}{p-1}\big)^{p-1}|v_n|^{p-2}v_n \Big] v_n h(x)\, dx \to 0\,.
\end{equation}
% \textcolor{red}{Attention si possible utiliser ici les estimations ce que l'on a d\'ej\`a obtenues dans la Section \ref{Coercive-Case}.}
Observe that
\begin{equation*}
%\begin{aligned}
 \int_{\Omega_n^+} \Big[ \big( 1+\frac{\mu}{p-1}v_n  \big)^{p-1}  - 
\big(\frac{\mu}{p-1}\big)^{p-1}v_n^{p-1} \Big] v_n h(x)\, dx
=
 (p-1) \int_{\Omega_n^+} \Big[ \int_0^1 \big( s + \frac{\mu}{p-1}v_n\big)^{p-2} ds \Big] v_n h(x)\, dx\,.
 %\end{aligned}
\end{equation*}
We consider separately the case $p\geq 2$ and the case $1<p<2$. 
In case $p\geq 2$, there exists  $D>0$ such that
$$
\Big|\int_{\Omega_n^+} \Big[ \int_0^1 \big( s + \frac{\mu}{p-1}v_n\big)^{p-2} ds \Big] v_n h(x)\, dx
\Big|
\leq 
\frac{p-1}{\mu}
\int_{\Omega_n^+} \Big( 1 + \frac{\mu}{p-1}v_n\Big)^{p-1} |h(x)|\, dx
\leq 
D \|h\|_q ( 1 + \|v_n\|^{p-1})\,.
$$
On the other hand, in case $1 < p < 2$,  we have a constant $D>0$ with
$$
\Big|\int_{\Omega_n^+} \Big[ \int_0^1 \big( s + \frac{\mu}{p-1}v_n\big)^{p-2} ds \Big] v_n h(x)\, dx
\Big|
\leq 
\frac{p-1}{\mu}
\int_{\Omega_n^+} \Big(\frac{\mu}{p-1}v_n\Big)^{p-1} |h(x)|\, dx
\leq 
D \|h\|_q\|v_n\|^{p-1}\,. 
$$
The claim \eqref{p19.1} follows then directly from the above inequalities. %doing  $\|v_n\|\to\infty$. 
So, substituting \eqref{cw00}, \eqref{p18.2}, \eqref{p18.3} and \eqref{p19.1} in \eqref{p18.1} and using that $g$ is bounded on $\R^{-}$, we deduce that
%\begin{equation*}
%\lambda \int_{\Omega_n^+} c(x) \frac{g(v_n)}{v_n^{p-1}} w_n^p \, dx 
%\to 
%1 - \left( \frac{\mu}{p-1} \right)^{p-1} \int_{\Omega} w^p h(x)\, dx.
%\end{equation*}
%In fact, since $g$ is bounded on $\R^{-}$, we have also
%\[ \frac{1}{\|v_n\|^p} \int_{\Omega_n^{-}} \lambda c(x)g(v_n)v_n dx \to 0\,,\]
%and hence
\begin{equation}
\label{cw01}
\lambda \int_{\Omega} c(x) \frac{g(v_n)}{v_n^{p-1}} w_n^p\, dx 
\to
1 - \Big( \frac{\mu}{p-1} \Big)^{p-1} \int_{\Omega} w^p h(x)\, dx.
\end{equation}
%Now, let us define $\Omega^{+} := \{ x \in \Omega: c(x)w(x) > 0\}$ and observe that 
%$|\Omega^{+}| \neq \emptyset$, since we have assumed that $cw^{+} \not\equiv 0$. 
Let us prove that this is a contradiction. By Lemma \ref{gProperties}, we know that  
$\lim_{s\to+\infty} g(s)/s^{p-1} = +\infty$ and as
$w_n \rightarrow w > 0$ a.e. in $\Omega^{+}$, it follows that
\[ c(x) \frac{g(v_n)}{v_n^{p-1}} w_n^p \to + \infty \textup{ a.e. in } \Omega^{+}.\]
Since $|\Omega^{+}| > 0$, we have
\begin{equation} \label{cw02}
\int_{\Omega^{+}} c(x) \frac{g(v_n)}{v_n^{p-1}} w_n^p dx \to +\infty\,.
\end{equation}
On the other hand, as $g\geq 0$ on $\R^+$, $\frac{g(s)}{s^{p-1}}$ is bounded on $\R^-$ and 
$\|w_n\|_{pq'}$ is bounded, we have
%Then, we should observe that
%\begin{equation} 
%\label{cw03}
%\begin{aligned}
%\int_{\Omega^{+}} c(x) \frac{g(v_n)}{v_n^{p-1}} w_n^p dx 
%= 
%\int_{\Omega} c(x) \frac{g(v_n)}{v_n^{p-1}} w_n^p dx 
%& 
%- \int_{(\Omega \setminus \Omega^{+}) \cap \Omega_n^{+}} c(x) \frac{g(v_n)}{v_n^{p-1}} w_n^p dx 
%\\
%&
% - \int_{(\Omega \setminus \Omega^{+}) \cap \Omega_n^{-}} c(x) \frac{g(v_n)}{v_n^{p-1}} w_n^p dx\,.
%\end{aligned}
%\end{equation}
%Nevertheless, by using that $g \geq 0$ on $\R^{+}$, it follows that
\begin{equation} \label{cw04}
\int_{\Omega \setminus \Omega^{+}} c(x) \frac{g(v_n)}{v_n^{p-1}} w_n^p \,dx \geq -D.
\end{equation}
%and, by using that $g(s)/s^{p-1}$ is bounded on $\R^{-}$ and that $\|w_n\|_{q'}$ is bounded,  it follows that
%\begin{equation} \label{cw05}
%\Big| \int_{(\Omega \setminus \Omega^{+}) \cap \Omega_n^{-}} c(x) \frac{g(v_n)}{v_n^{p-1}} w_n^p dx 
%\Big| 
%\leq D.
% \int_{(\Omega \setminus \Omega^{+}) \cap \Omega_n^{-}} |c(x)||w_n|^p dx
 %\leq \frac{D \|c\|_q \|w_n\|^p}{S_N}  = \frac{D \|c\|_q}{S_N} \,.
%\end{equation}
So \eqref{cw02} and \eqref{cw04} together give a contradiction with \eqref{cw01}. Consequently, 
we  conclude that $cw \equiv 0$. 
\medbreak

\noindent \textbf{Step 2:} \textit{Let us introduce a new functional $J_{\lambda}: \Wop \rightarrow \R$ defined as}
\begin{equation*}
%\begin{aligned}
J_{\lambda}(v) 
 = I_{\lambda}(v) - \frac{p-1}{\mu p} \int_{\{v \geq \alpha_{\lambda}\}} 
\Big[ \big(1+\frac{\mu}{p-1}v \big)^p - \big( \frac{\mu}{p-1} \big)^p|v|^p \Big] h^{-}(x)\, dx 
% & = \frac{1}{p} \int_{\Omega} |\gradu|^p - \left( \frac{\mu}{p-1} \right)^{p-1} |v|^p h(x)\ dx 
%- \lambda \int_{\OmSupAlpha} c(x)G(v)dx\\
% & \qquad \quad - \int_{\OmSupAlpha} \Bigg[ \left|1+\frac{\mu}{p-1}v \right|^p  
%- \left( \frac{\mu}{p-1} \right)^p|v|^p \Bigg] h^{+}(x)dx \\
% & \qquad \quad - \int_{\OmSubAlpha} \left[ \lambda c(x)g(\alpha_{\lambda}) 
%+ \left(1+\frac{\mu}{p-1}\alpha_{\lambda} \right)^{p-1}h(x)\right](v-\alpha_{\lambda})\ dx\,, \\
% & \qquad \quad - \int_{\OmSubAlpha} \lambda c(x)G(\alpha_{\lambda}) 
%+ \frac{p-1}{\mu p } \left( 1 + \frac{\mu}{p-1}\alpha_{\lambda} \right)^{p-1} h(x)\ dx \\
% & \qquad \quad + \left( \frac{\mu}{p-1} \right)^{p-1} \int_{\OmSubAlpha} |v|^p h(x)\ dx\,.
%\end{aligned}
\end{equation*}
\textit{and let us introduce the sequence $\{z_n\} \subset \Wop$ defined by $z_n = t_n v_n$, 
where $t_n\in [0,1]$ satisfies 
\[ J_{\lambda}(z_n) = \max_{t \in [0,1]} J_{\lambda}(t v_n)\,,\]
(if $t_n$ is not unique we choose its smallest possible value). We claim that}
\[ \lim_{n \rightarrow \infty} J_{\lambda}(z_n) = + \infty\,.\]
\medbreak
We argue again by contradiction. Suppose the existence of  $M < +\infty$ such that
\begin{equation} \label{zn1}
\liminf_{n \rightarrow \infty} J_{\lambda}(z_n) \leq M,
\end{equation}
and  introduce a sequence $\{k_n\} \subset \Wop$, defined as
\[ k_n = \Big(\frac{2pM}{m_p}\Big)^{\frac{1}{p}}w_n 
= \Big(\frac{2pM}{m_p}\Big)^{\frac{1}{p}} \frac{v_n}{\|v_n\|}\,.\]
Let us prove, taking $M$ bigger if necessary, that for $n$ large enough we have 
\begin{equation} \label{zn10}
J_{\lambda}(k_n) > \frac32 M\,.
\end{equation}
As $\Big(\frac{2pM}{m_p}\Big)^{\frac{1}{p}} \frac{1}{\|v_n\|}\in [0,1]$ for $n$ large enough, 
this will give the contradiction
$$
\frac32 M\leq\liminf_{n \rightarrow \infty} J_{\lambda}(k_n)\leq 
\liminf_{n \rightarrow \infty}J_{\lambda}(z_n)\leq M.
$$
First of all, observe that $k_n \rightharpoonup k:= \big(\frac{2pM}{m_p}\big)^{\frac{1}{p}} w$ 
in $\Wop$,  $k_n \rightarrow k$ in $L^r(\Omega)$, for $1 \leq r < \past$, and $k_n \rightarrow k$ a.e. in $\Omega$. 
 By  the properties of $G$ (see Lemma \ref{gProperties}) together with $k\geq 0$ and  $c\, k\equiv 0$,
 it is easy to prove that
\begin{equation} \label{zn2}
\int_{\{k_n \geq \alpha_{\lambda}\}} c(x)G(k_n)\, dx \to \int_{\Omega} c(x)G(k) \, dx=0.
\end{equation}
%On the other hand, since by Step 3, $c\, k\equiv 0$,   we see that
%\begin{equation} \label{zn3}
%\int_{\Omega} c(x)G(k)\, dx = 0.
%\end{equation}
As $w^-\equiv0$, we have $\chi_n^-\to 0$ a.e. in $\Omega$ where $\chi_n^-$ is the characteristic function of 
$\Omega_n^-$. Recall (see \eqref{Flambdaminus}) that we have $m\in L^q(\Omega)\,,\ q > \max\{N/p,1\}$, such that, for a.e. $x\in \Omega$ and all $s\leq 0$,
$$
|F_{\lambda}(x,s)|\leq m(x)(1+|s|)\,.
$$
% \textcolor{red}{On ne peut pas metre ``recall" vu que cela n'a jamais \'et\'e montr\'e, je pense que c'est mieux de juste dire que c'est une cons\'equence de la d\'efinition de $F_{\lambda}$, peut \^etre aussi que cette estimation sera d\'ej\`a apparue dans la Section \ref{Coercive-Case}.}
This implies that
\begin{equation} \label{zn3}
\int_{\{k_n \leq \alpha_{\lambda}\}} F_{\lambda}(x,k_n) \, dx \to 0\,,
\qquad \mbox{ as well as }\qquad
\int_{\{k_n \leq \alpha_{\lambda}\}} |k_n|^p h(x)\, dx \to 0\,.
\end{equation}
%By arguing on the same way, as we know that $w^{-} \equiv 0$ and so, that $k^{-} \equiv 0$, it follows that
%\begin{equation} \label{zn4}
%\int_{\{k_n \leq \alpha_{\lambda}\}} |k_n|^p h(x)\ dx \to 0\,,
%\end{equation}
%and that
%\begin{equation} \label{zn5}
%\int_{\{k_n \leq \alpha_{\lambda}\}} \left[ \lambda c(x)g(\alpha_{\lambda})
%+ \left(1+\frac{\mu}{p-1}\alpha_{\lambda}\right)^{p-1}h(x)\right] k_n\ dx 
%\xrightarrow{n \rightarrow \infty} 0\,.
%\end{equation}
%Finally, observe that there exist positive constants $D_1$ and $D_2$, independent of $M$, such that
%\begin{equation} \label{zn6}
%\left| \int_{\{k_n \leq \alpha_{\lambda}\}} \left[ \lambda c(x)g(\alpha_{\lambda})
%+ \left(1+\frac{\mu}{p-1}\alpha_{\lambda}\right)^{p-1}h(x)\right] \alpha_{\lambda}\ dx \right| \leq D_1\,,
%\end{equation}
%and
%\begin{equation} \label{zn7}
%\left| \int_{\{k_n \leq \alpha_{\lambda}\}} \lambda c(x)G(\alpha_{\lambda}) 
%+ \frac{p-1}{p\mu} \left( 1+\frac{\mu}{p-1} \alpha_{\lambda} \right)^p h(x)\ dx \right| \leq D_2\,.
%\end{equation}
Taking into account \eqref{zn2} and \eqref{zn3} we obtain that
%Considering all this together,
%\textcolor{red}{Je ne pas pense que c'est une bonne id\'ee de mettre une phrase du type ``Considering all this...", les calculs sont compliqu\'es et cela facile beaucoup lorsque les quantit\'es que l'on utilise sont num\'erot\'ees.}
\begin{equation} \label{zn8}
\begin{aligned}
J_{\lambda}(k_n) = \frac{1}{p}  \int_{\Omega} \Big(|\nabla k_n|^p  
&- \big(\frac{\mu}{p-1} \big)^{p-1} h(x)|k_n|^p\Big)\, dx 
\\
& -  \frac{p-1}{p\mu} \int_{\{k_n \geq \alpha_{\lambda}\}} \Big[ \big( 1 + \frac{\mu}{p-1}k_n \big)^p 
- \big( \frac{\mu}{p-1}\big)^p |k_n|^p \Big] h^{+}(x)\,  dx  + o(1)\,.
\end{aligned}
\end{equation}
Now, observe that, by definition of  $m_p$,
\begin{equation} \label{zn9}
\frac{1}{p} \int_{\Omega} \Big(|\nabla k_n|^p  - \big(\frac{\mu}{p-1} \big)^{p-1} h(x)|k_n|^p \Big) \,dx 
\geq \frac{1}{p} m_p \|k_n\|^p = 2 M\,. 
\end{equation}
Furthermore, arguing as in \eqref{ineqPminus1}, observe that
\begin{equation*}
\begin{aligned}
\frac{1}{p}  \Big| \int_{\{k_n \geq \alpha_{\lambda}\}} \Big[ \big( 1 + \frac{\mu}{p-1} k_n \big)^p 
&- \big( \frac{\mu}{p-1} \big)^p |k_n|^p \Big] h^{+}(x) \,dx \Big| 
%\\
%&=  \Big| \int_{\{k_n \geq \alpha_{\lambda}\}} \Big( \int_0^1 \big| s + %\frac{\mu}{p-1} k_n \big|^{p-2} 
%%\big(  s + \frac{\mu}{p-1} k_n \big) \,ds \Big) h^{+}(x) \, dx \Big| 
%\\
%& \leq \int_{\{k_n \geq \alpha_{\lambda}\}} \Big( 1+\frac{\mu}{p-1} |k_n| %%%%%\Big)^{p-1}h^{+}(x)\, dx
%\\
%& \leq 
%\int_{\Omega} \Big( 1+\frac{\mu}{p-1}\big(\frac{2pM}{m_p}\big)^{\frac{1}{p}} |%w_n| \Big)^{p-1}
%h^{+}(x)\, dx
%\\
%&
 \leq 
C \|h^+\|_q \Big(1+ \big(\frac{2pM}{m_p}\big)^{\frac{p-1}{p}} \Big), 
\end{aligned}
\end{equation*}
where $C$ is independent of $M$.
This implies that
\[ J_{\lambda}(k_n) \geq 2M - C \|h^+\|_q \Big(1+ \big(\frac{2pM}{m_p}\big)^{\frac{p-1}{p}} \Big)+ o(1)\,,\]
and,  taking $M$ bigger if necessary, for any $n \in \N$ large enough,  \eqref{zn10} follows.
%Now, observe that $k_n$ and $z_n$ lay on the same ray in $\Wop$ for all $n \in \N$. 
%Moreover, since $\|v_n\| \rightarrow \infty$, for $n \in \N$ large enough it follows that
%\[ \frac{1}{\|v_n\|}\left(\frac{2pM}{m_p}\right)^{\frac{1}{p}} < 1.\]
%So, by gathering together \eqref{zn1} and \eqref{zn10}, we obtain a contradiction with the definition 
%of the sequence $\{z_n\}$, and so,  we can conclude that
%\[ \lim_{n\rightarrow \infty} J_{\lambda}(z_n) = +\infty.\]
\medbreak

\noindent \textbf{Step 3:} For $n \in \N$ large enough, $t_n \in (0,1)$. \medbreak

%\noindent First of all, since $\alpha_{\lambda} \leq 0$, observe that 
%\[ J_{\lambda}(0) = - \frac{p-1}{p\mu} \|h^{+}\|_1.\]
%So, the previous step implies that $t_n \in (0,1]$. On the other hand, 
By the definition of $J_{\lambda}$ and using that 
\[ \big(1+\frac{\mu}{p-1}s \big)^p - \big(\frac{\mu}{p-1} \big)^p s^p \geq 0\,, \quad \forall\ s \geq 0\,,\]
observe that
\begin{equation*}
\begin{aligned}
J_{\lambda}(v_n) 
%= I_{\lambda}(v_n) & - \frac{p-1}{p \mu} \int_{\{v_n \geq \alpha_{\lambda}\}} 
%\Bigg[ \left( 1 + \frac{\mu}{p-1} v_n \right)^p - \left(\frac{\mu}{p-1} \right)^p |v_n|^p \Bigg] h^{-}(x)\ dx
%\\
%= I_{\lambda}(v_n) & - \frac{p-1}{p \mu} \int_{\{v_n \geq \alpha_{\lambda}\} \cap \Omega_n^{-}} 
%\Bigg[ \left( 1 + \frac{\mu}{p-1} v_n \right)^p - \left(\frac{\mu}{p-1} \right)^p |v_n|^p \Bigg] h^{-}(x)\ dx
%\\
%& - \frac{p-1}{p \mu} \int_{\Omega_n^{+}} \Bigg[ \left( 1 + \frac{\mu}{p-1} v_n \right)^p 
%- \left(\frac{\mu}{p-1} \right)^p |v_n|^p \Bigg] h^{-}(x)\ dx 
%\\
\leq I_{\lambda}(v_n) & - \frac{p-1}{p \mu} \int_{\{\alpha_{\lambda}\leq v_n  \leq 0\}} 
\Big[ \big( 1 + \frac{\mu}{p-1} v_n \big)^p - \big(\frac{\mu}{p-1} \big)^p |v_n|^p \Big] h^{-}(x)\, dx 
\\
%\leq I_{\lambda}(v_n) & - \frac{p-1}{p \mu} \int_{\{v_n \geq \alpha_{\lambda}\} \cap \Omega_n^{-}} 
%\left( 1 + \frac{\mu}{p-1} v_n \right)^p h^{-}(x)\ dx \\
%& + \frac{1}{p} \left(\frac{\mu}{p-1}\right)^{p-1} \int_{\{v_n \geq \alpha_{\lambda}\} \cap \Omega_n^{-}} 
%|v_n|^p h^{-}(x)\ dx \\
\leq I_{\lambda}(v_n) & + \frac{1}{p} \big(\frac{\mu}{p-1}\big)^{p-1} \int_{\{ \alpha_{\lambda} \leq v_n \leq 0\}} |v_n|^p\, h^{-}(x) dx \\
\leq 
I_{\lambda}(v_n) & + \frac{1}{p} \big(\frac{\mu}{p-1}\big)^{p-1} \|\alpha_{\lambda}\|_{\infty}^p 
\|h^{-}\|_1\,.
\end{aligned}
\end{equation*}
Consequently, since $I_{\lambda}(v_n) \rightarrow d$, there exists $D>0$ such that, 
for all $n\in \mathbb N$,  $J_{\lambda}(v_n) \leq D$. 
Thus, taking into account that $J_{\lambda}(0) = - \frac{p-1}{p\mu} \|h^{+}\|_1$ and $J_{\lambda}(t_nv_n)\to+\infty$, we conclude
that  $t_n \in (0,1)$ for $n$ large enough. 

% \textcolor{red}{Je trouve que la premi\`ere in\'egalit\'e devrait \^etre justifi\'ee et peut  \^etre aussi la deuxi\`eme majoration (m\^eme si on voit sans souci que ``cela marche").}
\medbreak

\noindent \textbf{Step 4:} \textit{Conclusion.} \medbreak

First of all, as $t_n\in(0,1)$ for $n$ large enough, by the definition of $z_n$, observe that $\langle J_{\lambda}'(z_n),z_n\rangle = 0$, for those $n$. 
Thus, it follows that
\begin{equation*}
\begin{aligned}
J_{\lambda}(z_n) 
& 
= J_{\lambda}(z_n) - \frac{1}{p}\langle J_{\lambda}'(z_n), z_n \rangle 
\\
& 
= \lambda \int_{\{z_n \geq \alpha_{\lambda}\}} c(x)H(z_n)\,dx 
- \frac{p-1}{p\mu}\int_{\{z_n \geq \alpha_{\lambda}\}} \big(1+\frac{\mu}{p-1}z_n \big)^{p-1} h^{+}(x) \,dx 
\\
& \hspace{70mm}
- \int_{\{z_n \leq \alpha_{\lambda}\}} 
\big[F_{\lambda}(x,z_n)-\frac1p f_{\lambda}(x,z_n) \, z_n\big]\,dx.
%& \qquad - \int_{\{z_n \leq \alpha_{\lambda}\}} \left[ \lambda c(x)g(\alpha_{\lambda}) 
%+ \left(1+\frac{\mu}{p-1}\alpha_{\lambda}\right)^{p-1}h(x) \right] 
%\left(\frac{p-1}{p} z_n - \alpha_{\lambda}\right) \, dx 
%\\
%& \qquad - \int_{\{z_n \leq \alpha_{\lambda}\}} \lambda c(x)G(\alpha_{\lambda}) + \frac{p-1}{p\mu} 
%\left( 1+ \frac{\mu}{p-1}\alpha_{\lambda} \right)^{p} h(x)\ dx\,,
\end{aligned}
\end{equation*}
Using the definition of $f_{\lambda}(x,s)$ for $s\leq \alpha_{\lambda}(x)$
 and the fact that $\|z_n^-\|$ is bounded, we easily deduce the existence of $D_1>0$ such that, for all $n$ large enough,
\begin{equation} \label{cc1}
\frac{p-1}{p\mu} \int_{\Omega} \Big|1+\frac{\mu}{p-1}z_n\Big|^{p-2}
\Big(1+\frac{\mu}{p-1}z_n\Big) h^{+}(x)\, dx 
\leq  -J_{\lambda}(z_n) + \lambda \int_{\Omega} c(x)H(z_n) \,dx + D_1\,.
\end{equation}
Now, since $\{v_n\}$ is a Cerami sequence, observe that (again for $n$ large enough)
\begin{equation*}
\begin{aligned}
d+1 & \geq I_{\lambda}(v_n) - \frac{1}{p} \langle I_{\lambda}'(v_n), v_n \rangle 
\\
& = \lambda \int_{\{v_n \geq \alpha_{\lambda}\}} c(x)H(v_n) \, dx 
- \frac{p-1}{p\mu} \int_{\{v_n \geq \alpha_{\lambda}\}} \big(1+\frac{\mu}{p-1}v_n\big)^{p-1} h(x)\, dx 
\\
 & \hspace{70mm} - \int_{\{v_n \leq \alpha_{\lambda}\}} 
 \big[F_{\lambda}(x,v_n)-\frac1p f_{\lambda}(x,v_n) \, v_n\big]\,dx
 %\left[ \lambda c(x)g(\alpha_{\lambda}) 
 %+ \left(1+\frac{\mu}{p-1}\alpha_{\lambda}\right)^{p-1}h(x) \right] 
 %\left(\frac{p-1}{p} v_n - \alpha_{\lambda}\right) dx \\
%& \qquad - \int_{\{v_n \leq \alpha_{\lambda}\}} \lambda c(x)G(\alpha_{\lambda}) + \frac{p-1}{p\mu} 
%\left( 1+ \frac{\mu}{p-1}\alpha_{\lambda} \right)^{p} h(x)\ dx + o(1)\,,
\end{aligned}
\end{equation*}
and, as above, there exists a constant $D_2>0$ such that
\begin{equation} \label{cc2}
\lambda \int_{\Omega}c(x)H(v_n)\, dx 
\leq \frac{p-1}{p\mu} \int_{\Omega} \Big|1+\frac{\mu}{p-1}v_n\Big|^{p-2}
\Big(1+\frac{\mu}{p-1}v_n \Big)h(x)\, dx + D_2.
\end{equation}
Moreover, observe that
\begin{equation*}
\begin{aligned}
\int_{\Omega} & \big| 1 + \frac{\mu}{p-1} v_n \big|^{p-2} \big(1 + \frac{\mu}{p-1} v_n \big)\,h(x)\, dx
\\
& = \int_{\Omega} \big| 1 + \frac{\mu}{p-1} v_n \big|^{p-2} \big(1 + \frac{\mu}{p-1} v_n \big)\,h^{+}(x)\, dx
 -
  \int_{\Omega} \big| 1 + \frac{\mu}{p-1} v_n \big|^{p-2} \big(1 + \frac{\mu}{p-1} v_n \big)\,h^{-}(x)\, dx 
  \\
& = \frac{1}{t_n^{p-1}} 
\int_{\Omega} \big| t_n + \frac{\mu}{p-1} z_n \big|^{p-2} \big(t_n + \frac{\mu}{p-1} z_n \big)\,h^{+}(x)\, dx
 - 
 \int_{\Omega} \big| 1 + \frac{\mu}{p-1} v_n \big|^{p-2} \big(1 + \frac{\mu}{p-1} v_n \big)\,h^{-}(x)\, dx 
 \\
& 
\leq \frac{1}{t_n^{p-1}} \int_{\Omega} \big| 1 + \frac{\mu}{p-1} z_n \big|^{p-2} 
\big(1 + \frac{\mu}{p-1} z_n \big)\,h^{+}(x)\, dx 
- 
\int_{\Omega} \big| 1 + \frac{\mu}{p-1} v_n \big|^{p-2} \big(1 + \frac{\mu}{p-1} v_n \big)\,h^{-}(x)\, dx\,. 
\end{aligned}
\end{equation*}
Considering together this inequality with \eqref{cc1} and \eqref{cc2}, we obtain that
\begin{equation}\label{cc3}
\begin{aligned}
\lambda \int_{\Omega}c(x)H(v_n)\, dx \leq 
&D_2 - \frac{J_{\lambda}(z_n)}{t_n^{p-1}} + \frac{\lambda}{t_n^{p-1}} \int_{\Omega} c(x)H(z_n)\, dx 
+ \frac{D_1}{t_n^{p-1}} \\ 
&\hspace{27mm} - \frac{p-1}{p\mu} \int_{\Omega} \big| 1 + \frac{\mu}{p-1} v_n \big|^{p-2} 
\big(1 + \frac{\mu}{p-1} v_n \big)\,h^{-}(x)\, dx.
\end{aligned}
\end{equation}
Now, since $H$ is bounded on $\R^{-}$, there exists $D_3>0$ such that, for all $n\in \N$,
\begin{equation} \label{cc4}
\int_{\Omega_n^{-}} c(x)H(z_n)\, dx \leq D_3\,.
\end{equation}
On the other hand, using $iv)$ of Lemma \ref{gProperties}, it follows that
\begin{equation} \label{cc5}
\int_{\Omega_n^{+}} c(x)H(z_n)\, dx \leq t_n^{p-1} \int_{\Omega_n^{+}} c(x)H(v_n)\, dx + D_4,
\end{equation}
for some positive constant $D_4$. 
Hence, substituting \eqref{cc4} and \eqref{cc5} in \eqref{cc3}, it follows that
\begin{equation*}
%\begin{aligned}
\lambda \int_{\Omega_n^{-}}c(x)H(v_n)\, dx \leq 
%&  
D_5 - \frac{J_{\lambda}(z_n)-D_6}{t_n^{p-1}}
%\\
%& 
- \frac{p-1}{p\mu} \int_{\Omega} \big| 1 + \frac{\mu}{p-1} v_n \big|^{p-2} 
\big(1 + \frac{\mu}{p-1} v_n \big)\,h^{-}(x)\, dx.
%\end{aligned}
\end{equation*}
Arguing as in the previous steps, observe that
\begin{equation*}
%\begin{aligned}
\int_{\Omega} 
%& 
\big| 1 + \frac{\mu}{p-1} v_n \big|^{p-2} \big(1 + \frac{\mu}{p-1} v_n \big)\,h^{-}(x)\, dx 
%\\
%& = \int_{\Omega_n^{+}} \left(1+\frac{\mu}{p-1}v_n\right)^{p-1}h^{-}(x)\ dx 
%+ \int_{\Omega_n^{-}} \left| 1 + \frac{\mu}{p-1} v_n \right|^{p-2} 
%\left(1 + \frac{\mu}{p-1} v_n \right)h^{-}(x)\ dx 
%\\
%& \geq \int_{\Omega_n^{-}} \big| 1 + \frac{\mu}{p-1} v_n \big|^{p-2} 
%\big(1 + \frac{\mu}{p-1} v_n \big) \,h^{-}(x)\, dx 
%\\
%& 
\geq - \int_{\Omega_n^{-}} \big| 1 + \frac{\mu}{p-1} v_n \big|^{p-1}\, h^{-}(x)\, dx
% \\
%& \geq -c_p \int_{\Omega_n^{-}} h^{-}(x)\ dx - c_p \left( \frac{\mu}{p-1} \right)^{p-1} 
%\int_{\Omega_n^{-}} |v_n|^{p-1} h^{-}(x) dx \\
%& 
\geq -D_7 \|h^{-}\|_q \big(1+ \|v_n^{-}\|^{p-1}\big),
%\end{aligned}
\end{equation*}
and so, we have that
\begin{equation*}
%\begin{aligned}
\lambda \int_{\Omega_n^{-}}c(x)H(v_n)\, dx \leq 
 D_5 - \frac{J_{\lambda}(z_n)-D_6}{t_n^{p-1}} %+ \frac{D_1+D_3}{t_n^{p-1}} 
 +D_7 \|h^{-}\|_q \big(1+ \|v_n^{-}\|^{p-1}\big)\,.
%\end{aligned}
\end{equation*}
By Step 1, we know that $\|v_n^{-}\|$ is bounded and Step 4  shows that $J_{\lambda}(z_n) \rightarrow \infty$. 
Recall also that, by  Step 4,  $t_n \in (0,1)$. 
%Hence,  the right hand side in the above inequality goes to minus infinity. 
This implies that
\begin{equation} \label{cc6}
\lambda \int_{\Omega_n^{-}}c(x)H(v_n)\, dx \to - \infty
\end{equation}
which contradicts the fact that  $H$ is bounded on $\R^{-}$. This allows to conclude that the Cerami sequences for $I_{\lambda}$ at level 
$d \in \R$ are bounded.
\end{proof}

%Next, let us introduce another quantity related with $m_p$. Let us introduce 
%\begin{equation*}
%m_p:= \inf \left\{ \int_{\Omega} |\nabla w|^p - \left(\frac{\mu}{p-1} \right)^{p-1} |w|^p h(x)\,dx:
% w \in \Wop,\, \|w\| = 1 \right\}\,.
%\end{equation*}
%By assuming $m_p > 0$, we can state the following Lemmata.

%\begin{remark}
%Observe that $m_p > 0$ implies immediately that $m_p^c > 0$. Hence, in the following lemmata we consider 
%stronger assumptions that in the Lemma \ref{boundednessCerami} and the Proposition \ref{propCeramiCondition}. 
%It is important pointing out that $m_p > 0$ is in fact a smallness condition on $h$. See Appendix \ref{AR}.
%\end{remark}

Now, we turn to the verification of the mountain pass geometry when $\lambda \geq 0$ is small.

\begin{lemma} \label{geometrySmall1}
Assume that \eqref{A1} holds and suppose that $m_p > 0$. For $\lambda \geq 0$ small enough, 
there exists $r > 0$ such that $I_{\lambda}(v) > I_{\lambda}(0)$ for $\|v\| = r$.
\end{lemma}

\begin{proof}
For an arbitrary fixed $r >0$, let $v \in \Wop$ be such that $\|v\| = r$. We can write
\begin{equation*}
\begin{aligned}
I_{\lambda}(v) 
%& = \frac{1}{p} \int_{\Omega} |\gradv|^p dx - \frac{p-1}{p\mu} \int_{\{v \geq \alpha_{\lambda}\}} 
%\left( 1 + \frac{\mu}{p-1}v \right)^p h(x)\,dx 
%\\
%& \qquad - \int_{\{v \leq \alpha_{\lambda}\}} h(x) 
%\left[ \left(1+\frac{\mu}{p-1}\alpha_{\lambda} \right)^{p-1}(v-\alpha_{\lambda}) + \frac{p-1}{p\mu} 
%\left(1+\frac{\mu}{p-1}\alpha_{\lambda} \right)^{p} \right] dx 
%\\
%& \qquad - \lambda \left( \int_{\{v \geq \alpha_{\lambda}\}} c(x)G(v)\,dx  
%+ \int_{\{v \leq \alpha_{\lambda}\}} c(x)\big[ g(\alpha_{\lambda})(v-\alpha_{\lambda}) 
%+ G(\alpha_{\lambda}) \big] dx \right) 
%\\
%& = \frac{1}{p} \int_{\Omega} |\gradv|^p - \left( \frac{\mu}{p-1} \right)^{p-1} |v|^p h(x)\, dx 
%+ \frac{1}{p} \left( \frac{\mu}{p-1} \right)^{p-1} \int_{\{v \leq \alpha_{\lambda}\}} |v|^p h(x)\, dx 
%\\
%& \qquad - \frac{p-1}{p\mu} \int_{\{v \geq \alpha_{\lambda}\}} \left[ \left( 1 + \frac{\mu}{p-1}v \right)^p 
%- \left( \frac{\mu}{p-1} \right)^{p} |v|^p \right] h(x)\,dx 
%\\
%& \qquad - \int_{\{v \leq \alpha_{\lambda}\}} h(x) 
%\left[ \left(1+\frac{\mu}{p-1}\alpha_{\lambda} \right)^{p-1}(v-\alpha_{\lambda}) 
%+ \frac{p-1}{p\mu} \left(1+\frac{\mu}{p-1}\alpha_{\lambda} \right)^{p} \right] dx 
%\\
%& \qquad - \lambda \left( \int_{\{v \geq \alpha_{\lambda}\}} c(x)G(v)\,dx  
%+ \int_{\{v \leq \alpha_{\lambda}\}} c(x)\big[ g(\alpha_{\lambda})(v-\alpha_{\lambda}) 
%+ G(\alpha_{\lambda}) \big] dx \right) 
%\\
& = \frac{1}{p} \int_{\Omega} \Big(|\gradv|^p - \big( \frac{\mu}{p-1} \big)^{p-1} (v^{+})^p h(x)\Big)\, dx 
- \frac{1}{p} \big( \frac{\mu}{p-1} \big)^{p-1} \int_{\{\alpha_{\lambda} \leq v \leq 0\}} |v|^p h(x)\, dx 
\\
& \qquad - \frac{p-1}{p\mu} \int_{\{v \geq \alpha_{\lambda}\}} \Big[ \big( 1 + \frac{\mu}{p-1}v \big)^p 
- \big( \frac{\mu}{p-1} \big)^{p} |v|^p \Big] h(x)\,dx 
\\
& \qquad - \int_{\{v \leq \alpha_{\lambda}\}} h(x) \Big[ \big(1+\frac{\mu}{p-1}\alpha_{\lambda} \big)^{p-1}
(v-\alpha_{\lambda}) + \frac{p-1}{p\mu} \big(1+\frac{\mu}{p-1}\alpha_{\lambda} \big)^{p} \Big]\, dx 
\\
& \qquad - \lambda \Big( \int_{\{v \geq \alpha_{\lambda}\}} c(x)G(v)\,dx  
+ \int_{\{v \leq \alpha_{\lambda}\}} c(x)\big[ g(\alpha_{\lambda})(v-\alpha_{\lambda}) 
+ G(\alpha_{\lambda}) \big] \, dx \Big).
%\\
%& \qquad - \frac{p-1}{p\mu}\|h^{-}\|_1 + \frac{p-1}{p\mu}\|h^{-}\|_1\,.
\end{aligned}
\end{equation*}
%& \geq \frac{1}{p} \min\{1,m_p\}\, r^p - \frac{1}{p} \left( \frac{\mu}{p-1} \right)^{p-1} 
%\int_{\{\alpha_{\lambda} \leq v \leq 0\}} |v|^p h(x)\, dx \\
%& \qquad - \frac{p-1}{p\mu} \int_{\{v \geq \alpha_{\lambda}\}} \left[ \left( 1 + \frac{\mu}{p-1}v \right)^p 
%- \left( \frac{\mu}{p-1} \right)^{p} |v|^p \right] h(x)\,dx \\
%& \qquad - \int_{\{v \leq \alpha_{\lambda}\}} h(x) 
%\left[ \left(1+\frac{\mu}{p-1}\alpha_{\lambda} \right)^{p-1}(v-\alpha_{\lambda} 
%+ \frac{p-1}{p\mu} \left(1+\frac{\mu}{p-1}\alpha_{\lambda} \right)^{p} \right] dx \\
%& \qquad - \lambda \left( \int_{\{v \geq \alpha_{\lambda}\}} c(x)G(v)\,dx  
%+ \int_{\{v \leq \alpha_{\lambda}\}} c(x)\big[ g(\alpha_{\lambda})(v-\alpha_{\lambda}) 
%+ G(\alpha_{\lambda}) \big] dx \right) \\
%& \qquad - \frac{p-1}{p\mu}\|h^{-}\|_1 + \frac{p-1}{p\mu}\|h^{-}\|_1\,.
Now, observe that, as above,
\begin{equation*}
%\begin{aligned}
%& 
\Bigg| \int_{\{v \geq \alpha_{\lambda}\}} 
\Big[ \big(1+\frac{\mu}{p-1}v \big)^p - \big( \frac{\mu}{p-1} \big)^p |v|^p \Big] h(x)\,dx  \Bigg|  
%\\
%& \qquad \leq \left| \int_{\Omega} \left[ \left(1+\frac{\mu}{p-1}v \right)^p 
%- \left( \frac{\mu}{p-1} \right)^p |v|^p \right] h(x)\,dx  \right| 
%\\
%& \qquad  
%&\leq \Bigg| p \int_{\Omega} \Big( \int_0^1 \big| s + \frac{\mu}{p-1}v_n \big|^{p-2} 
%\big( s + \frac{\mu}{p-1}v_n \big) ds \Big) h(x)\, dx \Bigg| 
%\\
%& \qquad 
%&
\leq p \int_{\Omega} \big( 1 + \frac{\mu}{p-1}|v| \big)^{p-1} |h(x)|\, dx %= (\ast) \,.
%\\
%&
\leq D_1 (1+r^{p-1}),
%p \int_{\Omega} \Big( 1 + \frac{\mu}{p-1}|v_n| \Big)^{p-1} |h(x)|\, dx. %= (\ast) \,.
%\end{aligned}
\end{equation*}
%Furthermore, by applying \eqref{ineqBinomio} and H\"older and Sobolev inequalities we deduce that
%\begin{equation} \label{mpg1}
%\begin{aligned}
%(\ast) & \leq c_p \left[ \|h\|_1 + \left( \frac{\mu}{p-1} \right)^{p-1} 
%\int_{\Omega} |v_n|^{p-1} |h(x)| dx \right] 
%\\
%& \leq c_p \left[ \|h\|_1 + \frac{1}{S_N} \left( \frac{\mu}{p-1} \right)^{p-1} \|h\|_q \,r^{p-1} \right]\,.
%\end{aligned}
%\end{equation}
with $D_1$ independent of $\lambda$. In the same way, using the fact that 
$\alpha_{\lambda}\in [-\frac{p-1}{\mu},0]$, we deduce that
$$
\begin{array}{c}
\displaystyle
\Big|
- \frac{1}{p} \big( \frac{\mu}{p-1} \big)^{p-1} \int_{\{\alpha_{\lambda} \leq v \leq 0\}} |v|^p h(x)\, dx 
\Big|
\leq D_2\,,
\\[2mm]
\displaystyle
\Big|
- \int_{\{v \leq \alpha_{\lambda}\}} h(x) \Big[ \big(1+\frac{\mu}{p-1}\alpha_{\lambda} \big)^{p-1}
(v-\alpha_{\lambda}) + \frac{p-1}{p\mu} \big(1+\frac{\mu}{p-1}\alpha_{\lambda} \big)^{p} \Big]\, dx 
\Big|
\leq D_3+D_4 r\,,
\end{array}
$$
with $D_2$, $D_3$ and $D_4$ independent of $\lambda$.
%By arguing on the same way, we deduce that
%\begin{equation} \label{mpg2}
%0 \leq \left|\int_{\{v  \leq \alpha_{\lambda}\}} 
%\left(1+\frac{\mu}{p-1} \alpha_{\lambda} \right)^{p-1} v h(x)\, dx \right| 
%\leq \left\| 1 + \frac{\mu}{p-1} \alpha_{\lambda} \right\|_{\infty}^{p-1} \frac{1}{S_N}\|h\|_q\, r\, ,
%\end{equation} 
%and, it is simpler to see that
%\begin{eqnarray} 
%\displaystyle 
%\frac{1}{p} \left( \frac{\mu}{p-1} \right)^{p-1} \int_{\{\alpha_{\lambda} \leq v \leq 0\}} |v|^p h(x)\, dx 
%\leq \frac{1}{p} \left( \frac{\mu}{p-1} \right)^{p-1} \|\alpha_{\lambda}\|_{\infty}^p \|h\|_1\,, 
%\label{mpg3} 
%\\
%\displaystyle 
%\int_{\{v \leq \alpha_{\lambda}\}} \left( 1 + \frac{\mu}{p-1} \alpha_{\lambda} \right)^{p-1} 
%\alpha_{\lambda} h(x)\, dx \geq - \frac{p-1}{\mu} \|h\|_1 
%\left\| 1 + \frac{\mu}{p-1} \alpha_{\lambda} \right\|_{\infty}^{p-1} \,, \label{mpg4}
%\\
%\displaystyle 
%\frac{p-1}{p\mu} \int_{\{ v \leq \alpha_{\lambda}\}} 
%\left( 1 + \frac{\mu}{p-1} \alpha_{\lambda} \right)^p h(x)\, dx \leq \frac{p-1}{p\mu} \|h\|_1 
%\left\| 1 + \frac{\mu}{p-1} \alpha_{\lambda} \right\|_{\infty}^{p} \,. \label{mpg5}
%\end{eqnarray}
Finally, observe that
\begin{equation*}
% \label{mpg6}
\begin{aligned}
\int_{\Omega} \Big(|\gradv|^p - \big( \frac{\mu}{p-1} \big)^{p-1} (v^{+})^p h(x)\Big)\, dx 
& 
= 
\int_{\Omega} \Big(|\nabla v^{+}|^p - \big( \frac{\mu}{p-1} \big)^{p-1} (v^{+})^p h(x)\Big)\, dx 
+  \int_{\Omega} |\nabla v^{-}|^p\, dx 
\\
& \geq m_p \|v^{+}\|^p + \|v^{-}\|^p \geq \min\{1,m_p\} \|v\|^p = \min\{1,m_p\}\,r^p\,.
\end{aligned} 
\end{equation*} 
So, we obtain that
\begin{equation}
\label{minI}
\begin{aligned}
I_{\lambda}(v) 
 \geq 
\frac{1}{p} \min\{1,m_p\}\, r^p &-  D_1 r^{p-1} 
- D_4 r - D_5\,,
\\
&% \hspace{25mm} 
- \lambda \Big( \int_{\{v \geq \alpha_{\lambda}\}} c(x)G(v)\,dx  
+ \int_{\{v \leq \alpha_{\lambda}\}} c(x)\big[ g(\alpha_{\lambda})(v-\alpha_{\lambda}) 
+ G(\alpha_{\lambda}) \big]\, dx \Big),
\end{aligned}
\end{equation}
where the constants $D_i$ are independent of $\lambda$. Moreover, observe that for $r$ large enough, 
%where $D = D(\|h\|_q,\|\alpha_{\lambda}\|_{\infty},\mu,N,p) > 0$. So, by choosing
%\[ r \geq \max \left\{ \frac{\frac{c_p}{S_N} \left( \frac{\mu}{p-1} \right)^{p-1} \|h\|_q}{\frac{1}{2p} 
%\min\{1,m_p\}}\, , \left( \frac{1+\left\| 1 + \frac{\mu}{p-1} \alpha_{\lambda} \right\|_{\infty}^{p-1} 
%\frac{1}{S_N}\|h\|_q}{\frac{1}{2p} \min\{1,m_p\}} \right)^{p-1} \,, \frac{D}{2} \right\}\,,\]
%we deduce that
\begin{equation} \label{mpg7}
%\frac{1}{p} \min\{1,m_p\}\, r^p - \frac{c_p}{S_N} \left( \frac{\mu}{p-1}  \right)^{p-1} \|h\|_q \, r^{p-1} 
%- \left\| 1 + \frac{\mu}{p-1} \alpha_{\lambda} \right\|_{\infty}^{p-1} \frac{1}{S_N}\|h\|_q\, r 
% - D \geq \frac{1}{2} r\,.
\frac{1}{p} \min\{1,m_p\}\, r^p -  D_1 r^{p-1} 
- D_4 r - D_5
\geq \frac{1}{2p} \min\{1,m_p\}\, r^p +I_{\lambda}(0)\,.
\end{equation}
On the other hand, by Lemma \ref{gProperties}, for every $\delta>0$, 
%as well as H\"older and Sobolev inequalities, 
%observe that
\begin{equation} \label{mpg8}
%\begin{aligned}
\Big| \Big( \int_{\{v \geq \alpha_{\lambda}\}} c(x)G(v)\,dx  
+ \int_{\{v \leq \alpha_{\lambda}\}} c(x)\big[ g(\alpha_{\lambda})(v-\alpha_{\lambda}) 
+ G(\alpha_{\lambda}) \big] \,dx \Big) \Big|
\leq 
D_6 r^{p+\delta} + D_7 r + D_8,
%
%\left| \int_{\{v \geq \alpha_{\lambda}\}} c(x) G(v)\, dx \right| \leq \int_{\Omega} c(x) G(v) \,dx 
%\leq D_1 \int_{\Omega} c(x) |v|^{p+\epsilon} + D_2 \|c\|_1 \leq \frac{D_1}{S_N}\|c\|_q\, r^{p+\delta}
% + D_2 \|c\|_1\,,
%\end{aligned}
\end{equation}
for some constant $D_6$, $D_7$, $D_8$ independent of $\lambda$. 
%By arguing on the same line, we deduce that there exists positive constant $D_3, D_4$ and $D_5$ such that
%\begin{eqnarray}
%\displaystyle \left| \int_{\{v \leq \alpha_{\lambda}\}} c(x)g(\alpha_{\lambda}) v\, dx \right| 
%\leq \frac{D_3}{S_N} \|c\|_q\, r\,, \label{mpg9} \\
%\displaystyle \left|\int_{\{v \leq \alpha_{\lambda}\}} c(x)g(\alpha_{\lambda}) \alpha_{\lambda} \, dx \right| 
%\leq D_4 \|c\|_1\,,
%\label{mpg10} \\
%\displaystyle \left|\int_{\{v \leq \alpha_{\lambda}\}} c(x) G(\alpha_{\lambda})\, dx \right| \leq D_5 \|c\|_1\,. 
%\label{mpg11}
%\end{eqnarray}
%So, by taking into account \eqref{mpg8}-\eqref{mpg11} and choosing $\lambda \geq 0$ small enough, 
%we deduce that
Hence, for $\lambda$ small enough, we have
\begin{equation} \label{mpg12}
\lambda \Big( \int_{\{v \geq \alpha_{\lambda}\}} c(x)G(v)\,dx  
+ \int_{\{v \leq \alpha_{\lambda}\}} c(x)\big[ g(\alpha_{\lambda})(v-\alpha_{\lambda}) 
+ G(\alpha_{\lambda}) \big] \, dx \Big) \leq  \frac{1}{4p} \min\{1,m_p\}\, r^p,
\end{equation}
and so, gathering \eqref{minI}, \eqref{mpg7} and \eqref{mpg12}, we  conclude that
\[ I_{\lambda}(v) 
%\geq  \frac{1}{4p} \min\{1,m_p\}\, r^p + \frac{p-1}{p\mu}\|h^{-}\|_1 
\geq  \frac{1}{4p} \min\{1,m_p\}\, r^p + I_{\lambda}(0) > I_{\lambda}(0)\,.
\]
\end{proof}

\begin{lemma} \label{geometrySmall2} Assume that \eqref{A1} holds and  that $m_p > 0$. For any $\lambda >0$, $M >0$,  and $r>0$, there exists $w \in \Wop$ such that $\|w\| > r$ and $I_{\lambda}(w) \leq -M$.
\end{lemma}

\begin{proof}
Consider $v \in \mathcal{C}_0^{\infty}(\Omega)$ such that $v \geq 0$ and $cv \not\equiv 0$ 
and let us take $t \in \R^{+}$, $t \geq 1$. First of all, as $\alpha_{\lambda} \leq 0$, observe that
\begin{equation*}
\begin{aligned}
I_{\lambda}(tv) 
%= \frac{1}{p}t^p \int_{\Omega} |\gradv|^p - \left( \frac{\mu}{p-1} \right)^{p-1} |v|^p h(x)\, dx 
%- \lambda \int_{\Omega} c(x)G(tv)\, dx \\
%& \qquad - \frac{p-1}{p\mu} \int_{\Omega} \left[ \left( 1 + \frac{\mu}{p-1}tv \right)^p 
%- \left( \frac{\mu}{p-1} \right)^p (tv)^p \right] h(x)\, dx \\
& \leq  \frac{1}{p} t^p \int_{\Omega} \Big(|\gradv|^p - \big( \frac{\mu}{p-1} \big)^{p-1} |v|^p h(x)\Big)\, dx 
- \lambda t^p \int_{\Omega} c(x)v^p \frac{G(tv)}{t^p v^p}\, dx \\
& \hspace{50mm}+ \frac{p-1}{p\mu} \int_{\Omega} \Big[ \big( 1 + \frac{\mu}{p-1}tv \big)^p 
- \big( \frac{\mu}{p-1} \big)^p (tv)^p \Big] h^{-}(x)\, dx.
\end{aligned}
\end{equation*}
As above, we have
$$
\frac{1}{p} \int_{\Omega} \Big[ \big( 1 + \frac{\mu}{p-1}tv \big)^p 
- \big( \frac{\mu}{p-1} \big)^p (tv)^p \Big] h^{-}(x)\, dx
\leq 
t^{p-1} \int_{\Omega} \big(1+ \frac{\mu}{p-1}v \big)^{p-1} h^{-}(x)\, dx.
$$
Hence we obtain
\begin{equation*}
%\begin{aligned}
I_{\lambda}(tv) 
 \leq 
 %& 
 t^p \Bigg[ \frac{1}{p} \int_{\Omega}\Big( |\gradv|^p 
 - \big( \frac{\mu}{p-1} \big)^{p-1} |v|^p h(x)\Big)\, dx 
- \lambda \int_{\Omega} c(x)v^p \frac{G(tv)}{t^p v^p}\, dx 
%\\
%& \hspace{50mm}  
+  \frac{1}{t} \frac{p-1}{\mu} \Big\|1+ \frac{\mu}{p-1}v \Big\|_{\infty}^{p-1}
 \| h^{-}\|_1 \Bigg]\,.
%\end{aligned}
\end{equation*}
Now, since by Lemma \ref{gProperties}, we have
\[ \lim_{t \rightarrow \infty} \lambda \int_{\Omega} c(x)v^p  \frac{G(tv)}{(tv)^p}\, dx = \infty\,,\]
%On the other hand, observe that
%\[ \lim_{t \rightarrow \infty} \frac{1}{t} \frac{p-1}{\mu} 
%\left\|1+ \frac{\mu}{p-1}v \right\|_{\infty}^{p-1} \| h^{-}\|_1 = 0\,.\]
%Hence, 
we deduce that $ \displaystyle \lim_{t \rightarrow \infty} I_{\lambda}(tv) = - \infty$ from which the lemma follows.
\end{proof}

%and so, we have the existence of $w \in \Wop$ with  $\|w\| > r$ and
%\[ I_{\lambda}(w) \leq -\frac{p-1}{p\mu} \int_{\Omega}h(x)\,dx = I_{\lambda}(0)\,.\]
%\end{proof}

%\section{Proofs of Theorem \ref{th1} } \label{SER}

\begin{prop} \label{localMinimumSmall}
Assume that \eqref{A1} holds and suppose that $m_p > 0$. Moreover, suppose that $\lambda \geq 0$ is small 
enough in order to ensure that the conclusion of Lemma \ref{geometrySmall1} holds. Then, $I_{\lambda}$ possesses a critical 
point $v \in B(0,r)$ with $I_{\lambda}(v) \leq I_{\lambda}(0)$, which is a local minimum of $I_{\lambda}$.
\end{prop}

\begin{proof}
From Lemma \ref{geometrySmall1}, we see that there exists  $r > 0$ such that
\[ m:= \inf_{v \in B(0,r)} I_{\lambda}(v) \leq I_{\lambda}(0) \qquad \textup{ and } \qquad 
I_{\lambda}(v) > I_{\lambda}(0) \,\,\textup{ if }\,\, \|v\| = r\,.\]
Let $\{v_n\} \subset B(0,r)$ be such that $I_{\lambda}(v_n) \to m$. Since $\{v_n\}$ is bounded, up to a subsequence, it follows that $v_n \rightharpoonup v \in \Wop$. By the weak lower semicontinuity of the norm and of the functional $I_{\lambda}$, we have
\[ \|v\| \leq \liminf_{n \to \infty} \|v_n\| \leq r \quad \textup{ and } \quad I_{\lambda}(v) \leq \liminf_{n \to \infty} I_{\lambda}(v_n) = m \leq I_{\lambda}(0)\,.\]
Finally, as $I_{\lambda}(v) > I_{\lambda}(0)$ if $\|v\| = r$, we deduce that $v \in B(0,r)$ is a local minimum of $I_{\lambda}$.
\end{proof}

\begin{proof}[\textbf{Proof of  Theorem \ref{th1}}]
%First of all, as we have already pointed out several times (Remark \ref{muPositive1}), we can assume $\mu > 0$ 
%and our results will hold for $\mu \in \R \setminus\{0\}$. 
Assume that 
$\lambda > 0$ is small enough in order to ensure that the conclusion of Lemma \ref{geometrySmall1} holds. 
% and  \ref{geometrySmall2}. 
By Proposition \ref{localMinimumSmall} we have  a first critical 
point, which is a local minimum of $I_{\lambda}$. On the other hand, since the Cerami condition holds, in view of Lemmata \ref{geometrySmall1}.
 and  \ref{geometrySmall2}, we can apply  Theorem \ref{mpTheorem} and obtain a second critical point of 
$I_{\lambda}$ at the mountain-pass level. This gives  two different solutions of \eqref{Qlambda}. % in $\Wop$. 
Finally, by Lemma \ref{linkProblem}, we obtain two solutions of \eqref{Plambda}.%, as desired. 
\end{proof}

 \section{Proof of Theorems \ref{th2} and  \ref{th4}} 
\label{SER2}
In this section, we assume the stronger assumption \eqref{A2}. 
In that case, we are able to improve our results on the non-coercive case. 
%Let us recall the following assumptions
%\[ \tag{$A_{2}$} 
%\left \{
%\begin{aligned}
%& \Omega \subset \R^N,\, N \geq 2\,, \textup{ is a bounded domain with } \partial \Omega 
%\textup{ of class } \mathcal{C}^2\,.\\
%& c \textup{ and } h \textup{ belong to } L^{\infty}(\Omega)\,, \\
%& c \gneqq 0  \textup{ and } \mu > 0\,.
%\end{aligned}
%\right.
%\]

\begin{prop} \label{lmhNonPos}
Assume that 
%$(P_0)$ has a solution $u_0 \in \Wop \cap L^{\infty}(\Omega)$ and suppose that 
\eqref{A2} 
holds with $h \lneqq 0$. Then, for every $\lambda > 0$, 
there exists $v \in \mathcal{C}_0^{1,\tau}(\overline{\Omega})$, for some $0 < \tau < 1$, with $v \ll 0$, 
which is a local minimum of $I_{\lambda}$ in the $W_0^{1,p}$-topology and a solution of \eqref{Qlambda} 
with $v\geq \alpha_{\lambda}$ (with $\alpha_{\lambda}$ defined by \eqref{alpha}).
\end{prop}

\begin{proof}
First of all, observe that, as $h \lneqq 0$, we have $m_p>0$ and hence, by Theorem \ref{th3},
\eqref{P0} has a solution $u_0 \in \Wop \cap L^{\infty}(\Omega)$. 
By Lemma \ref{linkProblem}, % if $u_0$ is a solution of $(P_0)$, then
\[ v_0 = \frac{p-1}{\mu} \big( e^{\frac{\mu}{p-1}u_0} - 1 \big) \in \Wop \cap \Linfty, \]
%is a solution of $(Q_0)$  and  every solution  $v$  of $(Q_0)$ satisfies 
%$v \geq \alpha_0$. Hence, $v_0$ 
is then a weak solution of
\begin{equation*}
\left\{
\begin{aligned}
\pLaplac v_0 & = \big(1+\frac{\mu}{p-1}v_0 \big)^{p-1} h(x)\lneqq 0, & \qquad \textup{ in }  \Omega,\\
v_0 & = 0, & \qquad \textup{ on } \partial \Omega.
\end{aligned}
\right.
\end{equation*}
As moreover, $\big(1+\frac{\mu}{p-1}v_0 \big)^{p-1} h(x)\in  L^{\infty}(\Omega)$, it follows from 
\cite{DB_1983,  L_1988} that 
$v_0 \in \mathcal{C}_0^{1,\tau}(\overline{\Omega})$, for some $\tau \in (0,1)$ and, 
by  the strong maximum principle %of V\'azquez 
(see \cite{V_1984}), that $v_0 \ll 0$. Now, we split the rest of the proof in three steps.
%. Moreover, observe that 
%$h \lneqq 0$ implies that $v_0$ satisfies 
%\begin{equation*}
%\left\{ 
%\begin{aligned}
%\pLaplac v_0 & \leq 0, \qquad \textup{ in } \Omega,\\
%v_0 & = 0, \qquad \textup{ on } \partial \Omega.
%\end{aligned}
%\right.
%\end{equation*}
%By applying then the Strong maximum principle %of V\'azquez 
%(see \cite{V_1984}), it follows that 
%, since $h \lneqq 0$. 
%%%Now, by taking into account that $g \leq 0$ on $\R^{-}$ (see Lemma 
%%%\ref{gProperties}), we have that
%\[ \pLaplac v_0 = \big(1+\frac{\mu}{p-1}v_0 \big)^{p-1} h(x) \geq \lambda c(x) g(v_0) 
%+ \big(1+\frac{\mu}{p-1}v_0 \big)^{p-1} h(x),  \quad \textup{ in } \Omega,\]
%for every $\lambda > 0$. 
%That is, $v_0 \in \Wop \cap L^{\infty}(\Omega)$ 
%%%$v_0$ is an upper solution of \eqref{Qlambda} for every $\lambda > 0$. 
\medbreak

\noindent\textbf{Step 1:} \textit{$0$ is a strict upper solution of  \eqref{Qlambda}.}
\medbreak

Observe that $0$ is an upper solution of \eqref{Qlambda}. In order to prove that $0$ is strict,  let $v \leq 0$ 
be a solution of \eqref{Qlambda}. As $g \leq 0$ on $\R^{-}$ (see Lemma 
\ref{gProperties}), it follows that $v$ is a lower solution of $(Q_0)$ and so, thanks to the comparison principle, see Corollary \ref{compPrincipPlambda}, 
 $v \leq v_0 \ll 0$. Hence, $0$ is a strict upper solution of \eqref{Qlambda}. 
%\medbreak

%\textcolor{red}{ATTENTION : Le lecteur ne va pas identifier facilement de  quel `` our comparison principle" vous parler. Ici encore vous n'aidez pas assez le lecteur.  Je propose (en vous laissant mettre au point les d\'etails) de rajouter entre le Lemma 3.2 et la preuve du Theorem 1.2  un corollaire du Theorem 3.1 consistant \`a donner le r\'esultat du Theorem 3.1 directement sur le probl\`eme  $(P_{\lambda})$. Il faudra alors bien sur modifier le preuve du Th\'eor\`eme 3.1 (puisque une partie sera d\'ej\`a donn\'ee dans la preuve de ce Corollaire).  Ensuite et pour passer au probl\`eme $(Q_{\lambda})$ pourquoi ne pas rajouter dans l'\'enonc\'e du Lemma 5.2 quelque chose sur les liens entre sur et sous solutions des probl\`emes $(P_{\lambda})$ et $(Q_{\lambda})$.?  } 
\medbreak

\noindent\textbf{Step 2:} \textit{\eqref{Qlambda} has a strict lower solution $\underline\alpha\ll0$.}
\medbreak
%On the other hand, 
%by Proposition \ref{lowerSol}, for every $\lambda > 0$, there exists a lower solution 
By construction %$\alpha_{\lambda} \leq 0$  is  a lower solution of  \eqref{Qlambda} and hence, 
 $\underline{\alpha} = \alpha_{\lambda} - 1$ 
 %satisfies
%\[
%\begin{array}{cc}
% \pLaplac \underline{\alpha} = \pLaplac \alpha_{\lambda} \leq f_{\lambda}(x,\alpha_{\lambda})=
%f_{\lambda}(x,\underline{\alpha}), &\textup{ in } \Omega,
%\\
%\underline{\alpha} <0&\textup{ on } \partial\Omega.
%\end{array}
%\]
%i.e. $\underline{\alpha}$ 
is a lower solution of \eqref{Qlambda}. Moreover, as every solution $v$ 
of \eqref{Qlambda} satisfies $v \geq \alpha_{\lambda}\gg\underline{\alpha}$, we conclude that 
 $\underline{\alpha}$ is a strict lower solution of \eqref{Qlambda}. 
 \medbreak
 
\noindent\textbf{Step 3:} \textit{Conclusion.} 
\medbreak
 By Corollary \ref{localMinimizer}, 
 %we can ensure the existence of $v \in \Wop$ which is a local minimizer of
 %$I_{\lambda}$ in the $\mathcal{C}_0^1$-topology and a solution of \eqref{Qlambda} such that 
 %$\underline{\alpha} \ll v \ll 0$. Finally, by applying 
 Proposition \ref{c01vsWop}, 
 %we can conclude that 
 %$v \in \mathcal{C}_0^{1,\tau}(\overline{\Omega})$ for some $0 < \tau < 1$, 
 %it is a local minimum of $I_{\lambda}$ in the $W_0^{1,p}$-topology, 
 and Lemma  \ref{linkProblem}, 
 we have the existence of $v \in \Wop\cap  \mathcal{C}_0^{1,\tau}(\overline{\Omega})$, 
  local minimum of $I_{\lambda}$ and solution of \eqref{Qlambda} such that  $\alpha_{\lambda}\leq v \ll0$
 as desired. 
\end{proof}

\begin{proof}[\textbf{Proof of  the  first part of Theorem \ref{th2}}]
By  Proposition \ref{lmhNonPos}, there exists a first critical point, which is a local minimum of $I_{\lambda}$. 
By Theorem \ref{characterizacionMinimizer} and since the Cerami condition holds, we have two options. 
If we are in the first case, then together with Lemma \ref{geometrySmall2}, we see that $I_{\lambda}$ has 
the mountain-pass geometry and by Theorem \ref{mpTheorem}, we have the existence of a second solution. 
In the second case, we have directly the existence of a second solution of \eqref{Qlambda}. Then by  Lemma \ref{linkProblem} we conclude to the existence of two solutions to \eqref{Plambda}. 
\end{proof}

% \textcolor{red}{Attention, dans la version pr\'ec\'edente le Lemma \ref{geometrySmall2} ne permettait sans doute pas de conclure directement (difficile de comparer $I_{\lambda}(0)$ avec le niveau de $I$ autour de la solution minimum local).}

Now, we consider the case  $h \gneqq 0$. 
%In order to do that, let us introduce some notations. 
%We denote by $\gamma_1$ the first eigenvalue of the problem
%\begin{equation} \label{BVPeigenvalue}
%\pLaplac u = \gamma c(x) |u|^{p-2} u, \qquad u \in \Wop\,.
%\end{equation}
%We can state the following result.

\begin{lemma} \label{lemmaBehaviour}
Assume that \eqref{A2} holds and suppose that $h \gneqq 0$. Recall that $\gamma_1$
 denotes the first eigenvalue of \eqref{BVPeigenvalue}. It follows that:
\begin{enumerate}
\item[i)] For any  $0 \leq \lambda < \gamma_1$, any solution $u$ of the problem \eqref{Plambda} satisfies 
$u\gg0$.
\item[ii)] For $\lambda = \gamma_1$, the problem \eqref{Plambda} has no solution. 
\item[iii)] For $\lambda > \gamma_1$, the problem \eqref{Plambda} has no non-negative solution. 
\end{enumerate}
\end{lemma}

\begin{proof}
Observe first that, taking  $u^{-}$ as  test function in \eqref{Plambda}, we obtain 
\begin{equation}
\label{Lemeq1}
%\begin{aligned}
%\int_{\Omega} \Bigl( |\gradu|^{p-2} \gradu \nabla u^{-}  - \lambda c(x) |u|^{p-2} u u^{-} \Bigl)\, dx 
%& 
%=
 -  \int_{\Omega} \bigl( |\nabla u^{-}|^p - \lambda\, c(x) |u^{-}|^p \bigr)\, dx 
%\\
%&
 =  \int_{\Omega} \bigl(  \mu |\gradu |^p u^{-} + h(x) u^{-} \bigl)\, dx.
%\end{aligned}
\end{equation}

 \noindent\textit{i)} For $\lambda < \gamma_1$, there exists  $\epsilon>0$ such that, for every $u\in \Wop$, 
\[ \int_{\Omega}  \bigl( |\nabla u|^p - \lambda c(x) |u|^p \bigr) \, dx \geq \epsilon \|u\|^p\,.\]
Consequently, as $h \gneqq 0$ and $\mu > 0$, we have that
\[ 0 \geq -\epsilon \|u^-\|^p \geq - \int_{\Omega}  \bigl( |\nabla u^{-}|^p - \lambda c(x) |u^{-}|^p \bigr) \,dx 
=  \int_{\Omega} \bigl(  \mu |\gradu |^p u^{-} + h(x) u^{-} \bigl) \,dx \geq 0,\]
which implies that $u^{-} = 0$ and so that $u \geq 0$.
Hence $\pLaplac u\gneqq 0$ and by  the strong maximum principle 
(see \cite{V_1984}), we have $u \gg 0$.
\medbreak
 \noindent \textit{ii)}
In case $\lambda=\gamma_1$ we have, for every $u\in \Wop$,
\begin{equation}
\label{Lemeq2}
\int_{\Omega}  \bigl( |\nabla u|^p - \gamma_1 c(x) |u|^p \bigr) \, dx \geq0.
\end{equation}
Assume by contradiction that  \eqref{Plambda} has a solution $u$.
By  \eqref{Lemeq1} and \eqref{Lemeq2}, and using that $h \gneqq 0$ and $\mu > 0$, we have in particular
\[ \int_{\Omega}  \Bigl( |\nabla u^{-}|^p - \gamma_1 c(x) |u^{-}|^p \Bigr)\, dx=0.\]
This implies that $u^-=k\varphi_1$ for some $k\in \mathbb R$ and $\varphi_1$ the first eigenfunction of 
\eqref{BVPeigenvalue} and hence,  either $u\equiv 0$ or $u\ll0$. As $h\not\equiv0$, 
the first case cannot occur  as $0$ is not a solution of  \eqref{Plambda}. In the second case, as $h\gneqq 0$, 
we have
\[\int_{\Omega}  h(x) u^{-} \, dx>0\]
which contradicts \eqref{Lemeq1}, \eqref{Lemeq2} and $\mu>0$.
\medbreak
 
\noindent \textit{iii)}
Suppose by contradiction that $u$ is a non-negative solution of \eqref{Plambda}. 
As in the proof of i), we prove $u\gg0$ and hence, there exists $D_1>0$ such that $u\geq D_1 d$ with 
$d(x)=\dist(x, \partial\Omega)$.
Let $\varphi_1 > 0$ be the first eigenfunction of \eqref{BVPeigenvalue}. 
As $\varphi_1 \in {\mathcal C}^1(\overline\Omega)$, we have $D_2>0$ such that $\varphi_1\leq D_2 d$. 
This implies that 
$\frac{\varphi_1}{u} \in L^{\infty}(\Omega)$ and 
$\frac{\varphi_1^p}{u^{p-1}} \in \Wop$ with 
\[
\nabla\Big(\frac{\varphi_1^p}{u^{p-1}}\Big)
=p \Big(\frac{\varphi_1}{u}\Big)^{p-1} \nabla\varphi_1 - (p-1) \Big(\frac{\varphi_1}{u}\Big)^{p} \nabla u.
\] 
Hence we can take $\frac{\varphi_1^p}{u^{p-1}}$ as test function in \eqref{Plambda} and we have that
\[
%\begin{aligned}
\lambda \int_{\Omega} c(x) \varphi_1^p\, dx 
+ \int_{\Omega}\Big[ \mu  |\gradu|^p + h(x)\Big] \frac{\varphi_1^p}{u^{p-1}}\, dx 
%& 
%= \int_{\Omega} c(x) u^{p-1} \frac{\varphi_1^p}{u^{p-1}}\, dx 
%+ \int_{\Omega} [\mu  |\nabla u|^p + h(x) ] \frac{\varphi_1^p}{u^{p-1}}\, dx \\
%& 
= \int_{\Omega} \nabla \Big( \frac{\varphi_1^p}{u^{p-1}} \Big) |\nabla u|^{p-2} \nabla u\, dx\,.
%\end{aligned}
\]
On the other hand, applying Proposition \ref{picone}, we obtain
\[ \gamma_1 \int_{\Omega} c(x) \varphi_1^p \,dx = \int_{\Omega} |\nabla \varphi_1|^p\, dx 
\geq \int_{\Omega} \nabla \Big( \frac{\varphi_1^p}{u^{p-1}} \Big) |\nabla u|^{p-2} \nabla u\, dx.\]
Consequently, gathering together both inequalities, we have the contradiction
\begin{equation} \label{contradiction} 
0\geq (\gamma_1 - \lambda) \int_{\Omega} c(x) \varphi_1^p \, dx 
\geq \int_{\Omega} [ \mu |\gradu|^p + h(x)] \frac{\varphi_1^p}{u^{p-1}}\, dx>0\,.
\end{equation}
%Now, observe that, under the assumptions \eqref{A2}, $\lambda \geq \gamma_1$ implies that
%\[ (\gamma_1 - \lambda) \int_{\Omega} c(x) \varphi_1^p dx \leq 0\,,\]
%On the other hand, the same assumptions \eqref{A2} and $\varphi_1 > 0$, imply that 
%\[ \int_{\Omega} [ \mu |\gradu|^p + h(x)] \frac{\varphi_1^p}{u^{p-1}} dx > 0.\]
%Consequently, \eqref{contradiction} cannot happen and we have a contradiction. 
\end{proof}

\begin{cor} 
\label{corlambdaimplique 0}
Assume that \eqref{A2} holds. If,  for some 
$\lambda >0$,  \eqref{Plambda}  has a solution $u_{\lambda}\geq0$  then  \eqref{P0} has a solution.
\end{cor}

\begin{proof}
Observe that $u_{\lambda}$ is an upper solution of   \eqref{P0}.  
By Proposition \ref{lowerSol}, we know that  \eqref{P0} has a 
 lower solution $\alpha$ with $\alpha\leq u_{\lambda}$. The conclusion follows from Theorem \ref{BMP1988}.
\end{proof}

\begin{cor} 
\label{corlambdaimplique 0.2}
Assume that \eqref{A2} holds with $h \gneqq 0$. If   \eqref{Plambda}  has a solution for some 
$\lambda \in(0, \gamma_1)$, then  \eqref{P0} has a solution.
\end{cor}

\begin{proof}
If   \eqref{Plambda}  has a solution $u$, by Lemma \ref{lemmaBehaviour}, we have  $u\gg0$.
The result follows from Corollary \ref{corlambdaimplique 0}.
\end{proof}

\begin{prop} \label{lmhNonNeg}
Assume that \eqref{P0} has a solution $u_0 \in \Wop \cap L^{\infty}(\Omega)$ and suppose that \eqref{A2} holds 
with $h \gneqq 0$. Then there exists $\overline{\lambda}< \gamma_1$ such that:
\begin{enumerate}
\item[i)] For every $0 < \lambda < \overline{\lambda}$, there exists 
$v \in \mathcal{C}_0^{1,\tau}(\overline{\Omega})$, for some $0 < \tau < 1$, with $v \gg 0$, 
which is a local minimum of $I_{\lambda}$ in the $W_0^{1,p}$-topology and a solution of \eqref{Qlambda}.
\item[ii)] For  $\lambda = \overline{\lambda}$, there exists 
$u \in \mathcal{C}_0^{1,\tau}(\overline{\Omega})$, for some $0 < \tau < 1$, with $u \geq u_0$, 
which is a solution of \eqref{Plambda}.
\item[iii)] For $\lambda > \overline{\lambda}$, the problem \eqref{Plambda} has no non-negative solution. 
\end{enumerate}
\end{prop}

%\begin{proof}
%First of all, by arguing exactly as in the Lemma \ref{lmhNonPos}, we deduce that 
%\[ v_0 = \frac{p-1}{\mu} \left( e^{\frac{\mu}{p-1}u_0} - 1 \right) \in\mathcal{C}_0^1(\overline{\Omega}), \] 
%is a lower solution of \eqref{Plambda}. Now, let us define 
%\[ \overline{\lambda} = \sup\{ \lambda: \eqref{Plambda} \textup{ has a positive solution } u_{\lambda} \}\,.\]
%By using $ii)$ of the previous lemma, it follows that $\overline{\lambda} \leq \gamma_1< \infty$. 
%Now, for all $0 < \lambda < \overline{\lambda}$ let us show that \eqref{Plambda} has a solution. 
%By the definition of $\overline{\lambda}$ we can find $\lambda \in (\lambda, \overline{\lambda})$ and a 
%positive solution of $(P_{\delta})$.
%\end{proof}

\begin{proof}
Defining
 \[ \overline{\lambda} = \sup\{ \lambda: \eqref{Plambda} \textup{ has a non-negative solution } u_{\lambda} \},\]
we directly obtain that, for $\lambda > \overline{\lambda}$, the problem \eqref{Plambda} has no 
non-negative solution and, by  Lemma \ref{lemmaBehaviour} $ii)$, we see that 
$\overline{\lambda} \leq \gamma_1$. Moreover,  arguing exactly as in the first part of 
Proposition \ref{lmhNonPos}, we deduce that
\[ v_0 = \frac{p-1}{\mu} \big( e^{\frac{\mu}{p-1}u_{0}} - 1 \big) 
\in \Wop \cap \mathcal{C}_0^1(\overline{\Omega})\]
%is a non-negative lower solution of \eqref{Qlambda}, belongs to $\mathcal{C}_0^1(\overline{\Omega})$ 
%and 
satisfies $v_0\gg0$. 
%Furthermore, by using Lemma \ref{lemmaBehaviour} $ii)$, we see that $\overline{\lambda} \leq \gamma_1$. 
Now, fix $\lambda\in (0, \overline{\lambda})$.
\medbreak

\noindent\textbf{Step 1:} \textit{$0$ is a strict lower solution of  \eqref{Qlambda}.}
\medbreak

The proof of this step follows the corresponding one of Proposition \ref{lmhNonPos}.
 \medbreak
 \pagebreak

\noindent\textbf{Step 2:} \textit{\eqref{Qlambda} has a strict upper solution.}
\medbreak

By the definition of $\overline{\lambda}$ we can find 
$\delta \in (\lambda, \overline{\lambda})$ and a non-negative solution $u_{\delta}$ of $(P_{\delta})$. As above, 
we easily see that
%As $\pLaplac u_{\delta} \gneqq 0$, we easily see that
%$u_{\delta}\gg0$ and $u_{\delta}$  is a positive upper solution of \eqref{Plambda}, an so, it follows that
\[ v_{\delta} = \frac{p-1}{\mu} \big( e^{\frac{\mu}{p-1}u_{\delta}} - 1 \big) 
\in \Wop \cap \mathcal{C}_0^1(\overline{\Omega})\ \]
is a non-negative upper solution of \eqref{Qlambda} and $v_{\delta}\gg0$. 
%Furthermore, arguing again as in Proposition \ref{lmhNonPos}, 
%we deduce that $v_{\delta} \in \mathcal{C}_0^1(\overline{\Omega})$. Next, the same argument of  Proposition 
%\ref{lmhNonPos} implies that $v_0 \gg 0$. 
%
%
%Consequently, it follows that $0$ is a strict lower solution of 
%\eqref{Qlambda} as it is easy to verify the 
%$0$ is a lower solution of \eqref{Qlambda} and by the above property, 
%if $v$ is a solution  of \eqref{Qlambda} with 
%$v\geq 0$ we have $v\geq v_0\gg0$. 
Moreover, if $v$ is a solution of \eqref{Qlambda} with 
$v \leq v_{\delta}$, Theorem \ref{smpLP} implies that $v \ll v_{\delta}$. Hence, $v_{\delta}$ is a strict upper 
solution of \eqref{Plambda}. 
\medbreak
 
\noindent\textbf{Step 3:} \textit{Proof of i).} 
\medbreak

The conclusion follows as in Proposition \ref{lmhNonPos}.
\medbreak
 
\noindent\textbf{Step 4:} \textit{Existence of a solution for  $\lambda = \overline{\lambda}$.} 
\medbreak

Let $\{\lambda_n\}$ be a sequence with $\lambda_n< \overline{\lambda}$ and 
$\lambda_n\to \overline{\lambda}$  and $\{v_n\}$ be the corresponding sequence of minimum
of $I_{\lambda_n}$ obtained in i). This implies that $\langle I'_{\lambda_n}(v_n), \varphi\rangle =0$ 
for all $\varphi\in \Wop$. By the above construction, we also have
\[
I_{\lambda_n}(v_n)\leq I_{\lambda_n}(0)= -\frac{p-1}{p\mu} \int_{\Omega}h(x)\,dx.
\]

Arguing exactly as in Lemmata \ref{boundednessCerami} and \ref{strongConvergenceCerami}, we prove easily
the existence of $v\in \Wop$ such that  $v_n\to v$ in $\Wop$ with $v$ a solution of \eqref{Qlambda} for
$\lambda=\overline \lambda$. As $v_n\geq 0$ we obtain also $v\geq 0$, and, by Lemma \ref{linkProblem},
we have  the existence of a solution $u$ of \eqref{Plambda} with $u\geq 0$. 
As $u$ is then an upper solution of \eqref{P0}, we conclude that $u\geq u_0$.
\medbreak

\noindent\textbf{Step 5:} \textit{$\overline{\lambda}<\gamma_1$.} 
\medbreak

As by Lemma \ref{lemmaBehaviour}, the problem \eqref{Plambda} has no solution for $\lambda=\gamma_1$,
this follows from Step 4. 
%
%By Corollary \ref{localMinimizer}, we can ensure the existence of $v \in \Wop$ which is a local minimizer of
% $I_{\lambda}$ in the $\mathcal{C}_0^1$-topology and a solution of \eqref{Qlambda} such that 
% $0 \ll v \ll v_{\delta}$. Finally, by applying Proposition \ref{c01vsWop}, we can conclude that 
% $v \in \mathcal{C}_0^{1,\tau}(\overline{\Omega})$ for some $0 < \tau < 1$, 
% it is a local minimum of $I_{\lambda}$ in the $W_0^{1,p}$-topology, and, by Lemma  \ref{linkProblem}, 
% it is a solution of \eqref{Qlambda} with $v\geq \alpha_{\lambda}$ as desired. 
\end{proof}

\begin{proof}[\textbf{Proof of  the second part of Theorem \ref{th2}}]
By Lemma \ref{lemmaBehaviour}, we have $u_0\gg0$.
Let us consider $\overline{\lambda} \in (0,\gamma_1)$ 
given by  Proposition \ref{lmhNonNeg}. Hence, for $\lambda<\overline{\lambda}$, there exists a first critical point
$u_1$, 
which is a local minimum of $I_{\lambda}$. We then argue as in the proof of the first part to obtain
%Now, we can apply Theorem \ref{characterizacionMinimizer} and 
%we have two option. Either we obtain multiplicity of solution, or we obtain the suitable mountain-pass geometry 
%in order to apply Theorem \ref{mpTheorem}. If we have multiplicity of solution the proof is done. 
%If not, by arguing as in the previous theorem, we obtain the existence of a second critical point of $I_{\lambda}$ 
%and so, we obtain
 the second  solution  $u_2$ of \eqref{Plambda}. %, thanks to Lemma \ref{linkProblem}. 
By Lemma \ref{lemmaBehaviour}, these two solutions satisfy $u_i\gg0$ and, by Theorem \ref{compPrinciple}, 
we conclude that  $u_i\geq u_0$. Now, for $\lambda = \overline{\lambda}$, respectively $\lambda > \overline{\lambda}$, the result follows respectively from Proposition \ref{lmhNonNeg} \textit{ii)} and \textit{iii)}.
%Moreover, thanks to the 
%definition of $\overline{\lambda}$ it is obvious that \eqref{Plambda} has no positive solution for 
%$\lambda > \overline{\lambda}$.
\end{proof}

\begin{proof}[\textbf{Proof of Theorem \ref{th4}}]

\noindent\textbf{Part 1:} \textbf{Case $\boldsymbol{\lambda\in(0,\gamma_1)}$.}
\medbreak
\noindent\textbf{Step 1:} \textit{There exists $k>0$ such that \eqref{eq cor} has at least one solution.}
\medbreak
Let 
$\lambda_0\in(\lambda, \gamma_1)$ and $\delta$ small enough such that
$$
\lambda_0 \, s^{p-1}\geq  \lambda \Big(\frac{p-1}{\mu}\big( 1+\frac{\mu}{p-1}s \big)
 \ln \big(1+ \frac{\mu}{p-1}s \big)\Big)^{p-1}, \qquad \forall\ s\in [0,\delta].
$$
Define $w$ as a solution of 
\begin{equation}
\label{cll3g}
\pLaplac w = \lambda_0 \, c(x) |w|^{p-2}w  + h(x),\quad w \in W^{1,p}_0(\Omega).
\end{equation}
As $\lambda_0<\gamma_1$, we have $w\gg0$.

For $l$ small enough, $\tilde \beta= l w$ satisfies $0\leq \tilde \beta\leq\delta$ and, for $k$ such that 
$l^{p-1}\geq \big(1+ \frac{\mu}{p-1}\delta \big)^{p-1}k$, it is easy to prove that 
$\beta= \frac{p-1}{\mu} \ln \big(1+ \frac{\mu}{p-1}\tilde\beta \big)$ is an upper solution of  \eqref{eq cor} with 
$\beta\geq0$. As $0$ is a lower solution of \eqref{eq cor}, the claim follows from Theorem \ref{BMP1988}.
\medbreak
%\pagebreak

\noindent\textbf{Step 2:} \textit{For $k\geq k_0$, the problem \eqref{eq cor} has no solution.}
\medbreak

Let $u$ be a solution of \eqref{eq cor}. By Lemma \ref{lemmaBehaviour}, we have $u\gg0$. This implies that  
$u$ is an upper solution of $(P_{0,k})$. As $0$ is a lower solution of $(P_{0,k})$, 
by  Theorem \ref{BMP1988}, the problem  $(P_{0,k})$ has a solution and hence, by Proposition 
\ref{necessaryCondition}, $m_p>0$ which means that  $k< k_0$. This implies that, for $k\geq k_0$, 
the problem \eqref{eq cor} has no solution.

\medbreak

\pagebreak

\noindent\textbf{Step 3:} \textit{$\overline k = \sup\{k\in(0,k_0): 
\eqref{eq cor} \mbox{ has at least one  solution}\} <k_0$.}
\medbreak
Assume by contradiction that $\overline k=k_0$.
Let $\{k_n\}$ be an increasing sequence such that $k_n\to \overline k$, $k_n\geq \frac12\overline k$ 
and there exists $\{u_{n}\}$  a sequence of solutions of $(P_{\lambda,k_n})$. 
As in the previous step we have that 
 $u_{n}$  is an upper solution  of $(P_{0,\frac12\overline k})$. 
By Theorem \ref{compPrinciple}, we know that  $u_{n}\geq u_0$ with $u_0\gg0$ the solution of   
$(P_{0,\frac12\overline k})$.
Now, let $\phi \in \Wop\cap {\mathcal C}^1_0(\overline\Omega)$ with $\phi\gg0$ and 
\[
\Big( \frac{p-1}{\mu} \Big)^{p-1} \int_{\Omega} | \nabla \phi|^p \,dx
= k_0 \int_{\Omega} h(x) \phi^p \,dx.
\]
Using $\phi^p$ as test 
function and applying Young inequality as in the proof of Proposition \ref{necessaryCondition}, it follows that
\begin{equation*}
\begin{aligned}
\Big( \frac{p-1}{\mu} \Big)^{p-1} \int_{\Omega} | \nabla \phi|^p \,dx
&\geq 
\int_{\Omega} |\gradu_n|^{p-2}\gradu_n \nabla (|\phi|^p)\, dx 
- \mu \int_{\Omega} |\phi|^p |\gradu_n|^p\, dx 
\\
&= 
\lambda \int_{\Omega} c(x)|u_n|^{p-2}u_n \phi^p \,dx + k_n \int_{\Omega} h(x) \phi^p \,dx
\\
&\geq \lambda \int_{\Omega} c(x)|u_0|^{p-2}u_0 \phi^p \,dx +  k_n \int_{\Omega} h(x) \phi^p \,dx.
 \end{aligned}
\end{equation*}
Passing to the limit, we have the contradiction
\begin{equation*}
\Big( \frac{p-1}{\mu} \Big)^{p-1} \int_{\Omega} | \nabla \phi|^p \,dx
\geq 
\lambda \int_{\Omega} c(x)|u_0|^{p-2}u_0 \phi^p \,dx 
+  \Big( \frac{p-1}{\mu} \Big)^{p-1} \int_{\Omega} | \nabla \phi|^p \,dx.
\end{equation*}
\medbreak
\noindent\textbf{Step 4:} \textit{For $k> \overline k$, the problem \eqref{eq cor} has no 
solution and for $k< \overline k$, the problem \eqref{eq cor} 
has at least two solutions $u_1$, $u_2$ with $u_i\gg0$.}
\medbreak
The first statement is obvious by definition of $\overline k$. Now, for $k< \overline k$, let $\tilde k \in (k, \overline k)$ such that $(P_{\lambda, \tilde k})$ has a solution $\tilde u$. By Lemma \ref{lemmaBehaviour}, we have $\tilde u\gg0$. Then, it is easy to observe that
$\beta_1=\big(\frac{k}{\tilde k}\big)^{\frac{1}{p-1}}\tilde u$ and $\beta_2= \tilde u$
are both upper solutions of \eqref{eq cor} with $0\ll \beta_1\ll \beta_2$.

Observe that $0$ is a strict lower solution of \eqref{eq cor}. As $\beta_1\gg0$ is an upper solution 
of \eqref{eq cor}, by Theorem \ref{BMP1988}, the problem  \eqref{eq cor} has a minimum solution $u_1$ with 
$0\ll u_1\leq \beta_1$.

In order to prove the existence of the second solution, observe that if $\beta_2$ is not strict, it means that 
 \eqref{eq cor} has a solution $u_2$ with $u_2\leq \beta_2$ but $u_2\not\ll \beta_2$. 
 Then $u_2\not=u_1$ and we have our two solutions.
 If $\beta_2$ is strict, we argue as in the proof of Theorem \ref{th2}.
\bigbreak
\noindent\textbf{Step 5:} \textit{The function $\overline k(\lambda)$ is non-increasing.}
\medbreak
Let us consider $\lambda_1<\lambda_2$,  $\tilde k < \overline k(\lambda_2)$ and $\tilde u\gg0$ a solution of 
$(P_{\lambda_2, \tilde k})$. It is easy to prove that $\tilde u$ is an upper solution of 
$(P_{\lambda_1, \tilde k})$. As $0$ is a lower solution of $(P_{\lambda_1, \tilde k})$ with $0\leq \tilde u$, 
by Theorem  \ref{BMP1988}, the problem $(P_{\lambda_1, \tilde k})$ has a solution. 
This implies that  $\overline k(\lambda_1)\geq \overline k(\lambda_2)$.
\medbreak

\noindent\textbf{Part 2:} \textbf{Case $\boldsymbol{\lambda=\gamma_1}$.}
\medbreak
By Lemma  \ref{lemmaBehaviour}, we know that the problem $(P_{\gamma_1})$ 
has no solution for $k>0$. Moreover, by \eqref{Lemeq1}, we see that if $(P_{\gamma_1})$ with $h\equiv 0$ 
has a non-trivial solution, then $u\gneqq 0$ and hence, by the strong maximum principle $u\gg0$.
Arguing as in the proof of iii) of Lemma \ref{lemmaBehaviour}, 
we obtain the same contradiction \eqref{contradiction}.
\medbreak

\noindent\textbf{Part 3:} \textbf{Case $\boldsymbol{\lambda>\gamma_1}$.}
\medbreak

\noindent\textbf{Step 1:} \textit{There exists $k>0$ such that \eqref{eq cor} has at least one solution $u\ll0$.}
\medbreak
 
By Proposition \ref{anti} with $\bar{h} =h$, 
there exists $\delta_0>0$ such that, for $\lambda\in(\gamma_1,\gamma_1+\delta_0)$, the solution of 
\begin{equation}
\label{cll2}
\pLaplac w = \lambda \, c(x) |w|^{p-2}w  +  h(x)\,,\quad w \in W^{1,p}_0(\Omega)\,,
\end{equation}
satisfies $w\ll0$. Let us fix $\lambda_0\in (\gamma_1,\min(\gamma_1+\delta_0, \lambda))$ and
 $\delta$ small enough such that
$$
\lambda_0|s|^{p-2}s\geq  \lambda \Big|\frac{p-1}{\mu}\big( 1+\frac{\mu}{p-1}s \big)
 \ln \big(1+ \frac{\mu}{p-1}s \big)\Big|^{p-2}
\frac{p-1}{\mu}\big( 1+\frac{\mu}{p-1}s \big) \ln \big(1+ \frac{\mu}{p-1}s \big), \quad \forall\ s\in [-\delta,0]\,.
$$
Define $w$ as a solution of 
\begin{equation}
\label{cll3g}
\pLaplac w = \lambda_0 \, c(x) |w|^{p-2}w  +  h(x),\quad u \in W^{1,p}_0(\Omega).
\end{equation}
As $\gamma_1<\lambda_0<\gamma_1+\delta_0$, we have $w\ll0$.

For $l$ small enough, $\tilde \beta= l w$ satisfies $-\min(\delta, \frac{p-1}{\mu}) < \tilde \beta\leq 0$ and, 
for $k\leq l^{p-1}$, it is easy to prove that 
$\beta= \frac{p-1}{\mu} \ln \big(1+ \frac{\mu}{p-1}\tilde\beta \big)$ is an upper solution of  \eqref{eq cor} with 
$\beta\ll0$. By Proposition \ref{lowerSol},  \eqref{eq cor} has  a lower solution $\alpha$ with $\alpha\leq \beta$ 
and  the claim follows from Theorem \ref{BMP1988}.
\medbreak
\noindent\textbf{Step 2:} \textit{For $k$ large enough, the problem \eqref{eq cor} has no solution.}
\medbreak

Otherwise, let $u$ be a solution of \eqref{eq cor}. By Lemma 
\ref{lowerBound} and Remark \ref{remuniflb}, we have $M_{\lambda}>0$ such that, for all $k>0$, 
the corresponding solution $u$ satisfies $u\geq -M_{\lambda}$.
Let $\phi \in {\mathcal C}^1_0(\overline\Omega)$ with $\phi\gg0$. 
Using $\phi^p$ as test 
function,  by Young inequality as in the proof of Proposition \ref{necessaryCondition}, it follows  that
\begin{equation*}
\begin{aligned}
\Big( \frac{p-1}{\mu} \Big)^{p-1} \int_{\Omega} | \nabla \phi|^p \,dx
&\geq 
\int_{\Omega} |\gradu|^{p-2}\gradu \nabla (\phi^p)\, dx 
- \mu \int_{\Omega} \phi^p |\gradu|^p\, dx 
\\
&= 
\lambda \int_{\Omega} c(x)|u|^{p-2}u \phi^p \,dx + k\int_{\Omega} h(x) \phi^p \,dx
\\
&\geq - \lambda  M^{p-1}\int_{\Omega} c(x) \phi^p \,dx +  k \int_{\Omega} h(x) \phi^p \,dx.
 \end{aligned}
\end{equation*}
which is a contradiction for $k$ large enough.
\medbreak

\noindent\textbf{Step 3:} \textit{Define $\tilde k_1 = \sup\{k>0: 
\eqref{eq cor} \mbox{ has at least one  solution } u \ll0 \}$. For $k< \tilde k_1 $, the problem \eqref{eq cor} 
has at least two solutions with $u_1\ll0$ and $\min u_2<0$.}
\medbreak

For $k< \tilde k_1 $, let $\tilde k \in (k, \tilde k_1 )$ such that $(P_{\lambda, \tilde k})$ 
has a solution $\tilde u\ll0$. It is then easy to observe that $\beta_1= \tilde u$ and 
$\beta_2=\big(\frac{k}{\tilde k}\big)^{\frac{1}{p-1}}\tilde u$ 
are both upper solutions of \eqref{eq cor} with $ \beta_1\ll \beta_2\ll0$.

By Proposition \ref{lowerSol},  \eqref{eq cor} has  a lower solution $\alpha$ with $\alpha\leq \beta_1$ and hence, 
by Theorem \ref{BMP1988}, the problem  \eqref{eq cor} has a minimum solution $u_1$ with 
$\alpha \leq u_1\leq \beta_1$.

In order to prove the existence of the second solution, observe that if $\beta_2$ is not strict, it means that 
\eqref{eq cor} has a solution $u_2$ with $u_2\leq \beta_2$ but $u_2\not\ll \beta_2$. 
Then $u_2\not=u_1$ and we have our two solutions.
If $\beta_2$ is strict, we argue as in the proof of Theorem \ref{th2}.

\medbreak

\noindent\textbf{Step 4:} \textit{Define $\tilde k_2 = \sup\{k>0: 
\eqref{eq cor} \mbox{ has at least one  solution}\}$. For $k> \tilde k_2 $, the problem \eqref{eq cor} 
has no solution and, in case $\tilde k_1< \tilde k_2$, 
 for all $k\in(\tilde k_1, \tilde k_2)$, the problem \eqref{eq cor}
 has at least one solution $u$ with 
$u\not\ll 0$ and $\min u<0$.}
\medbreak

The first statement follows directly from the definition of $\tilde k_2$. In case $\tilde k_1< \tilde k_2$, 
for $k\in(\tilde k_1, \tilde k_2)$,  let $\tilde k \in (k, \tilde k_2)$ such that $(P_{\lambda, \tilde k})$ has a solution 
$\tilde u$. Observe that $\tilde u$ is an upper solution of  \eqref{eq cor}. Again, Proposition \ref{lowerSol} gives us 
a lower solution  $\alpha$  of  \eqref{eq cor}  with $\alpha\leq \tilde u$ and hence, 
by Theorem \ref{BMP1988}, the problem  \eqref{eq cor} has a solution $u$.
By definition of $\tilde k_1$, we have that $u\not\ll0$ and by Lemma \ref{lemmaBehaviour}, 
we know that $\min u<0$.
\medbreak

\noindent\textbf{Step 5:} \textit{The function $\tilde k_1(\lambda)$ is non-decreasing.}
\medbreak

Let us consider $\lambda_1<\lambda_2$,  $k < \tilde k_1(\lambda_1)$ and $u\ll0$ a solution of 
$(P_{\lambda_1,  k})$. It is easy to prove that $u$ is an upper solution of 
$(P_{\lambda_2, k})$. Again, applying Proposition \ref{lowerSol}  and Theorem  \ref{BMP1988}, we prove that 
the problem $(P_{\lambda_2, k})$ has a solution $u \ll 0$. 
This implies that  $\tilde k_1(\lambda_1)\leq \tilde k_1(\lambda_2)$.
\end{proof}

%{\color{red}Questions ouvertes: troisieme branche? Condition de (PSC) sous l'hypothese $m_p^c>0$? 
%Pas travailler sous l'hypothese $h\gneqq 0$ mais $u_0\geq 0$?
%Sous solution pour $c$ change de signe $p$-Laplacian?
%}

\appendix

\section{Sufficient condition} \label{AR}

\begin{lemma} 
Given $f \in L^r(\Omega)$, $r > \max\{ N/p,1\}$ if $p\not=N$ and  $1<r<\infty$ if $p=N$, let us consider
\[ E_f(u) = \Big( \int_{\Omega} \big(|\gradu|^p - f(x)|u|^p\big) \, dx \Big)^{\frac{1}{p}}\]
for an arbitrary $u \in \Wop$. It follows that:
\begin{enumerate}
\item[i)] If $1 < p < N$ and $\|f^{+}\|_{N/p} < S_N$, $E_f(u)$ is an equivalent norm in $\Wop$.
\item[ii)] If $p = N$ and $\|f^{+}\|_{r} < S_{N,r}$, 
$E_f(u)$ is an equivalent norm in $\Wop$.
\item[iii)] If $p > N$ and $\|f^{+}\|_1 < S_N$, $E_f(u)$ is an equivalent norm in $\Wop$.
\end{enumerate}
where, for $p \neq N$, $S_N$ denotes the optimal constant in the Sobolev inequality, i.e. 
\[ S_N = \inf \big\{ \|\gradu\|_p^p: u \in \Wop, \|u\|_{\past} = 1 \big\}\,,\]
and, for $p = N$,
\[ S_{N,r} = \inf \left\{ \|\gradu\|_p^p: u \in \Wop, \|u\|_{\frac{Nr}{r-1}} = 1 \right\}\,.\]
%where, we recall that:
%\begin{itemize}
%\item[i)] For $1 < p < N$, $p^{\ast} = \frac{Np}{N-p}$.
%\item[ii)] For $p \geq N$, $p^{\ast} = +\infty.$
%\end{itemize}
\end{lemma}

\begin{proof}
We give the proof for $1 < p < N$. The other cases can be done in the same way. 
First of all, by applying H\"{o}lder and Sobolev's inequalities, observe that, 
for any $h \in L^{\frac{N}{p}}(\Omega)$, it follows that
\[ \int_{\Omega} h(x) |u|^p dx \leq \|h\|_{\frac{N}{p}}\|u\|_{\past}^p 
\leq \frac{1}{S_N}\|h\|_{\frac{N}{p}}\|\gradu\|_p^p.\]
On the one hand, by using this inequality,  observe that
\begin{equation*}
%\begin{aligned}
\int_{\Omega} \Big(|\gradu|^p - f(x)|u|^p \Big)\, dx  
%\leq \int_{\Omega} \Big( |\gradu|^p + f^{-}(x)|u|^p \Big) \,dx 
\leq 
%\|u\|^p + \frac{1}{S_N}\|f^{-}\|_{\frac{N}{p}}\|u\|^p \\
%& = 
\|u\|^p \Big( 1 + \frac{\|f\|_{\frac{N}{p}}}{S_N} \Big). 
%= B \|u\|^p.
%\end{aligned}
\end{equation*}
On the other hand, following the same argument, we obtain that
\begin{equation*}
%\begin{aligned}
\int_{\Omega} \Big(|\gradu|^p - f(x)|u|^p \Big) \,dx %& 
\geq \int_{\Omega} \Big(|\gradu|^p - f^{+}(x)|u|^p\Big) \,dx 
\geq 
%\|u\|^p - \frac{1}{S_N}\|f^{+}\|_{\frac{N}{p}}\|u\|^p \\
%& = 
\|u\|^p \Big( 1 - \frac{\|f^{+}\|_{\frac{N}{p}}}{S_N} \Big) = A \|u\|^p
%\end{aligned}
\end{equation*}
with $A > 0$  since $\|f^{+}\|_{\frac{N}{p}} < S_N$. The result follows. 
%Consequently, by setting $\overline{A} = A^{1/p}$ and $\overline{B} = B^{1/p}$, for any $u \in \Wop$, 
%it follows that
%\[\overline{A} \|u\| \leq E_f(u) \leq \overline{B} \|u\|,\]
%and so, that $E_f(\cdot)$ defines a equivalent norm on $\Wop$.
\end{proof}

%Now, let us recall that we have defined
%\begin{equation*}
%m_p:= \inf \left\{ \int_{\Omega} |\nabla w|^p - \left(\frac{\mu}{p-1} \right)^{p-1} |w|^p h(x)\,dx: 
%w \in \Wop,\, \|w\| = 1 \right\}\,.
%\end{equation*}

As an immediate Corollary, we have a sufficient condition to ensure that $m_p > 0$.

\begin{cor}
Recall that $m_p$ is defined by \eqref{mp}.
Under the assumptions \eqref{A1},  it follows that:
\begin{enumerate}
\item[i)] If $1 < p < N$, then  $\|h^{+}\|_{N/p} < \big( \frac{p-1}{\mu} \big)^{p-1} S_N$ implies $m_p > 0$.
\item[ii)] If $p = N$, then $\|h^{+}\|_q < \big( \frac{p-1}{\mu} \big)^{p-1} S_{N,q}$ implies $m_p > 0$.
\item[iii)] If $p > N$, then $\|h^{+}\|_1 < \big( \frac{p-1}{\mu} \big)^{p-1} S_N$ implies $m_p > 0$.
\end{enumerate}
\end{cor}

\bibliographystyle{plain}
\bibliography{Bibliography}

\end{document}